\documentclass[smallcondensed]{svjour3}     
\smartqed  
\journalname{manuscript}

\usepackage{amssymb}
\usepackage[running,mathlines]{lineno} 
\usepackage{amsmath}

\usepackage{subfigure}

\AtBeginDocument{%
  \paperwidth=\dimexpr
    1in + \oddsidemargin
    + \textwidth
    + 1in + \oddsidemargin
  \relax
  \paperheight=\dimexpr
    1in + \topmargin
    + \headheight + \headsep
    + \textheight
    + 1in + \topmargin
  \relax
  
 \usepackage[pass]{geometry}\relax}

\usepackage{algpseudocode}
\usepackage{algorithm}
\usepackage{tikz}
\usepackage{tikz-qtree}
\usepackage{txfonts}
\usepackage{pgfplots}
\newcommand{\subfloat}{\subfigure}
\usepackage{mathrsfs}

\newcommand{\T}{\mathcal{T}}
\newcommand{\LP}{l\mathcal{P}}
\newcommand{\LF}{l\mathcal{F}}
\newcommand{\I}{\mathcal{I}}
\newcommand{\Oc}{\mathcal{O}}

\newtheorem{ass}{Assuption}
\newtheorem{fact}{Fact}

\begin{document}
\large

\title{Computing Overlaps of Two Quadrilateral Mesh Of the Same Connectivity \thanks{The project is supported by the National Natural Science Foundation of China (NSFC No. 11501044,11571047, 11571002).
}
}

\titlerunning{ Overlaps of Two Quadrilateral Mesh}

\author{
Xihua Xu \and Shengxin Zhu
}


\institute{
Shengxin Zhu \at
             Laboratory of Computational Physics \\ Institute of Applied Physics and Computational Mathematics. \\
             P.O.Box 8009, Beijing 10088, China.
              \email{zhu\_shengxin@iapcm.ac.cn}           
}

\date{ \today }

\maketitle

\begin{abstract}

An exact conservative remapping scheme requires overlaps between two tessellations and a reconstruction scheme on the old cells (Lagrangian mesh). While the are intensive discussion on reconstruction schemes, there are relative sparse discussion on how to calculate overlaps of two grids. Computing the exact overlaps was believed complicated and often be avoided. This paper introduces the mathematical formulation of such a problem and tools to solve the problem. We propose methods to calculate the overlaps of two dismissable general quadrilateral mesh of the same logically structure in a planar domain. The quadrilateral polygon intersection problem is reduced to as a problem that how an edge in a new mesh intersects with a local frame which consists at most 7 connected edges in the old mesh. As such, locality of the method is persevered. The alternative direction technique is applied to reduce the dimension of the searching space. It reduces more than 256 possible intersections between a new cell with the old tessellation to 34 (17 when considering symmetry) programmable intersections between an edge and an local frame whenever the intersection between the old and new cell does not degenerate. At the same time, we shall how the computational amount of the overlaps depends on the underlying problem in terms of singular intersection points. A simple and detailed classification on the type of overlaps is presented, according to classification, degeneracy of an overlap can be easily identified.

\end{abstract}

\section{Introduction}
Consider the following problem: suppose $\rho \in L_1(\Omega)$ is an unknown function, given observations of the form
$\int \varphi_i^a \rho \, d x = m_i^a$, where $\{ \varphi_i \}_{i=1}^N$ is a set of independent functions with the property of partition of unity(POU),
i.e., $\sum_{i=1}^N\varphi_i^a=1$, we want to estimate the value
$m_i^b=\int \varphi_j^b \rho \, dx$ with respect to a different set of independent functions $\{ \varphi_j^b\}_i^M$.
If we require $f$ to have a better regularity, then the classical pointwise interpolation problem is just a special case of such a formulation by setting $\varphi_i^a=\delta(x-x_i)$ and $\varphi_j^b=\delta(x-y_j)$ for two different point sets $\{ x_i\}_{i=1}^N$ and $\{ y_j\}_{j=1}^{M}$, where $\delta(\cdot)$ is the Dirac delta function.

In this paper, we focus on the case that $\varphi_i^a=\chi(T_i^a)$ is the characteristic function for $T_i^a$ and $T_i^a$ is an element of a tessellation/mesh $\T^a=\{T_1^a,\ldots,T_N^a \}$ of $\Omega \subset R^d$. Similarly, $\varphi_i^b=\chi(T_i^b)$ is the characteristic function on $T_i^b$, and $T_i^b $ is an element of another tessellation $ \T^b= \{ T_1^b, \ldots, T_M^b\}$ of $\Omega$. Such a generalised interpolation problem arises when it is required to transfer information from the old tessellation to a new one. For example, in the Arbitrary Lagrangian-Eulerian(ALE) method \cite{H74,L05}, or a moving mesh method, when the old tessellation becomes poor or severely distorted, one has to rezoning the tessellation and to map physical data on an old tessellation to a new one. The process of transfer the physical variables based on the observations $m_i^a=\int \varphi_i^a \rho \,dx$, $i=1,\ldots,N$, on an old mesh $\T^a$ to $m_i^b=\int \varphi_i^b \rho dx$, $i=1,\ldots,M$, on a new mesh $\T^b$, is referred to as \textit{remapping}.

As one of the three main steps in the ALE framework, remapping has kept researchers busy in the past decades. Currently there are two main streams of remapping schemes: one is the \textit{Cell-Intersection-Based Donor Cell(CIB/DC)}. Such an approach is conceptually simple; the information $\int \varphi_i^b \rho dx $ is aggregated from intersections between old cells and the new cells:
\begin{linenomath}
\begin{equation}
\int \varphi_i^b \rho dx = \sum_{j}\int_{\sup \varphi_i^b  \cap \sup \varphi_j^a} \varphi_j^a \rho dx.    \label{eq:ai}
\end{equation}
\end{linenomath}
Once this information is aggregated exactly, then the conservative property of the generalised interpolation follows due to the property of partition of unity:
\begin{linenomath}
$$
\sum_{i}^N \int_\Omega \varphi_i^a \rho dx = \int \sum_{i=1}^N \varphi_i^a \rho dx = \int_{\Omega} \rho dx =\int_{\Omega } \sum_{j=1}^M \varphi_j^b \rho dx = \sum_{j=1}^M \int_\Omega \varphi_j^b \rho dx.
$$
\end{linenomath}
Such an approach based on set or support intersection operation is also referred to as \textit{aggregated intersection Based Donor Cell(AIB/DC)} method.
The remapping error depends only on the approximation/reconstruction of the density function on the old cells whenever the integral over the overlays between the old and new cells is computed exactly. Precisely, piecewise constant reconstruct of the density function on the old cell grantees first order accuracy and convergence and piecewise linear construction grantees second order accuracy according to the standard approximation theory. Dukowicz in \cite{D84} considered applying the Gauss divergence theory to calculate the overlap(s) between an old cell and a new one. His procedure requires to transform a physical general quadrilateral meshes to the standard regular grid. Ramshaw in \cite{Ramshaw1985,Ramshaw1986} improves Dukowicz' result, the improved method works on the physical mesh directly without any coordinate transform. A first order three dimensional remapping scheme based on intersecting arbitrary polyhedra was consider by Grandy \cite{grandy1999} and Azarenok \cite{Azarenok2009}. These approaches are first order methods. Dukowicz pointed out that the first order CIB/DC based remmapping scheme is equivalent to upwind difference of the advective term  of a continuous equation \cite[p.412]{D84} \footnote{It originally reads "It can be easily shown that the continuous rezone with constant cell density corresponds to donor cell(or upwind) differencing of the advective terms" }. In \cite{dukowicz2000}, Dukowicz and Baumgardner articulated this equivalence. Second order CIB/DC methods were constructed by using piecewise linear reconstruction of the density function on the old cells \cite{dukowicz1987accurate}.   When the mass in a new cell are calculated by mass transfer from an old cell through the overlapped regions, the CIB/DC approach is also referred to as the \textit{flux-intersection-based} approach (generalised flux in fact). This approach surfers from programming difficulties as we shall demonstrate, but it works for all kinds of tessellations. In particular, it is still used when \textit{mixed cells}(cells contains several materials) appear \cite{B11}.

 The other method is the \textit{Faced-Based Donor-Cell(FB/DC) method} or the \textit{swept-region based methods}. In such an approach, the mass transfer between \textit{a home/departing cell} and a \textit{target/arriving cell} is calculated by surface integral(line integral in $R^2$) according to the Green formula or the divergence theory. The transfer is based on the \textit{fluxing area or swept region} between the arriving/target cell and neighbouring cells of the home cell. Dukowicz and Baumgardner shows how the mass exchanges between a home cell and a target cell \cite{dukowicz2000}. In such an approach, there is flexibility to approximate the the transferred mass, the mass in the swept region is not necessarily calculated exactly, simplification and neat approximation to the flux bring efficient and high order methods, see \cite{L11}\cite{margolin03} for details. And it is easier to generalize into three dimensional space compared with the CIB/DC approach. See \cite{Garimella2007} for example. Such an approach is also referred to the \textit{flux-based/face-based} method. It is valid when the Lagrangian and rezoned mesh have the same connectivity and are close to each other.

  The FB/DC approach is simpler for the CIB/DC methods due to the \textit{degeneracy of the intersections}. Increasing the robustness of a CIB/DC methods requires a special procedure to handle such degeneracies.  Some authors believe the CIB/DC approach is very complicated. This argument is basically true but is mathematical imprecise, because as we shall see similar degeneracy also arise for the FB/DC method. For both method, neglecting such degeneracy will result flawed procedure. The ability to handle such degeneracy decides the robustness of a remapping procedure.

  While the techniques vary between the two class methods, the essence of these mapping schemes is to compute or approximate the transfer of mass between a \textit{home cell or departing cell} in the old tessellation and a \textit{target cell or arriving cell} in the new tessellation. In the CIB/DC approach, this mass transfer is equivalent to computing the integration/quadrature on the intersections. Calculating the cell intersections and a proper quadrature rule (in terms of volumetric or surface integral) severs as the mathematical foundation of this approach. In the FB/DC method, calculating the \textit{fluxing/swept area} between the old faces and a good reconstruction on the faces and vertices is critical. Since the density function $\rho$ is usually approximated by piecewise functions on the old tessellation. Lower regularity appears on the faces of the old tessellation. Therefore, accurate computing the remapped mass or other physical variables requires summarize the underlying quantity piecewisely on the overlapped regions in the old tessellation to avoid singularity. That is why an exact remapping scheme requires to calculate the intersections or the swept/fluxing area between the old and new mesh. We shall mention that calculating or approximation the overlapped regions between the old and new tessellations and a reconstruction scheme (for the density function or flux) are two aspects of an remapping scheme. There are intensive discussion on reconstruction scheme, see \cite{Cheng2008,ms03,Z11} and reference therein, while there are relative sparse discussion on how to calculate the overlapped regions except a couple of noticeable examples \cite{Azarenok2009}, \cite{grandy1999} and \cite{dukowicz2000}.

 The aim of this paper is to articulate the mathematical formulation of these problems, present detailed classification of intersection types between a quadrilateral and another, and provide simple method to calculate the intersections between two quadrilateral mesh of the same connectivity based on the classification. The classification is simper and clearer than that in \cite[p.327, Table I]{dukowicz2000}. According the classification, one can quickly identify a degeneracy when it may arise. This helps to develop a robust software. When there is non degeneracy of the overlapped region, we demonstrate that the present approach only requires 34 programming cases(17 when considering symmetry), while the CIB/DC approach used as in \cite{Ramshaw1985,Ramshaw1986} requires 98 programming cases(Neither Dukowicz nor Ramshaw gave this quantitative result, the number is derived from Fig.~\ref{fig:cswap}, Fig.~\ref{fig:branch} in this paper  and Fig.1 in \cite{Ramshaw1986}).  When consider degeneracy, more benefit can be obtained. We shall demonstrates how the degeneracy of the intersection depends the so called \textit{singular intersection points}, and how the computational amount depends on the singular intersection points. As far as we known this is the first result on how the computational complexity depends on an underlying problem.

The paper is enlightened by the fact that calculating the intersection of a line and a polygon only takes $O(\log(p))$ complexity \cite[p.2, Fig.1]{CD87}, where $p$ is the number of vertices of a polygon. When considering how a continues piecewise linear curve intersects with a mesh, the computations can be further reduced. In either the CIB/DC approach or the FB/DC approach, we will reduce cell intersections or the swept region to a case that an edge/face intersects with a \textit{local frame}(see Definition \ref{def:lf}). By focusing on how an edge intersects with old cells and exploring the structure of the mesh,  the intersection area between the new and old tessellation, and the union of the fluxing/swept area can be traveled in $\mathcal{O}(n)$ time (Theorem \ref{thm:1}),where $n$ is the number of elements in the underlying mesh or tessellation. Our starting point is a direct cell-intersection formula based on purely set operation. It is conceptually simple and serves as the basis for the proposed algorithms. In parallel, we shall show the basic formula for the FB/DC method based on the \textit{fluxing/swept region}. It is derived from the Green formula. Here are some preliminary for presenting our results.

\

\section{Preliminary}

A \textit{tessellation} $\T=\{T_1,T_2,\ldots,T_N\}$ of a domain $\Omega$ is a partition of $\Omega$ such that  $\bar{\Omega} =\cup_{i=1}^N T_i$  and $\dot{T}_i \cap \dot{T}_j =\emptyset $, where $\dot{T}_i$ is the interior of the \textit{element} or \text{cell} $T_j$ and $\bar{\Omega}$ is the closure of $\Omega$.
A \textit{admissible} quadrilateral tessellation has no hanging node on any edge of the tessellation. Precisely, if $T_i\cap T_j \neq \emptyset$, then $T_i $ and $T_j$ can only share a common vertex or a common edge. We shall also refer to such a tessellation as a quadrilateral mesh.

A admissible quadrilateral tessellation can be indexed only by its vertices.  We use the conversional notation  as follows. The first three terms are identical to that used in \cite{margolin03}
 \begin{itemize}
 \item  $P_{i,j}$, for $i=1:M, j=1:N$ are the vertices;
 \item   $F_{i+ \frac{1}{2},j }$ for $i=1:M-1,j=1:N$  and $F_{i, j+\frac{1}{2}}$ for $i=1:M,j=1:N-1$ are the edges with vertices $P_{i,j}$ and $P_{i+1,j}$;
 \item  $C_{i+\frac{1}{2},j+\frac{1}{2}}$ stands for the quadrilateral cell $P_{i,j} P_{i,j+1} P_{i+1,j+1}P_{i,j+1}$ for $1\le i \le M-1$ and $1\le j \le N+1$. $C_{i+\frac{1}{2},j+\frac{1}{2}}$ is an element of $\T$, we shall also referred to it as $T_{i,j}$ with integer subindex for convenience;
\item $x_i(t)$ is the piecewise curve which consists all the face of $F_{i,j+\frac{1}{2}}$ for $j=1:N-1$, $y_j(t)$ is the piece wise curve which consists all the faces of $F_{i+\frac{1}{2},j}$ for $i=1:M-1$.
 \end{itemize}

Let $\T^a$ and $\T^b$ be two admissible quadrilateral mesh. The vertices of $\T^a$ and $\T^b$ are denoted as $P_{i,j}$ and $Q_{i,j}$ respectively, if there is a one-to-one map between $P_{i,j}$ and $Q_{ij}$, we shall say that the two tessellation share the same \textit{logical structure} or \textit{connectivity}.
Here we shall also assume the two admissible of the same logical structure of the same domain have the following property:
\begin{ass}[A1]
For the two admissible quadrilateral mesh $\T^a$ and $\T^b$, each vertex $Q_{ij}$ of $\T^b$ can only lie in the interior of the union of the cells  $C^a_{i\pm \frac{1}{2},j\pm \frac{1}{2} }$ of $\T^a$.
\label{ass:1}
\end{ass}

\begin{ass}[A2]
For the two admissible quadrilateral mesh $\T^a$ and $\T^b$, each curve $x_i^b(t)$ has at most one intersection point with $y_j^a(t)$, so does for $y_j^b(t)$ and $x^a_i(t)$.
\label{ass:2}
\end{ass}
\begin{ass}[A3]
The intersection point of an new face and an old face lies in the middle of the two faces.
\end{ass}

 For convenience, we also introduce the following notation and definition.  They will be frequently refereed to in the remaining of this paper.
\begin{definition}
A local patch $\LP_{i,j}$ of an admissible quadrilateral mesh consists an element $ T_{i,j}$ and its neighbours in $\T$. Take an interior element of $T_{i,j}$ as example,
\begin{linenomath}
 \begin{equation}
 \LP_{i,j}:=\{ T_{i,j}, T_{i\pm1,j}, T_{i,j\pm1}, T_{i\pm1,j\pm1} \}.
 \end{equation}
 \end{linenomath}
The index set of $\LP_{i,j}$ are denoted as
\begin{linenomath}
\begin{equation}
\mathcal{J}_{i,j}=\{ (k,s): T_{k,s} \in \LP_{i,j} \}.
\end{equation}
\end{linenomath}
\end{definition}
The local patch is similar to the notation $ \mathscr{C_i}= \cup_k C_k$ such that $\tilde{C}_i \in \mathscr{C}(C_i)$ in previous publication like \cite[eq.2.1]{ms03} and \cite[eq.1]{B11}, which is the smallest patch in the old tessellation which contains $C_i$, and it depends on the how a new cell intersect with the old cells. Here, the local patch is a fixed patch and is independent of how the cell intersects.  Under the Assumption~\ref{ass:1}, $\mathscr{C}(T_{i,j}) \subset \LP_{i,j}$. The introduce of a local patch for quadrilateral mesh is for programming simplicity.

Let $\T^a$ and $\T^b$ be two admissible quadrilateral mesh of the same domain with the Assumption~\ref{ass:1}.

\begin{definition}
An \textit{invading set} of the element $T_{i,j}^b$ with respect to the local patch $\LP^a_{i,j}$ is defined as
\begin{linenomath}
 \begin{equation}
 \mathcal{I}_{i,j}^b=(T^b_{i,j} \cap \LP^a_{i,j}) \backslash (T^b_{i,j}\cap T^a_{i,j}). \label{eq:I}
 \end{equation}
 \end{linenomath}
\end{definition}
\begin{definition}
 An \textit{occupied set} of the element $T^a_{i,j}$ with respect to the local patch $\LP^b_{i,j}$ is defined as
 \begin{linenomath}
 \begin{equation}
 \Oc_{i,j}^a=(T^a_{i,j} \cap \LP^b_{i,j}) \backslash (T^a_{i,j}\cap T^b_{i,j}). \label{eq:O}
 \end{equation}
 \end{linenomath}
\end{definition}
\begin{definition}
A \textit{swept area or fluxing area} is the area which is enclosed by a quadrilateral polygon with an edge in $\T^a$ and its counterpart edge in $\T^b$. For example, the swept area enclosed by the quadrilateral polygon with edges $F^a_{i+\frac{1}{2},j}$ and $F^b_{i+\frac{1}{2},j}$ will be denoted as $\partial F_{i+\frac{1}{2},j}$.  $\partial^{b+} F_{k,s}$ stands for boundary of the fluxing/swept area $\partial F_{k,s}$ is ordered such that direction of edges of the cell $T^b_{i,j}$ are counterclockwise in the cell $T^b_{i,j}$. Precisely
\begin{linenomath}
\begin{align*}
\partial^{b+}F_{i+\frac{1}{2},j}  &=Q_{i,j}Q_{i,j+1}P_{i,j+1}P_{i,j}, \\
\partial^{b+} F_{i+1, j+\frac{1}{2}}&=Q_{i+1,j}Q_{i+1,j+1}P_{i+1,j+1}P_{i+1,j}, \\
\partial^{b+}F_{i+\frac{1}{2},j+1} &=Q_{i+1,j+1}Q_{i,j+1}P_{i+1,j+1}P_{i,j+1},  \\
\partial^{b+} F_{i, j+\frac{1}{2}} &=Q_{i,j+1}Q_{i,j}P_{i,j}P_{i,j+1}.
\end{align*}
\end{linenomath}
where ${Q_{i,j}Q_{i+1,j}Q_{i+1,j+1}Q_{i,j+1}}$ are list in the counterclockwise order.
 \end{definition}
The invading set $\I^b_{i,j}$ has no interior intersection with the occupied set $\Oc^a_{i,j}$, while the fluxing areas associated to two connect edges of can be overlapped.
The invading and occupied sets consist of the whole intersection between an old cell and a new cell, while corners of a fluxing area may only be part of an intersection between an old cell and a new cell.
Fig.~\ref{fig:IO}, Fig.~\ref{fig:swepta} and Fig.~\ref{fig:swept2} illustrate such differences. In a local patch, the union of the occupied and invading set is a subset of the union of the swept/fluxing area of a home cell $T^a_{i,j}$. However, the difference (extra corner area), if any,  will be self-canceled when summing all the \textit{signed} fluxing area, see the north-west and south-east color region in Fig.~\ref{fig:swepta} and Fig.~\ref{fig:swept2}.

\begin{figure}[!t]
\centering
\subfloat[\label{fig:IO}]{\includegraphics[width=0.22\textwidth]{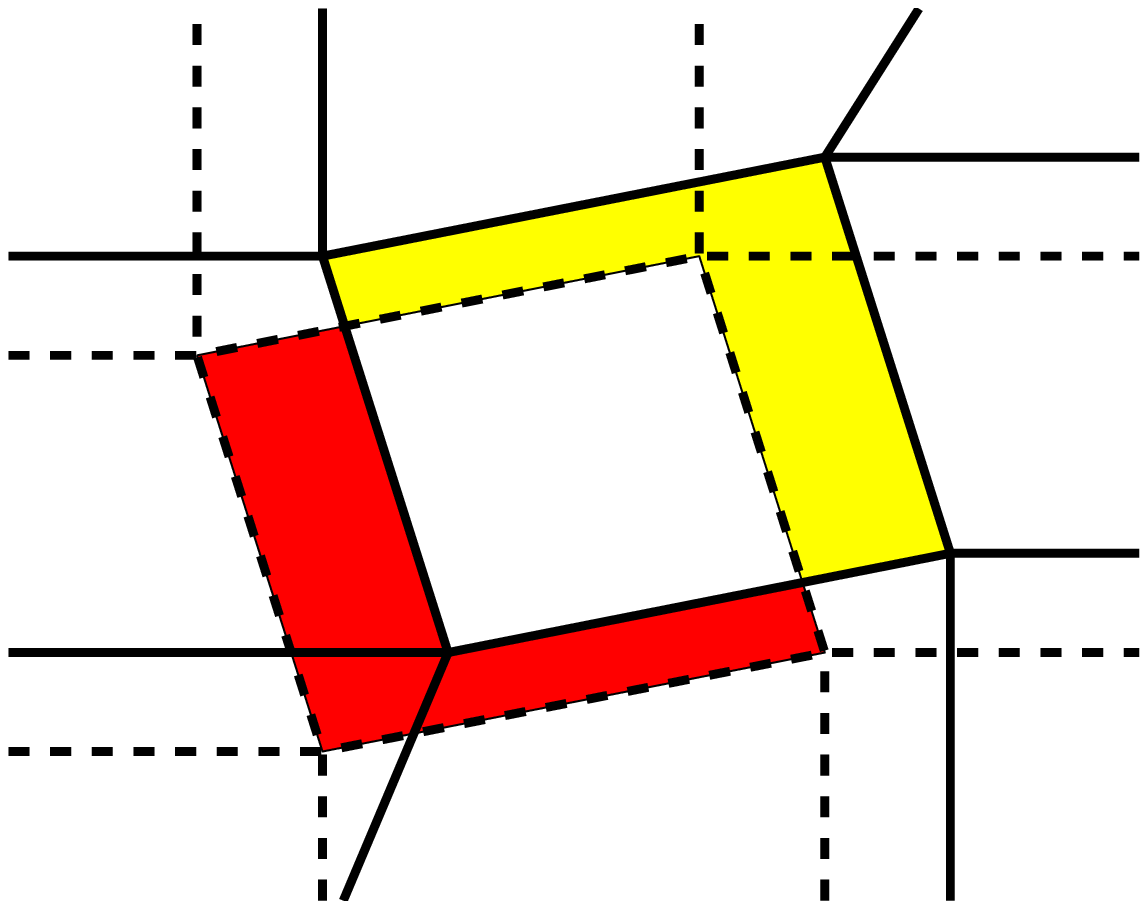}}
\subfloat[\label{fig:swepta}]{\includegraphics[width=0.22\textwidth]{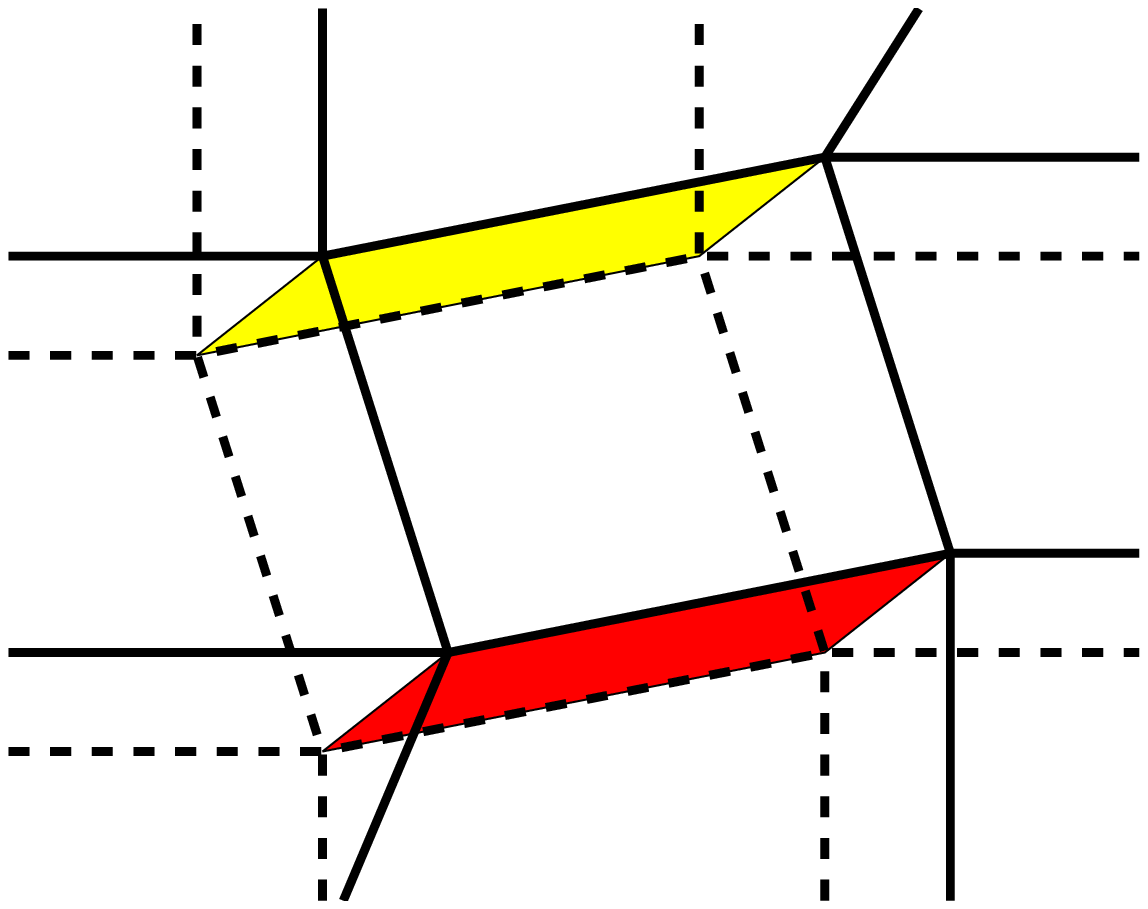}}
\subfloat[\label{fig:swept2}]{\includegraphics[width=0.22\textwidth]{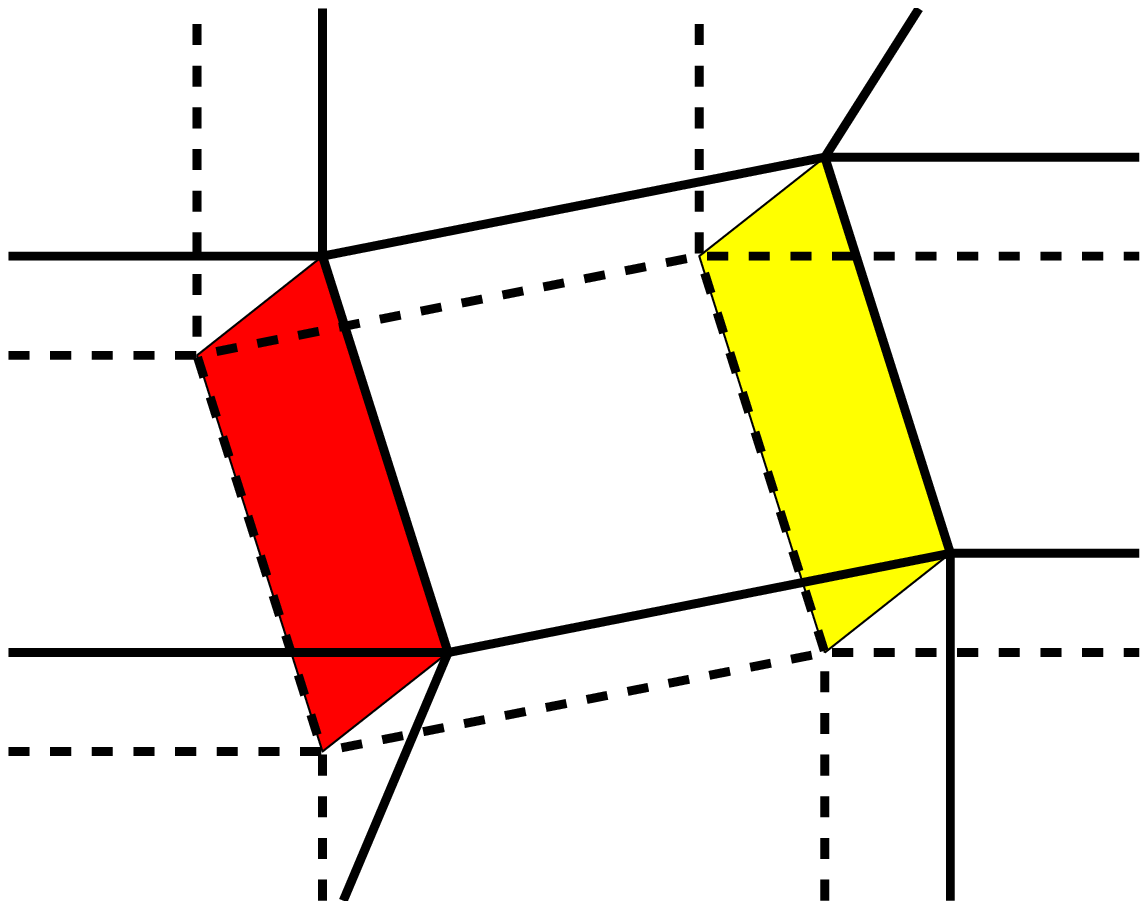}}
\subfloat[\label{fig:swap}]{\includegraphics[width=0.22\textwidth]{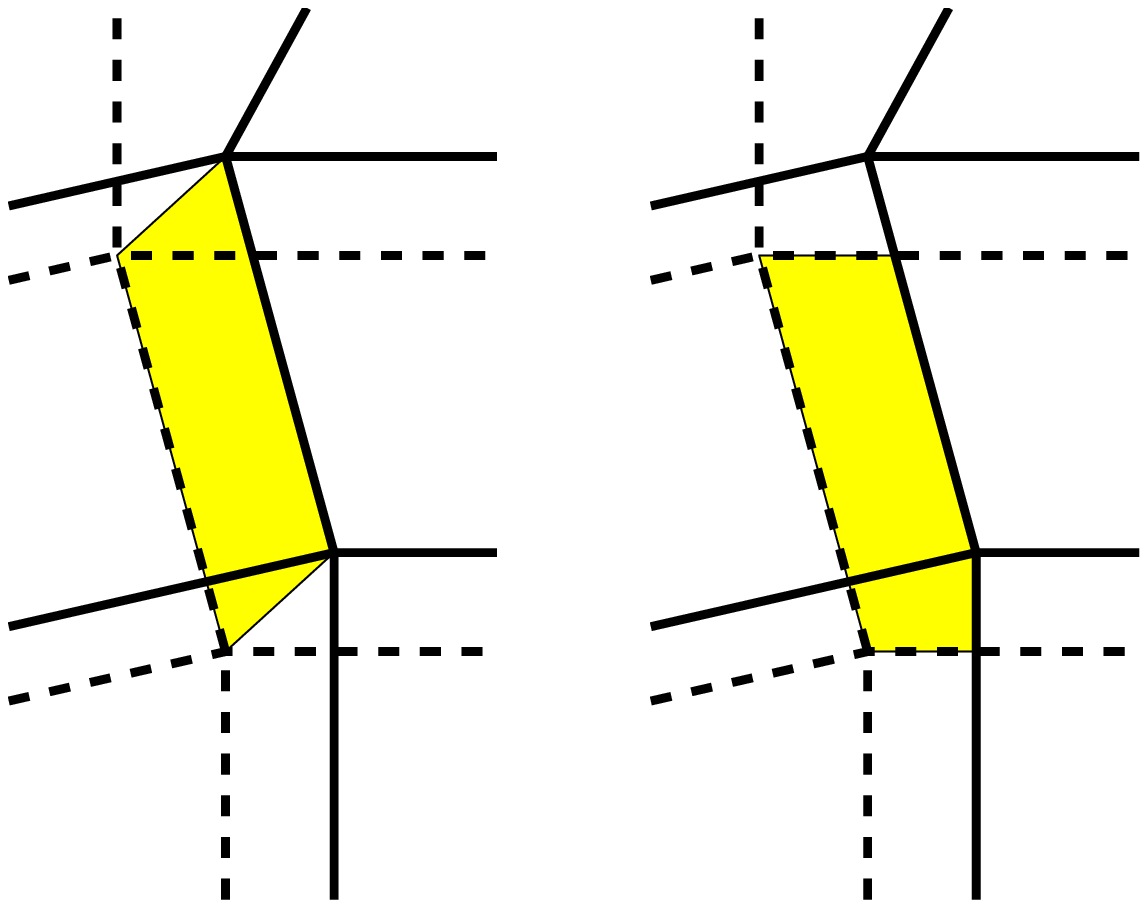}}
\caption{(a) the invading set $\I^b_{i,j}$ (red) and occupied set $\Oc^a_{i,j}$ (b) the swept/fluxing area $\partial F_{i\pm \frac{1}{2},j}$ (c) The swept/fluxing area $\partial F_{i,j\pm \frac{1}{2}}$ (d) the swept area (left) v.s. the local swap set(right) in a local frame.}
\label{fig:IOswept}
\end{figure}

\begin{definition}
A \textit{local frame} consist an edge and its neighbouring edges. Taking the edge $F_{i,j+\frac{1}{2}}$ as an example,
\begin{linenomath}
 \begin{equation}
 \LF_{i, j+\frac{1}{2}}=\{ F_{i,j\pm \frac{1}{2}}, F_{i,j+\frac{3}{2}}, F_{i\pm\frac{1}{2},j}, F_{i\pm \frac{1}{2},j+1} \}.
\end{equation}
\end{linenomath}
\label{def:lf}
\end{definition}
 Figure \ref{fig:lf} illustrates the local frame $\LF_{i,j+\frac{1}{2}}$.
\begin{figure}
\centering
\subfloat[\label{fig:swapping} swap regions]{\includegraphics[width=0.3\textwidth]{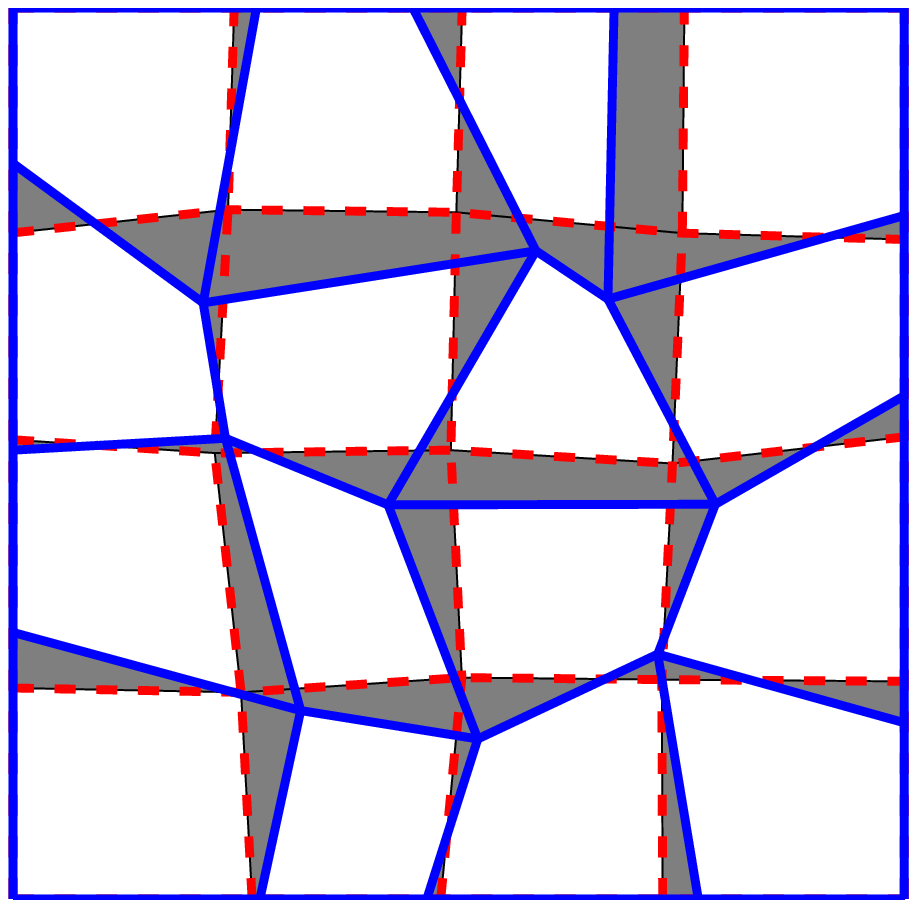}} \quad
\subfloat[$\LF_{i,j+\frac{1}{2}}$\label{fig:lf}]{\includegraphics[width=0.3\textwidth]{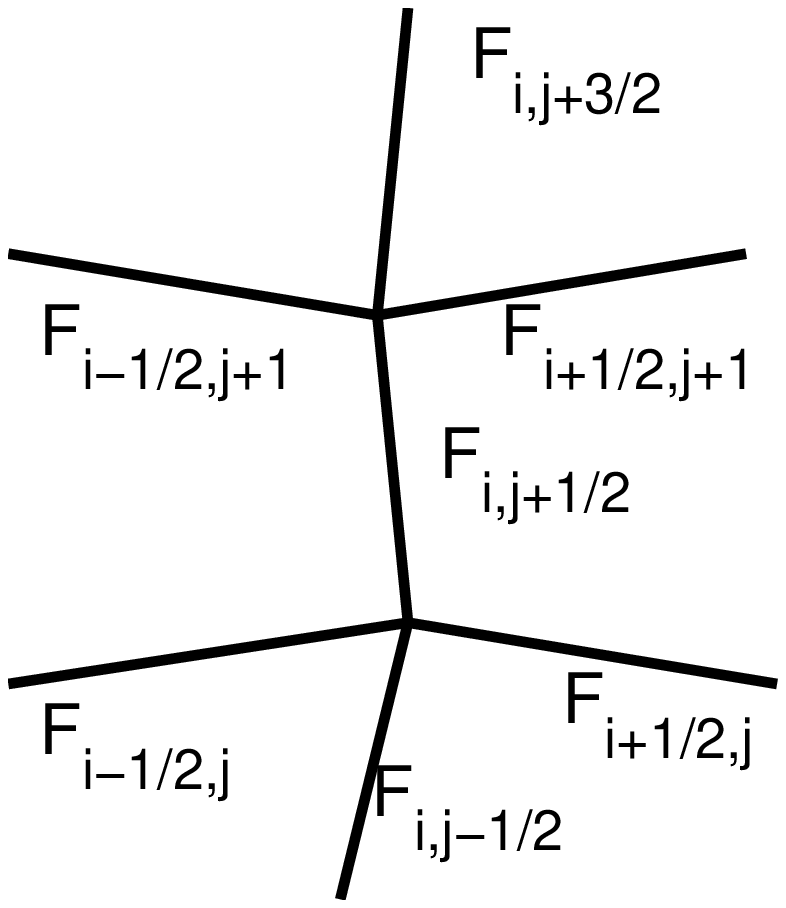}}
\subfloat[\label{fig:abcd}]{\includegraphics[width=0.3\textwidth]{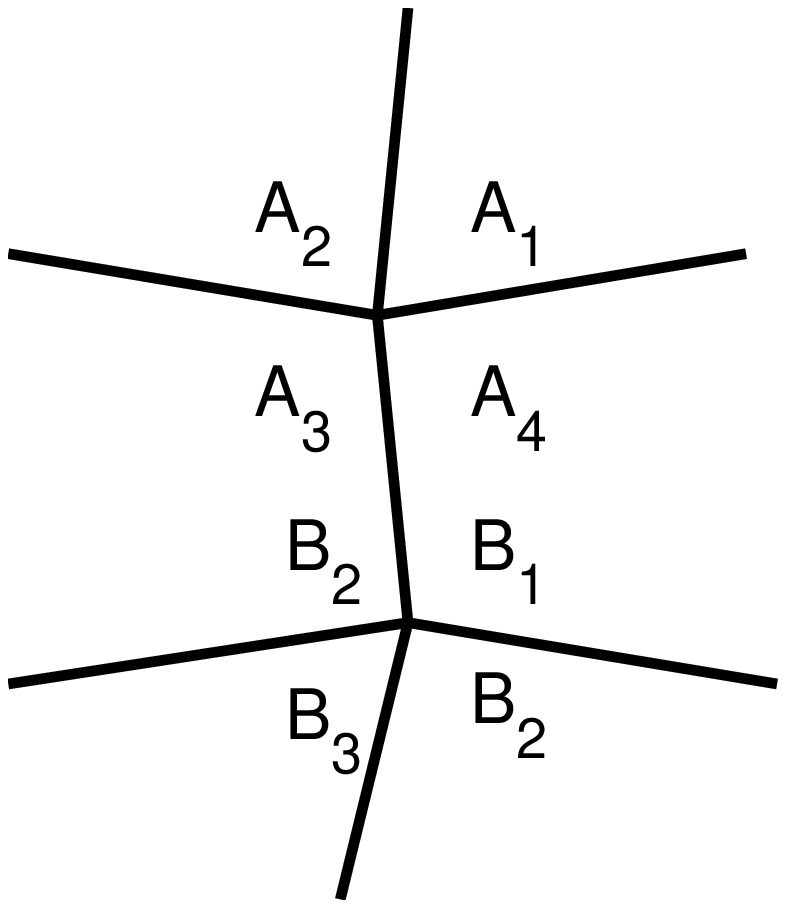}}
\caption{Illustration of the swept/swap region and the local frame}
\end{figure}

 \begin{definition}
 The region between the two curve $x^a_i(t)$ and $x^b_{i}(t)$ in $\Omega $ is referred to a vertical \textit{swap region}. The region between $y_j^a(t)$ and $y_j^b(t)$ in the domain $\Omega$ is referred to as a horizonal swap region. The region enclosed by $x^a_{i}(t), x_i^b(t) , y^b_j(t)$ and $y_{j+1}^b(t) $ is referred to as a \textit{local swap region} in the local frame $\LF_{i,j+\frac{1}{2}}$.
 \end{definition}
We introduce this two definition to tell the difference between the CIB/DC methods and the FB/DC method. In a CIB/DC method, a vertical swap region is divided into local swap region, while in the FB/DC method, a vertical swap region is divided by local swept regions. Such a difference results in different programming difficult.
\begin{definition}
The intersection points of $x^a_i(t)$ and $x^b_i(t)$ or $y^a_j(t)$ and $y^b_j(t)$ will be referred to as \textit{singular intersection points}. The total number of singular intersection points between $x^a_i(t)$ and $x^a_i(t)$ is denoted as $ns_{xx}$ for $ 1\leq i \leq M$, and the total singular intersection points between $y^a_j(t)$ and $y_j^b(t)$ will be denoted as $ns_{yy}$.
\end{definition}

\section{Facts and results}
 Let $\T^a$ and $\T^b$ be two admissible quadrilateral mesh with the Assumption A1, then the following facts hold  \begin{fact}
 The element $T^b_{i,j}$ of $\T^b$ locates in the interior of a local patch $\LP^a_{i,j}$ of $\T^a$. \end{fact}
\begin{fact}
 The face $F^b_{i,j+\frac{1}{2}}$ locates in the local frame $\LF^a_{i,j+\frac{1}{2}}$.
\end{fact}
\begin{fact}
An inner element of $\T^b$ has at least $4^4$ possible ways to intersect with a local patch in $\T^a$.
\end{fact}
\begin{fact}
If there is no singular point in the local swap region between $x_i^a(t)$ and $x_i^b(t)$ in $\Omega$ for some $i \in \{2, 3, \ldots,M-1 \}$, the local swap region consists of $2(N-1)-1$ polygons. Each singular intersection point in the swap region will  bring one more polygon.
\end{fact}

\subsection{Basic lemma}

The following results serve as the basis of the CIB/DC and FB/DC methods.
\begin{lemma}
Let $\T^a$ and $\T^b$ be two admissible mesh of the same structure. Under the Assumption A1, we have
\begin{itemize}
\item [(a)]
\begin{linenomath}
\begin{equation}
\mu(T^b_{i,j}) = \mu(T^a_{ij}) -\mu(\mathcal{O}^a_{i,j}) + \mu (\mathcal{I}^b_{i,j}), \label{eq:c}
\end{equation}
\end{linenomath}
where $\mu( \cdot )$ is the area of the underlying set, $\mathcal{I}^b_{i,j}$ is the invading set of $T^b_{i,j}$ and $\Oc^a_{i,j}$ is the occupied set of $T^a_{i,j}$ defined in \eqref{eq:I} and \eqref{eq:O} respectively.
\item [(b)]
\begin{linenomath}
\begin{equation}
\mu(T^b_{i,j}) =\mu(T^a_{i,j}) + \overrightarrow{\mu} (\partial F^{b+}_{i +\frac{1}{2},j}) + \overrightarrow{\mu} (\partial^{b+} F_{i+1,j+\frac{1}{2}})
+ \overrightarrow{\mu} (\partial F^{b+}_{i +\frac{1}{2},j+1}) + \overrightarrow{\mu} (\partial^{b+} F_{i,j+\frac{1}{2}})
,
\end{equation}
\end{linenomath}
where  $\overrightarrow{\mu}$ stands for the signed area calculated by directional line integrals.
\end{itemize}

\label{lem:b}
\end{lemma}
\begin{proof}
(a)
According to the identities
\begin{linenomath}
\begin{align*}
T^b_{i,j} &=\left (T^b_{i,j} \cap \LP^a_{i,j}) \backslash (T^b_{i,j} \cap T^a_{i,j}) \right) \bigcup (T^b_{i,j} \cap T^a_{i,j} )=\I^b_{i,j} \bigcup (T^a_{i,j} \cap T^b_{i,j}), \\
T^b_{i,j} &=\left (T^a_{i,j} \cap \LP^b_{i,j}) \backslash (T^a_{i,j} \cap T^b_{i,j}) \right) \bigcup (T^a_{i,j} \cap T^b_{i,j} )=\Oc^b_{i,j} \bigcup (T^a_{i,j} \cap T^b_{i,j}),
\end{align*}
\end{linenomath}
we have
\begin{linenomath}
\begin{align}
\mu(T^b_{i,j})= \mu (\I^b_{i,j}) +\mu(T^a_{i,j} \cap T^b_{i,j}),   \label{eq:c1}\\
\mu(T^a_{i,j})=\mu(\Oc^a_{i,j}) +\mu (T^a_{i,j} \cap T^b_{i,j}).  \label{eq:c2}
\end{align}
\end{linenomath}
Employ \eqref{eq:c1} and \eqref{eq:c2} and cancel $\mu (T^a_{i,j} \cap T^b_{i,j}) $, we obtain the result.

(b)The area of a quadrilateral polygon with vertices $v_1v_2v_3v_4$ counterclockwise arranged can be calculated  according to the Green formula by line integrals.
\begin{linenomath}
\begin{equation}
\iint_{\{v_1v_2v_3v_4 \}} dxdy = {\ointctrclockwise}_{\overrightarrow {v_1v_2}+\overrightarrow {v_2v_3}+ \overrightarrow {v_3v_4} +\overrightarrow {v_4v_1}} xdy.
\end{equation}
\end{linenomath}
Then one can verify that
\begin{linenomath}
\begin{align}
{\ointctrclockwise}_{\partial T^b_{i,j}} xdy &= {\ointctrclockwise}_{\partial T^a_{i,j}} xdy + \sum_{k,s}\oint_{\partial^{b+} F_{k,s}} xdy, \label{eq:cf}
\end{align}
\end{linenomath}
where $(k,s) \in \{ (i+\frac{1}{2},j), (i+1, j+\frac{1}{2}), (i+\frac{1}{2},j+1) , (i,j+\frac{1}{2}) \}$.
With the help of Fig.~\ref{fig:IOswept}, one can check that those line integrals which does not go along an edge of $T^b_{i,j}$ on the right hand side of \eqref{eq:cf} are canceled.
\qed
\end{proof}
\begin{remark}
If we write
\begin{linenomath}
$$
\I_{i,j}^b =\bigcup_{(k,s) \in \mathcal{J}_{i,j}^a} ( \I_{i,j}^b\cap T_{k,js}^a), \quad \Oc_{i,j}^a =\bigcup_{(k,s)\in \mathcal{J}_{i,j}^a} (\Oc_{i,j}^a \cap T_{k,s}^a).
$$
\end{linenomath}
and define
\begin{linenomath}
\begin{equation}
FA_{i,j\rightarrow k,s}= ( \I_{i,j}^b\cap T_{k,s}^a) - (\Oc_{i,j}^a \cap T_{k,s}^a) \textit{ for } (k,s) \in \mathcal{J}_{i,j}^a.
\end{equation}
\end{linenomath}
then
\begin{linenomath}
\begin{equation}
\mu(T_{ij}^b) =\mu(T_{i,j}^a)+\sum_{(k,s) \in \mathcal{J}_{i,j}^a} \mu(FA_{i,j\rightarrow k,s}). \label{eq:gfl}
\end{equation}
\end{linenomath}
$FA_{i,j\rightarrow k,s}$ corresponds to the so called  \textit{generalised (mass) flux } in \cite[eq. 19]{B11} and \cite[eq. 3.12]{margolin03}. The \textit{generalised flux} terms can have up to 9 terms in a local patch while the (physical) flux area can only have 4 terms.
\end{remark}

 Suppose a density function is a piecewise function on the tessellation $\T^a$ of $\Omega$. To avoid the interior singularity, we have to calculate the mass on $T^b_{i,j}$ piecewisely to avoid interior singularity as
 \begin{linenomath}
\begin{equation}
\int_{T_{i,j}^b} \rho d \Omega = \sum_{(k,s)\in \mathcal{J}^a_{i,j}} \int_{T^b_{i,j} \cap T^a_{k,s}} \rho d\Omega =  \int_{T^a_{i,j}} \rho d\Omega -  \int_{\I^b_{i,j}} \rho d\Omega -\int_{\Oc^a_{i,j}} \rho d\Omega.
\end{equation}
\end{linenomath}

The following result is a directly consequence  of Lemma \ref{lem:b}.
\begin{corollary}
Let $\T^a$ and $\T^b$ be two admissible quadrilateral mesh of $\Omega$, and $\rho$ is a piecewise function on $\T^a$, then the mass on each element of $\T^b$ satisfies
\begin{linenomath}
\begin{equation}
m(T_{i,j}^b) =m(T^a_{i,j}) -m(\Oc^a_{i,j})+m(\I^b_{i,j}). \label{eq:CIB}
\end{equation}
\end{linenomath}
and
\begin{linenomath}
\begin{equation}
m(T_{i,j}^b)=m(T^a_{i,j}) + \sum_{k,s} m(\partial^{b+} F_{k,s}). \label{eq:FB}
\end{equation}
\end{linenomath}
where $(k,s) \in \{ (i+\frac{1}{2},j), (i+1, j+\frac{1}{2}), (i+\frac{1}{2},j+1) , (i,j+\frac{1}{2}) \}$,  $ m(\partial^{b+} F_{k,s})$ is directional mass, the sign is  consistent with the directional area of $ \partial^{b+} F_{k,s}$.
\end{corollary}
The formulas \eqref{eq:CIB} and \eqref{eq:FB} are the essential formula for the CIB/DC method and FB/DC method respectively. The so called \textit{flux-intersection-based} approach based on the formula \eqref{eq:gfl} is equivalent to the CIB/DC approach.

It is easy to find that
\begin{linenomath}
$$
\bigcup_{i,j} \Oc^a_{i,j} =\bigcup_{i,j} \I^b_{i,j} = \bigcup_{i,j} (\partial F_{i + \frac{1}{2},j} \cup \partial F_{i,j+\frac{1}{2}})
$$
\end{linenomath}
This is the total swap region and the fluxing area. Since the swap region is nothing but the union of the the intersections between elements in $\T^a$ and $\T^b$, we shall introduce the following definition to characterise the ways of the intersections.
\begin{theorem}
Let $\T^a$ and $\T^b$ be two admissible quadrilateral mesh of a square in $R^2$ with the Assumption A1 and A2. $ns_{xx}$ and $ns_{yy}$ be the singular intersection numbers between the vertical and horizonal edges of the two mesh. If there is no common edge in the interior of $\Omega$ between $\T^a$ and $\T^b$. Then the swapping region of the two mesh consists of
\begin{linenomath}
\begin{equation}
3(N-1)(M-1)-2((M-1)+(N-1))+1 +ns_{xx} +ns_{yy}. \label{eq:comp}
\end{equation}
\end{linenomath}
polygons.

\label{thm:1}
\end{theorem}
\begin{proof}
We first consider the case where there is no singular intersection points.
There are $2(N-1)$ horizontal curves, say,  $$y_1^a(t)=y_1^b(t), y_2^a(t), y_2^b(t), \ldots, y_{N-1}^a(t),y_{N-1}^b(t), y_N^a(t)=y_N^b(t).$$
For each pair of $x_i^a(t),x_i^b(t)$, $i=2,\ldots,M-1$, there are $2(N-1)-1$ quadrilateral polygons between them. And there are $N-2$ polygons between
$\max_{x}\{ x_i^a(t),x_i^b(t) \}$, $\min_{x} \{ x_{i+1}^a(t), x_{i+1}^b(t) \}$ for $i=1, M-1$ and $y_j^a(t)$ and $y_j^b(t)$ for $j=2,\ldots,N-1$. Therefore there are in total
\begin{equation}
3(N-1)(M-1)-2((M-1)+(N-1))+1.
\end{equation}

The result is follows because each singular intersection points divides one not simple polygon into two simple polygons. \qed
\end{proof}

Notice that $(N-1)(M-1)$ is the number of the elements of the tessellation of $\T^a$ and $\T^b$. Then \eqref{eq:comp} implies for the CIB/DC method, the swap region can be computed in $O(n)$ time when every the overall singular intersection points is bounded in $\mathcal{O}(n)$, where $n$ is the number of the cells. The complexity depends on the singular intersection points. Such singular intersection points depends on the underlying problem, for example, a rotating flow can bring such singular intersections.

 \subsection{Intersection between a face  and a local frame}

\begin{table}
\centering
\caption{Cases of intersections of a vertical edge with a local frame}
\begin{tabular}{cccc}
\hline
  & \multicolumn{3}{c}{ Intersection \# with horizonal/vertical frames (H\#/V\#)} \\ \cline{2-4}
  & H0  & H1 & H2    \\\hline
 0 & shrunk & shifted & stretched  \\
 V1& diagonally shrunk & diagonally shifted  & diagonally stretched \\
 V2& --& shifted & stretched \\
 V3& --& --   &  diagonally stretched \\
 \hline
\end{tabular}
\label{tab:case}
\end{table}

 For the two admissible quadrilateral mesh $\T^a$ and $\T^b$ of the same connectivity, we classify the intersections between a face $F^b_{i,j+\frac{1}{2}}$ and the local frame $\LF^a_{i,j+\frac{1}{2}}$ into six groups according to the relative position of the vertices $Q_{i,j}$ and $Q_{i,j+1}$  in the local frame $\LF^a_{i,j+\frac{1}{2}}$. The point $Q_{i,j+1}$ can locate in $A_1$,$A_2$, $A_3$ and $A_4$ in the local frame of $\LF_{i,j+\frac{1}{2}}$ in Figure~\ref{fig:abcd}, and the point $Q_{i,j}$ can locate in $B_1$,$B_2$, $B_3$ and $B_4$ region.  Compared with the face $F^a_{i,j+\frac{1}{2}}$, the face $F^b_{i,j+\frac{1}{2}}$ can be
 \begin{itemize}
 \item shrunk: $A_3B_2$ and $A_4B_1$;
 \item shifted: $A_1B_1$, $A_2B_2$, $A_3B_3$ and $A_4,B_4$;
 \item stretched: $A_1B_2$ and $A_2B_3$;
 \item diagonally shrunk: $A_3B_1$ and $A_4B_3$;
 \item diagonally shifted: $A_1B_2$, $A_2B_1$, $A_3B_2$, and $A_4B_3$;
 \item diagonally stretched: $A_1B_3$ and $A_2B_4$.
 \end{itemize}

 And then for each group, we choose one representor/generator to check the intersection numbers between the face $F^a_{i,j+\frac{1}{2}}$ and the local frame $\LF^a_{i,j+\frac{1}{2}}$. Finally,  according to intersection numbers of the $F^b_{i,j+\frac{1}{2}}$ with the horizonal edges and vertical edges in the local frame $\LF^a_{i,j+\frac{1}{2}}$, we classify the intersection cases into six groups in Table \ref{tab:case} and 17 symmetric cases in Fig.~\ref{fig:case}.
 \begin{fact}
 Let $\T^a$ and $\T^b$ be two admissible quadrilateral mesh of the same structure. Under the Assumption A1, A2 and A3, an inner edge $F^b_{i,j+\frac{1}{2}}$ of $\T^b$ has 17 symmetric ways to intersect with the local frame $\LF^a_{i,j+\frac{1}{2}}$. A swept/fluxing area has 17 possible symmetric cases with respect to the local frame.
 \end{fact}

 \begin{figure}[!t]
 \centering
 \subfloat[ H0V0\label{fig:h0v0}]{\includegraphics[height=1.6cm,width=1.2cm]{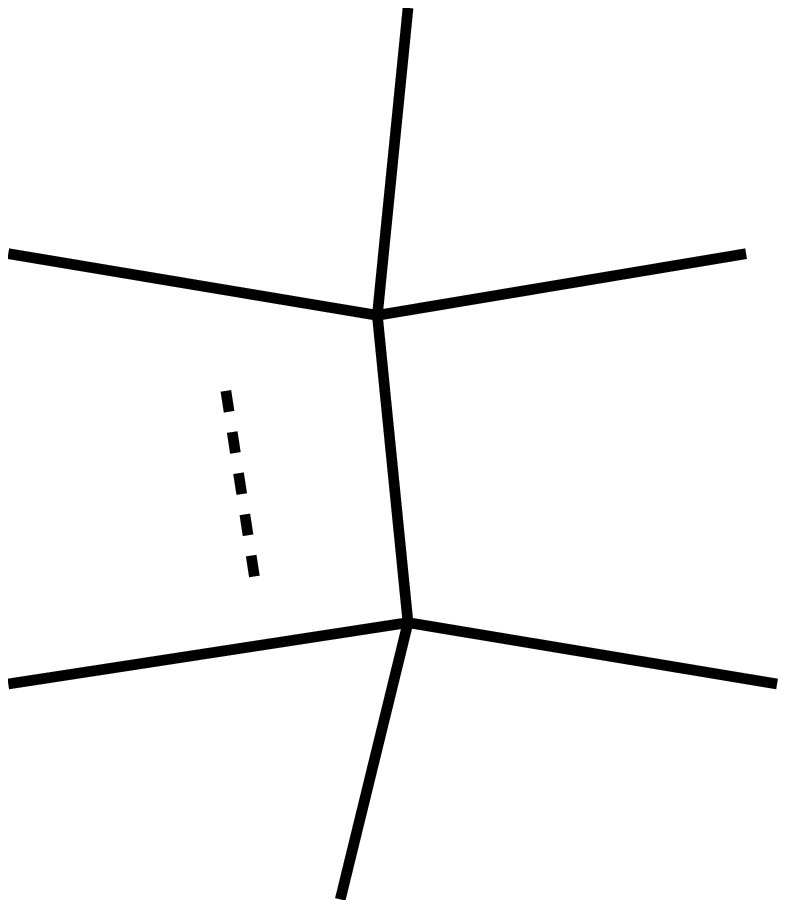}
                \includegraphics[height=1.6cm,width=1.2cm]{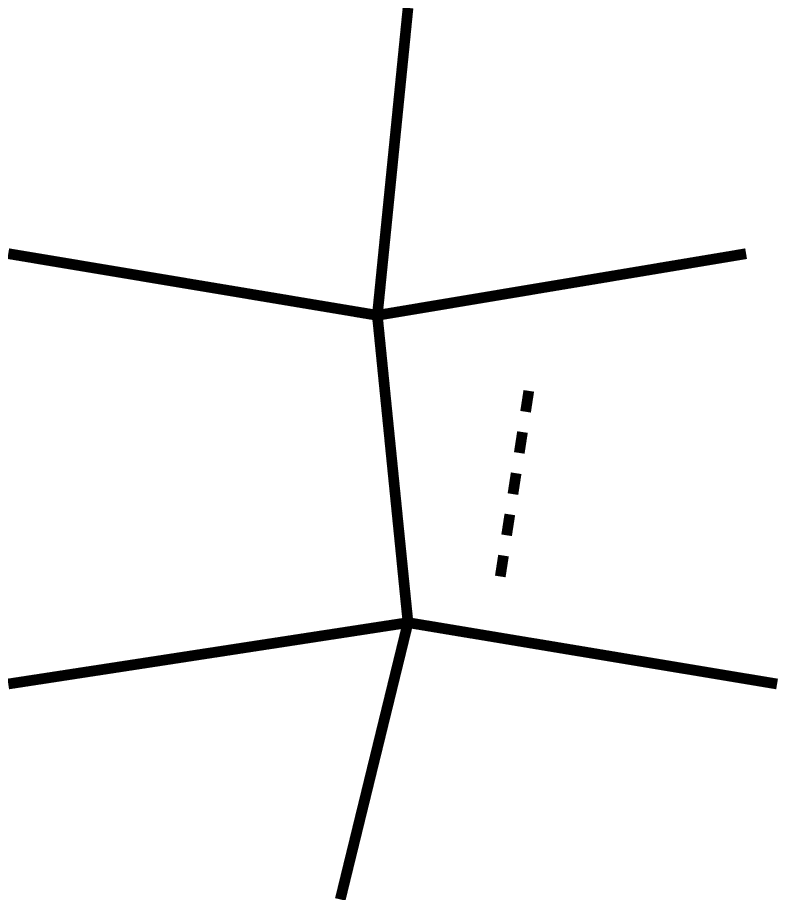}} \,
\subfloat[ H0V1\label{fig:h0v1}]{\includegraphics[height=1.6cm,width=1.2cm]{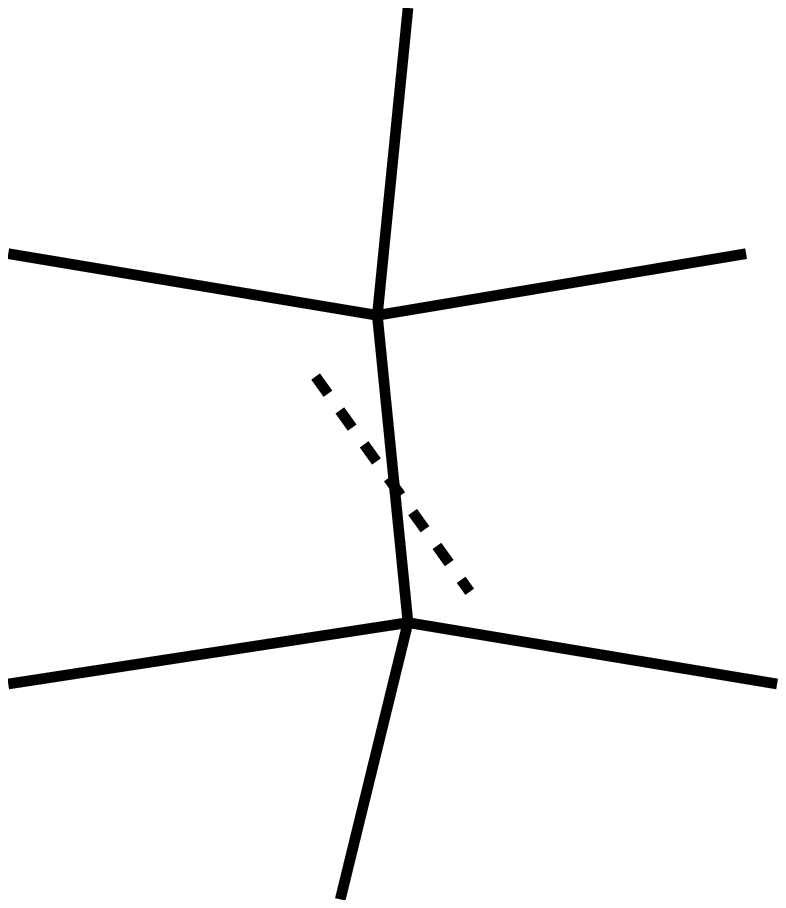}
                \includegraphics[height=1.6cm,width=1.2cm]{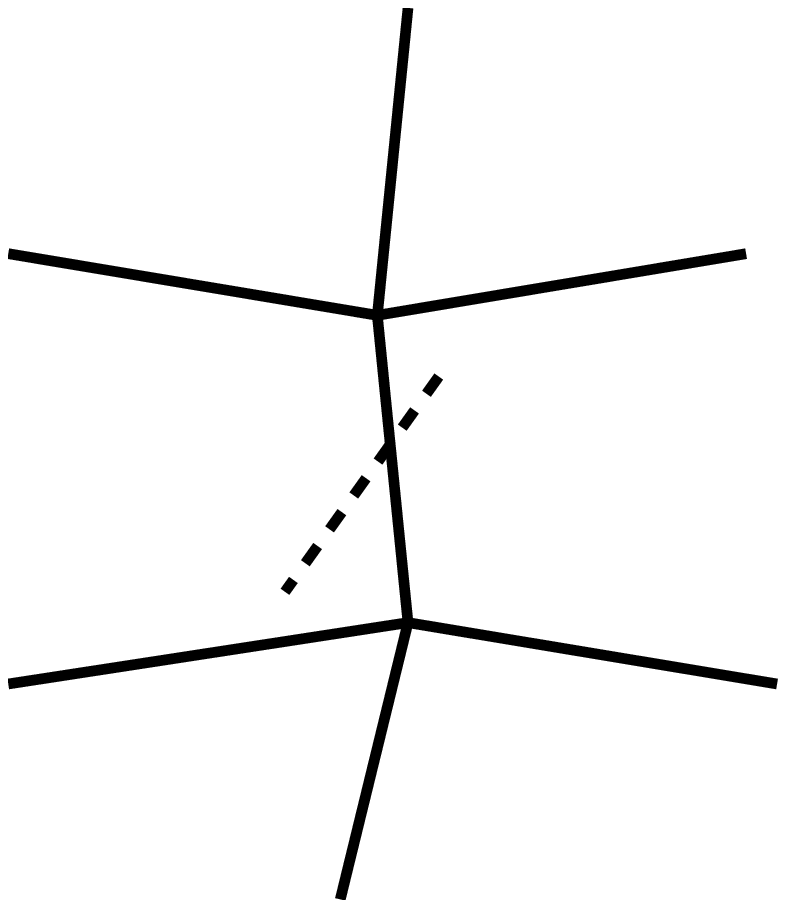}}
\center{ shrunk and diagonally shrunk cases} \\
 \centering
 \subfloat[H1V0\label{fig:h1v0}]{\includegraphics[height=1.6cm,width=1.2cm]{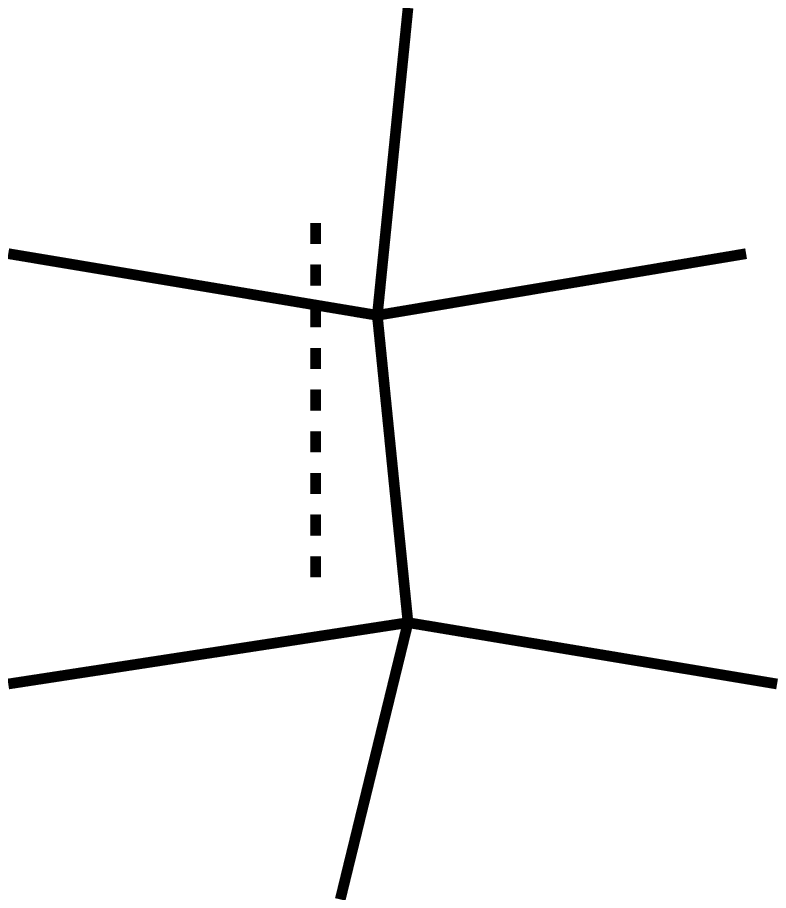}
                \includegraphics[height=1.6cm,width=1.2cm]{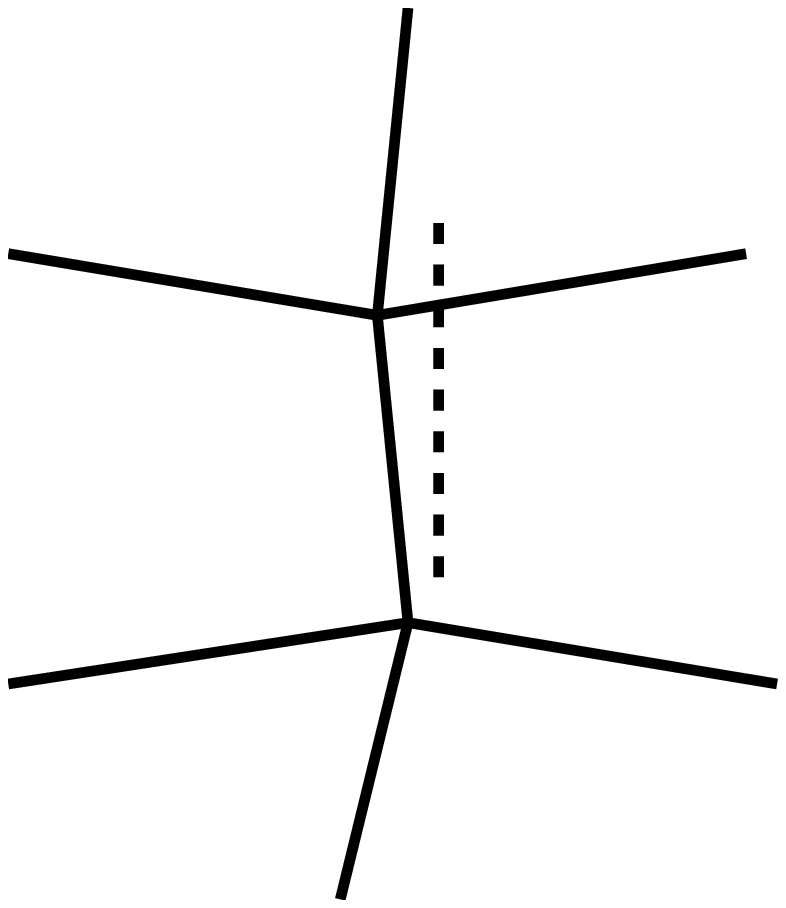}} \,
\subfloat[H1V0]{\includegraphics[height=1.6cm,width=1.2cm]{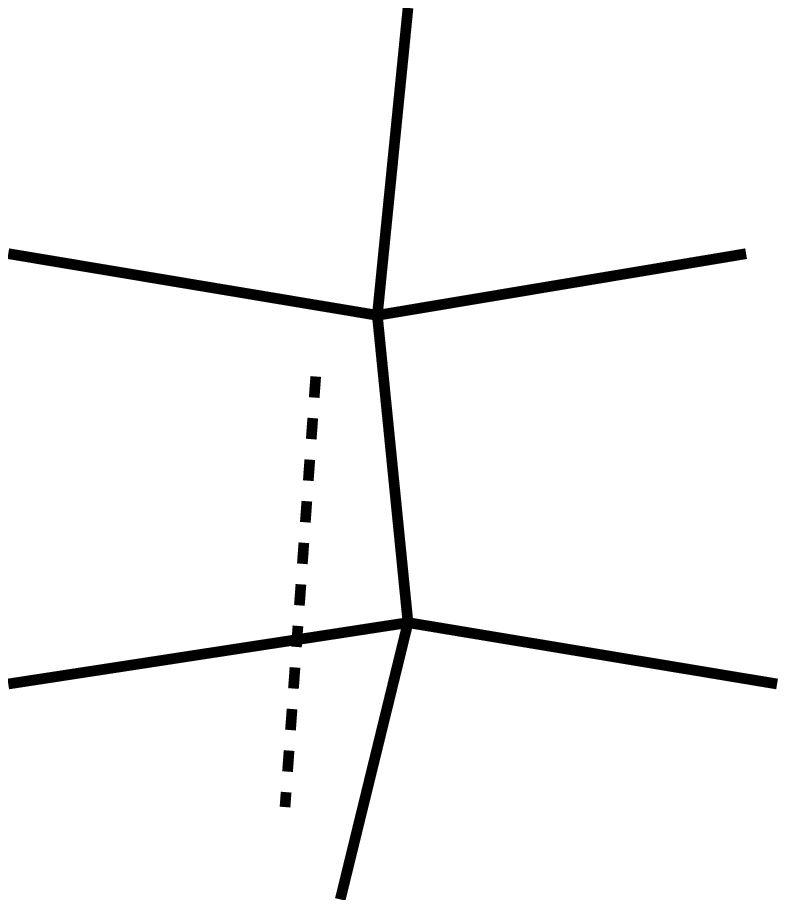}
                \includegraphics[height=1.6cm,width=1.2cm]{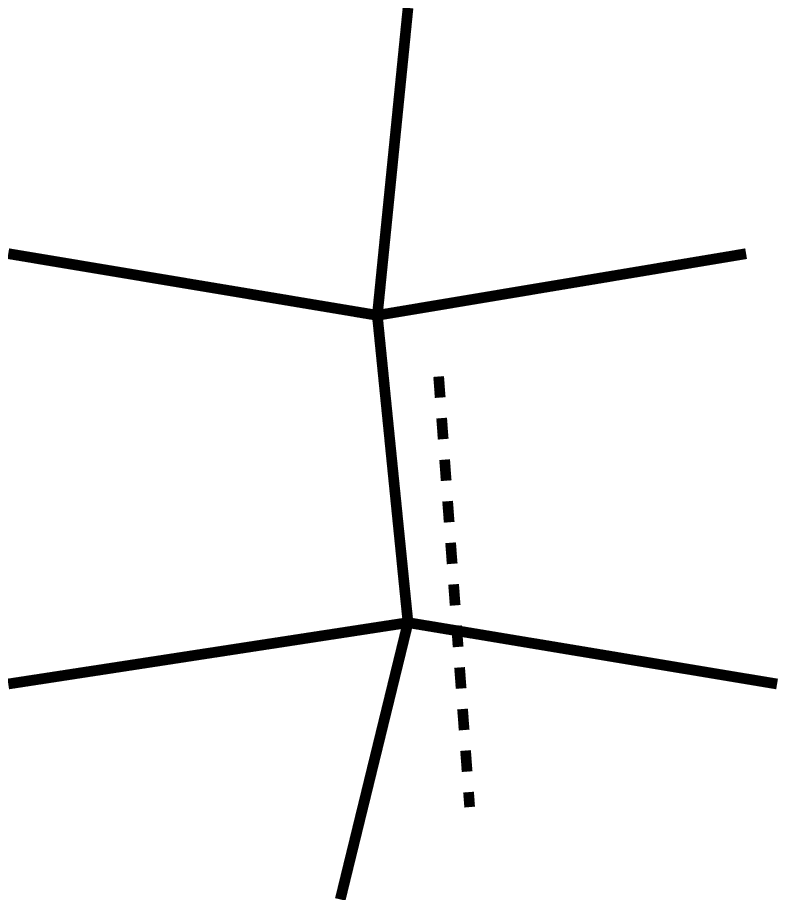}}
\subfloat[H1V2A]{\includegraphics[height=1.6cm,width=1.2cm]{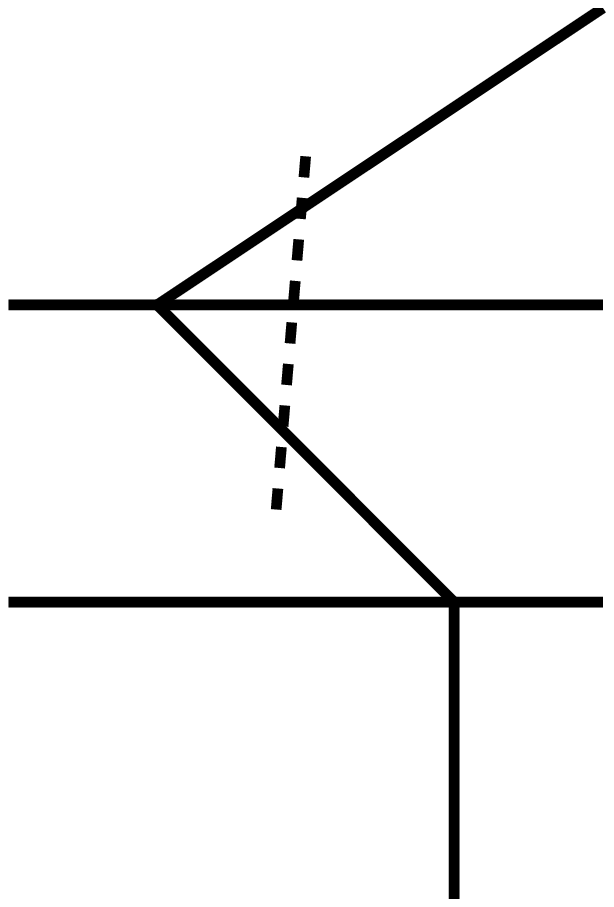}
                \includegraphics[height=1.6cm,width=1.2cm]{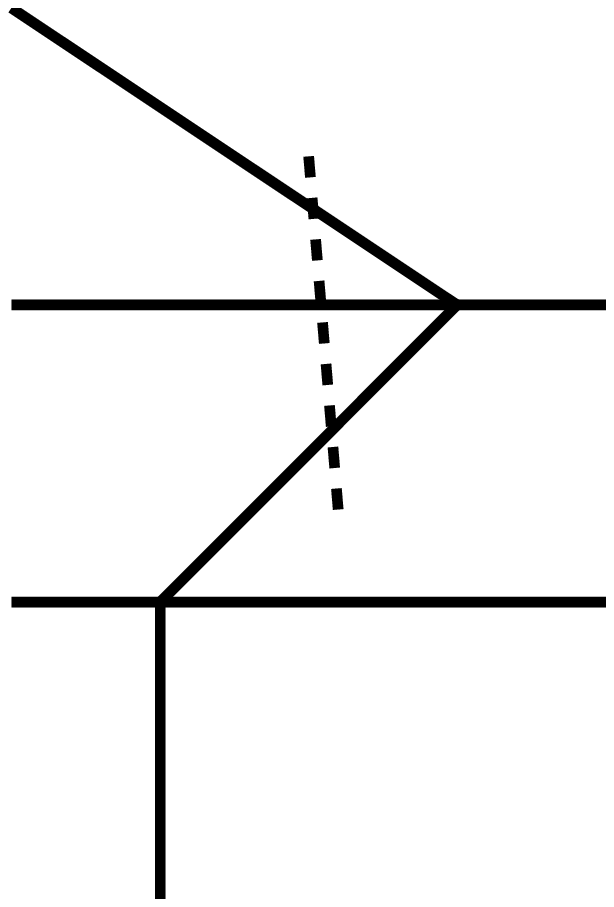}} \,
\subfloat[H1V2B\label{fig:h1v2}]{\includegraphics[height=1.6cm,width=1.2cm]{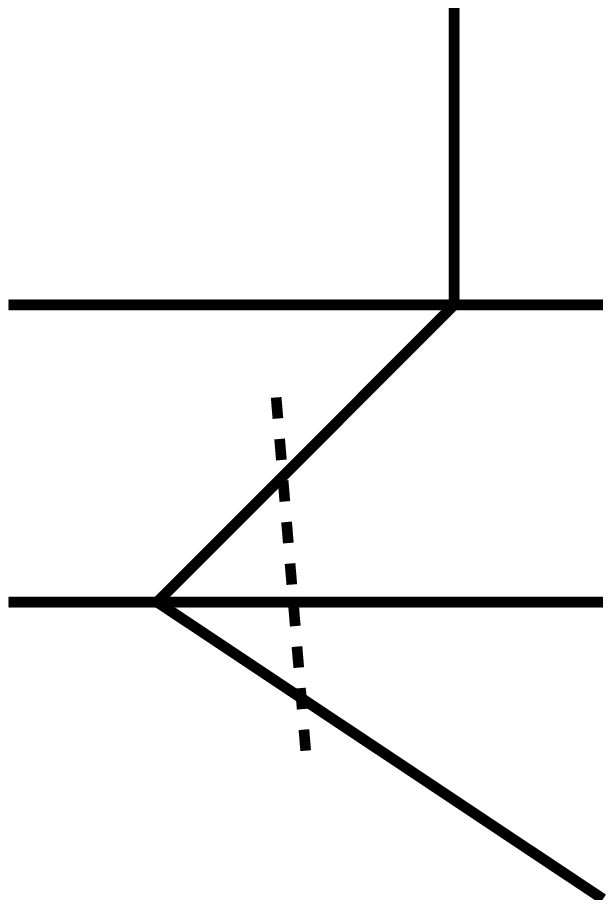}
                \includegraphics[height=1.6cm,width=1.2cm]{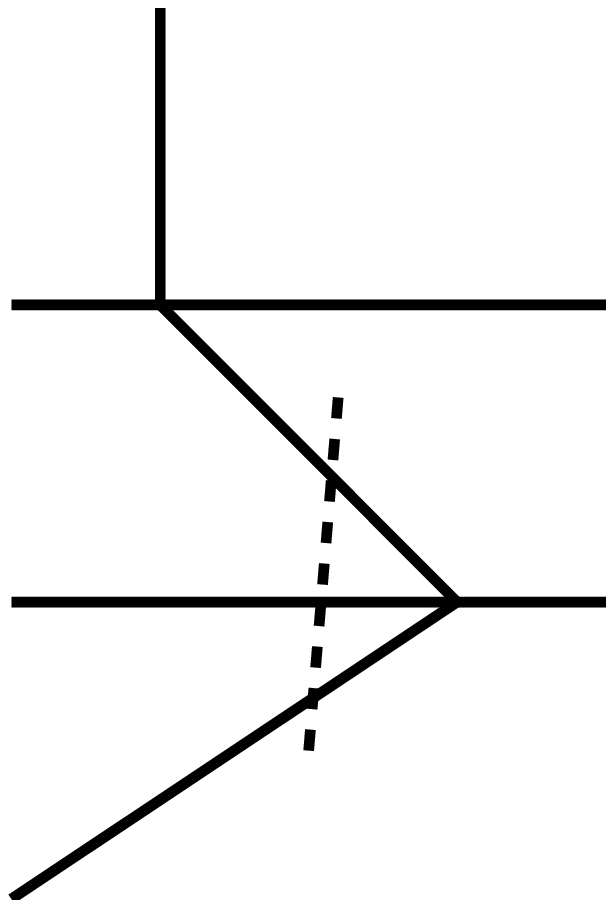} \,
                }
 \center{ shifted cases } \\
  \subfloat[H1V1A\label{fig:h1v1}]{\includegraphics[height=1.6cm,width=1.2cm]{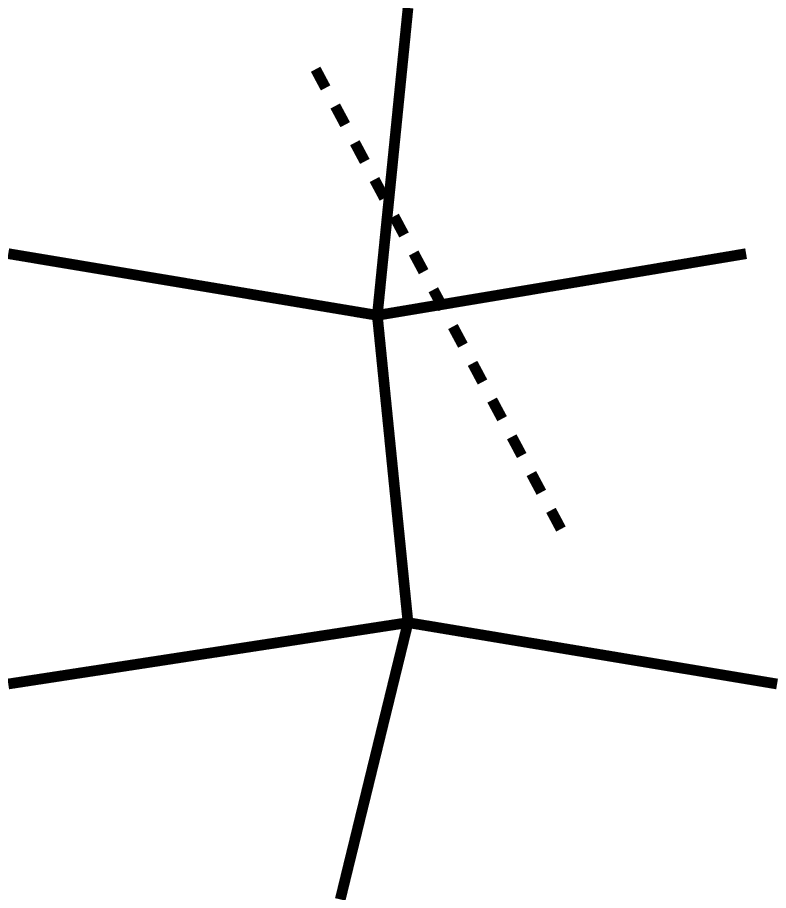}
                \includegraphics[height=1.6cm,width=1.2cm]{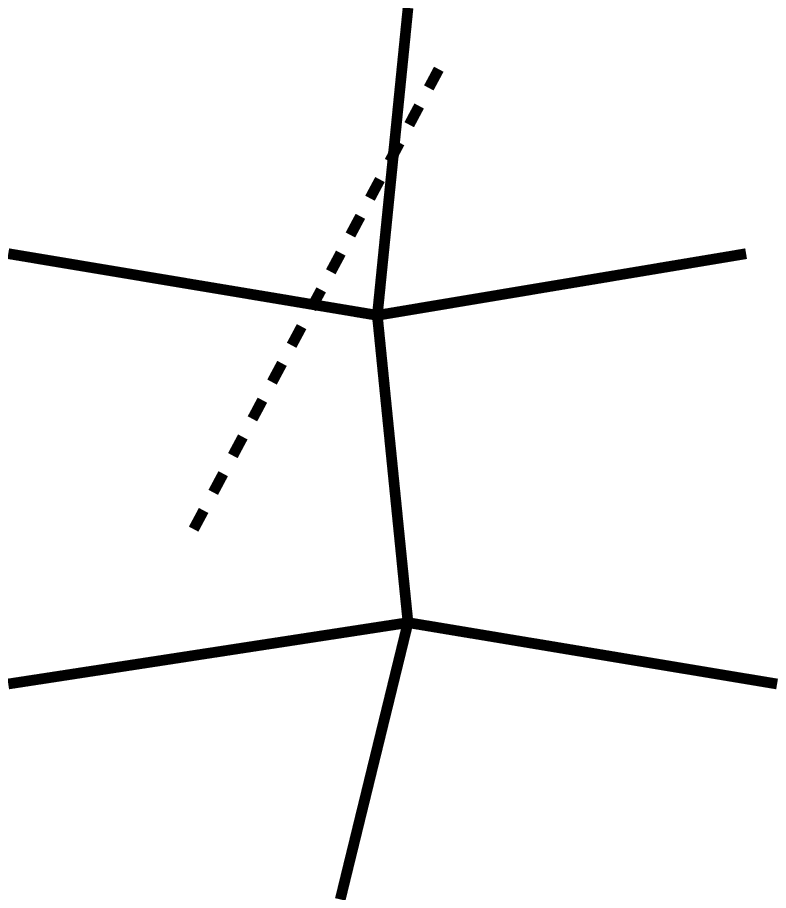}}\,
\subfloat[H1V1A]{\includegraphics[height=1.6cm,width=1.2cm]{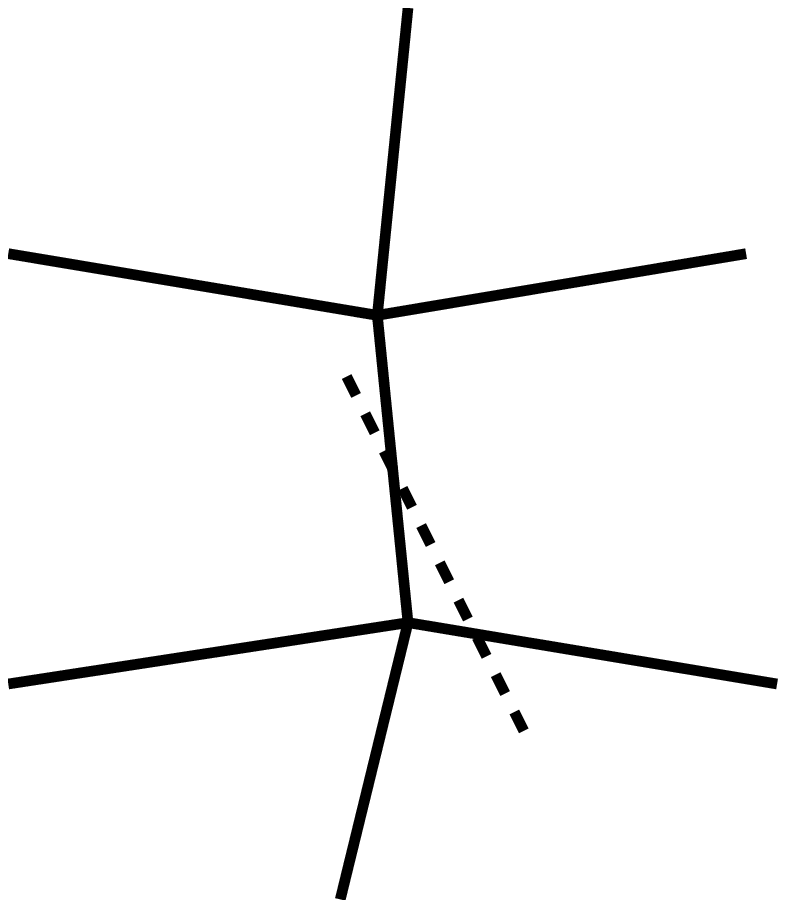}
                \includegraphics[height=1.6cm,width=1.2cm]{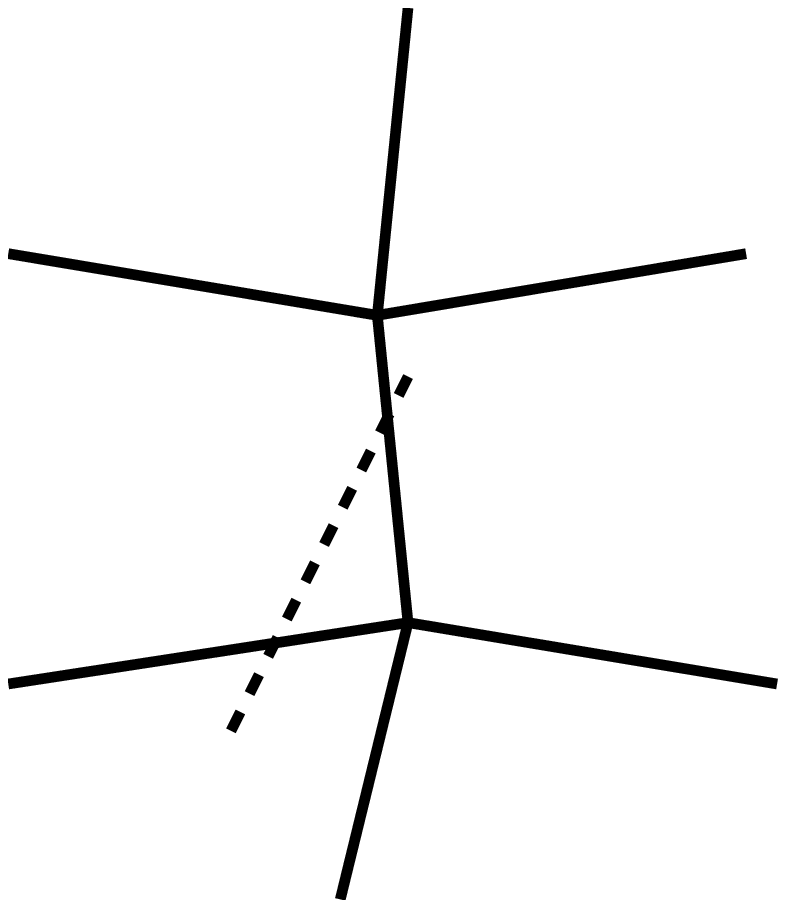}}\,
\subfloat[H1V1B]{\includegraphics[height=1.6cm,width=1.2cm]{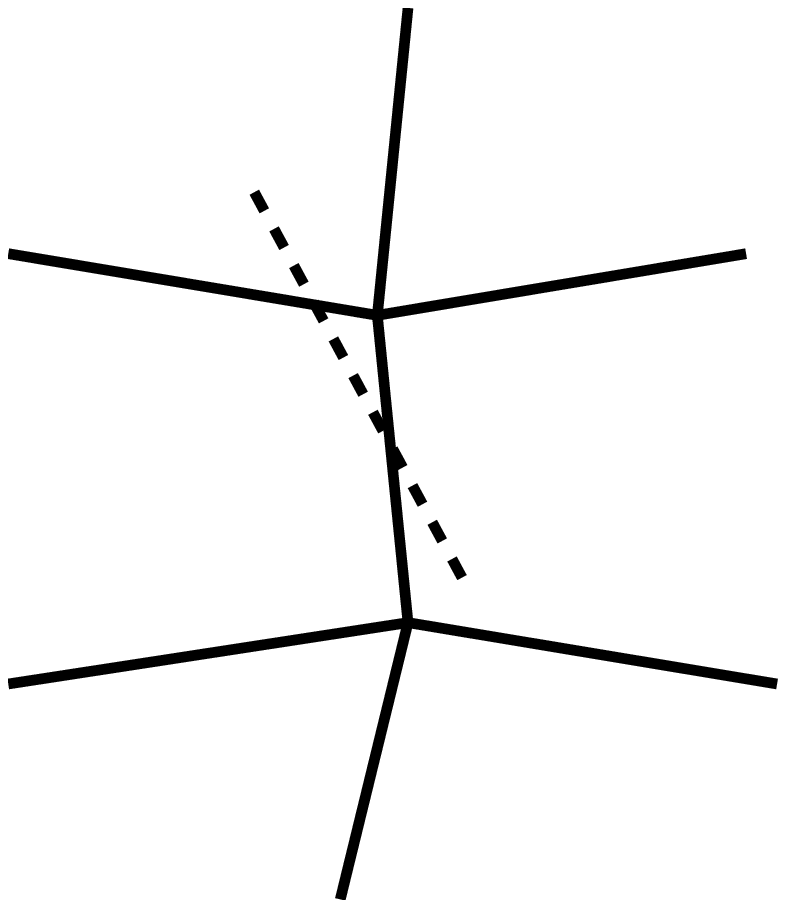}
                \includegraphics[height=1.6cm,width=1.2cm]{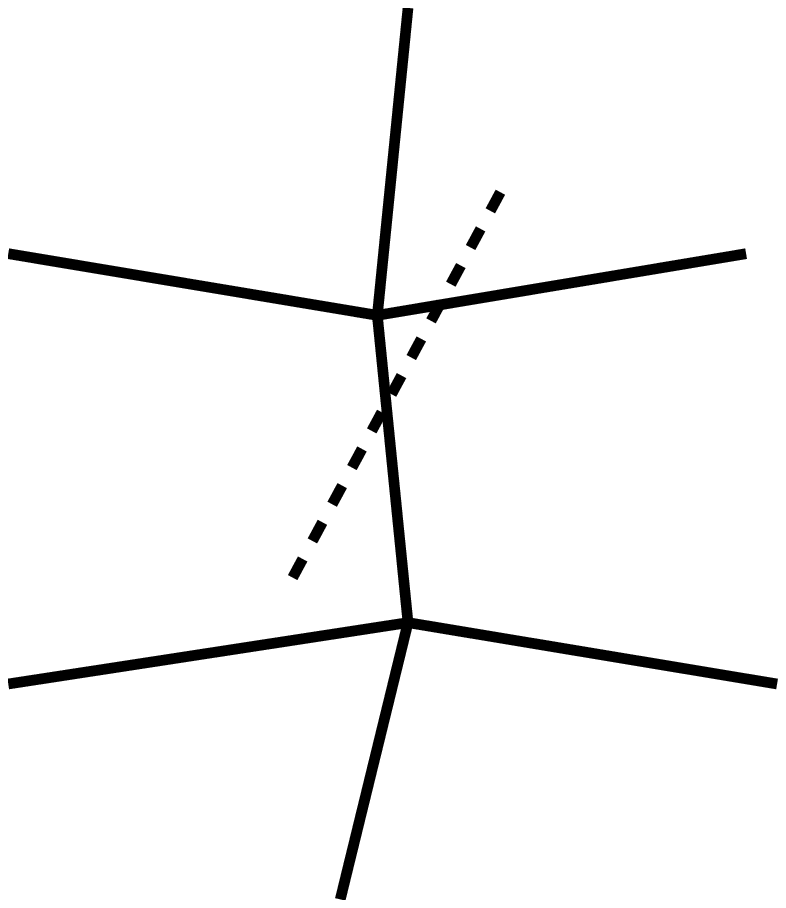}}\,
\subfloat[H1V1B\label{fig:h1v1c}]{\includegraphics[height=1.6cm,width=1.2cm]{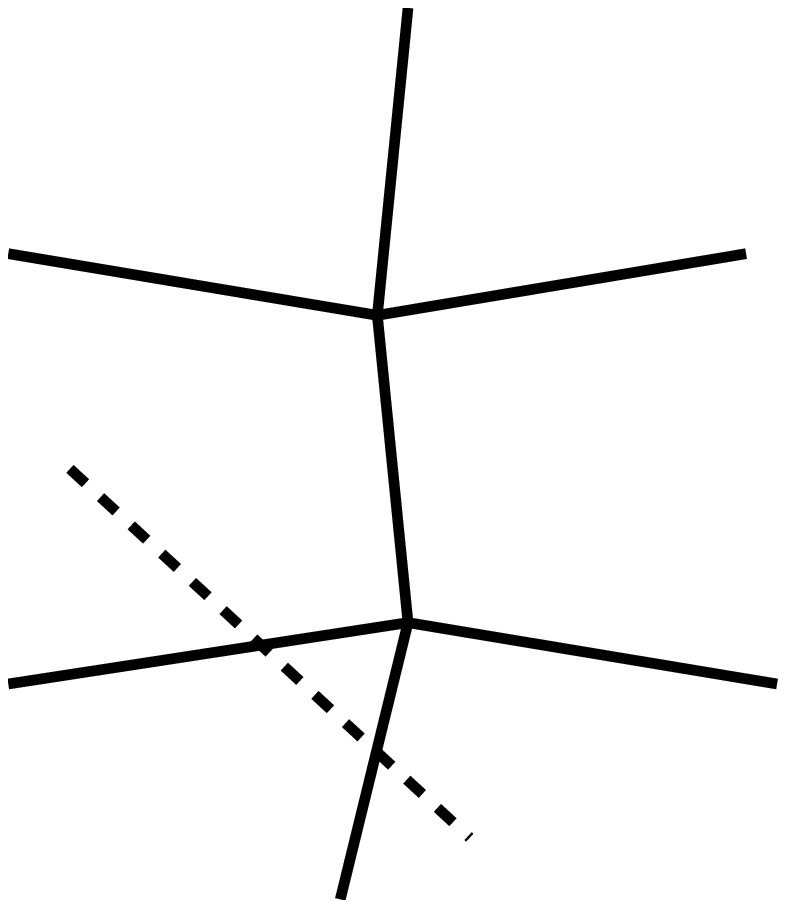}
                \includegraphics[height=1.6cm,width=1.2cm]{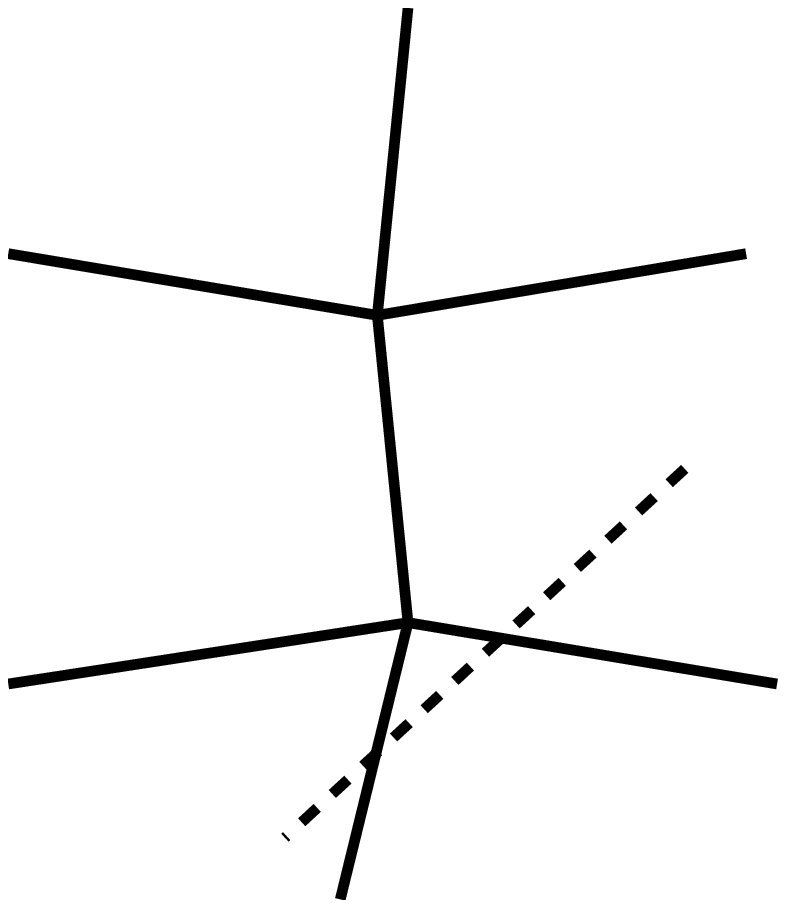}}
 \center{ diagonally shifted cases}\\

 \subfloat[H2V0\label{fig:h2v0}]{\includegraphics[height=1.6cm,width=1.2cm]{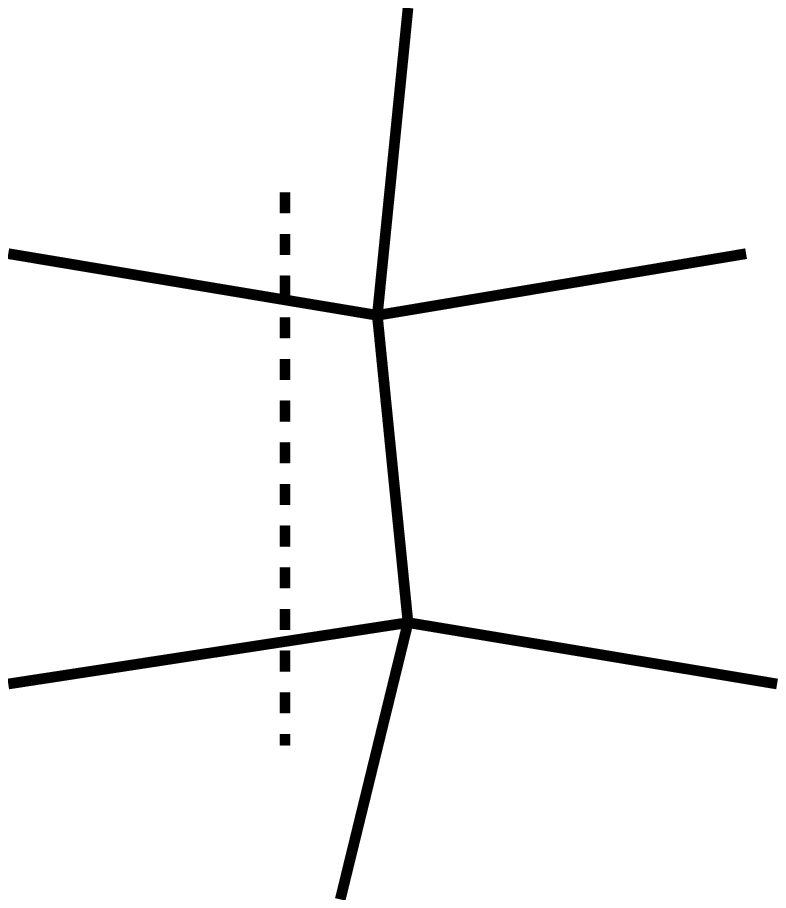}
            \includegraphics[height=1.6cm,width=1.2cm]{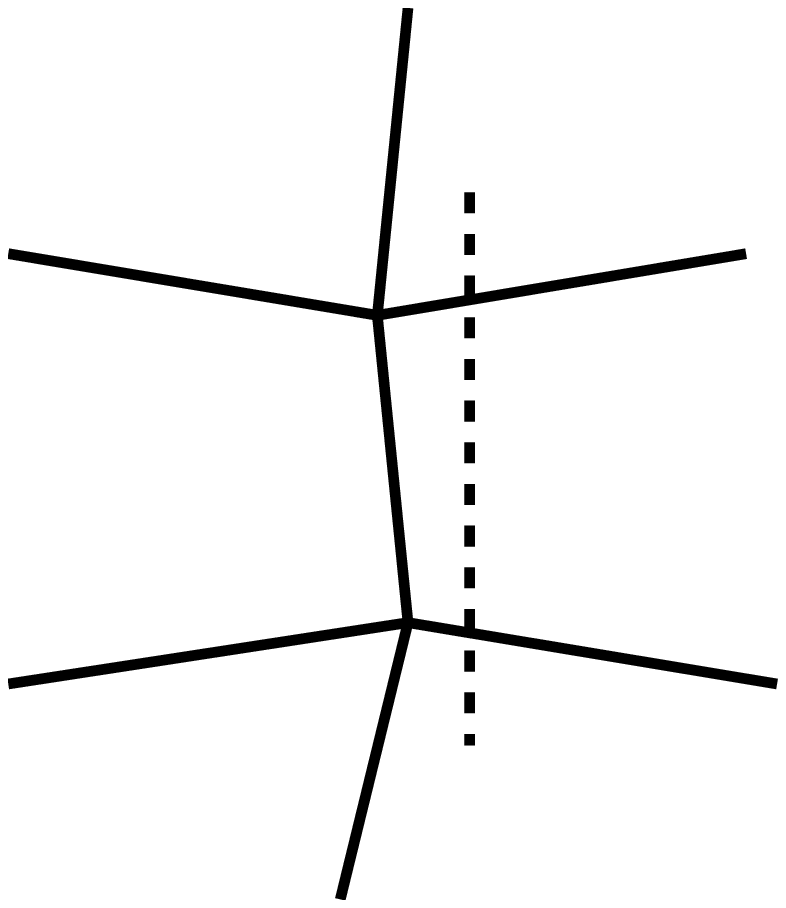}} \,
 \subfloat[H2V2A]{\includegraphics[height=1.6cm,width=1.2cm]{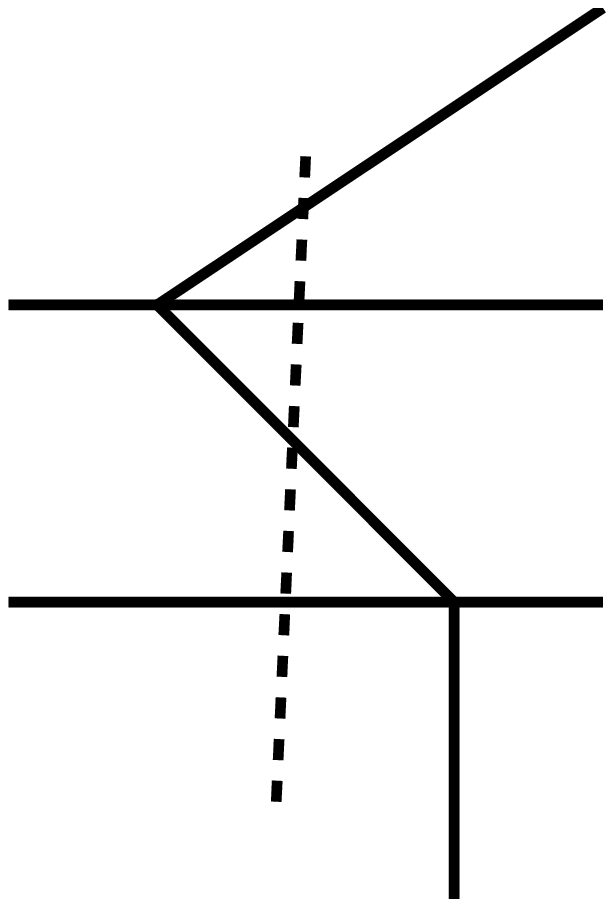}
                \includegraphics[height=1.6cm,width=1.2cm]{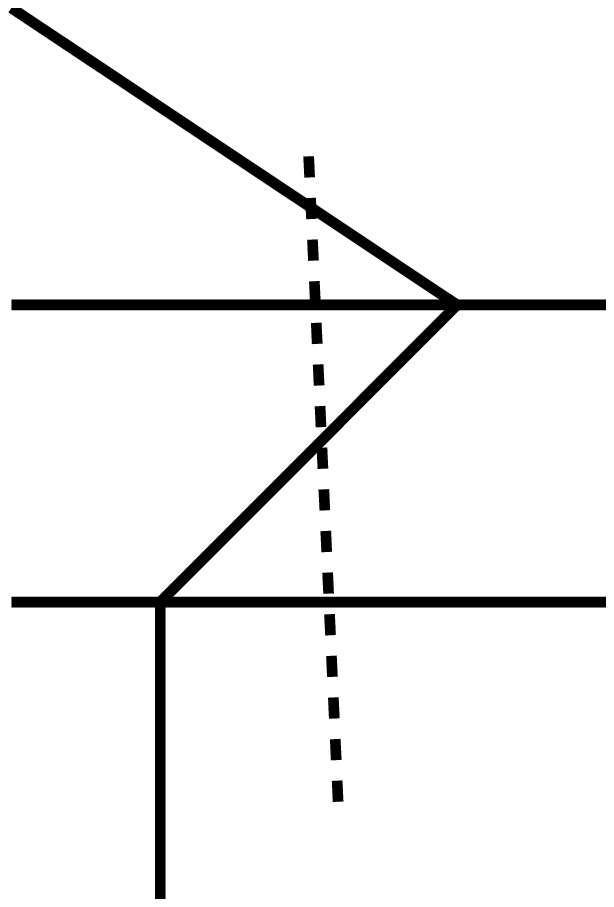}} \,
 \subfloat[H2V2B\label{fig:h2v2b}]{\includegraphics[height=1.6cm,width=1.2cm]{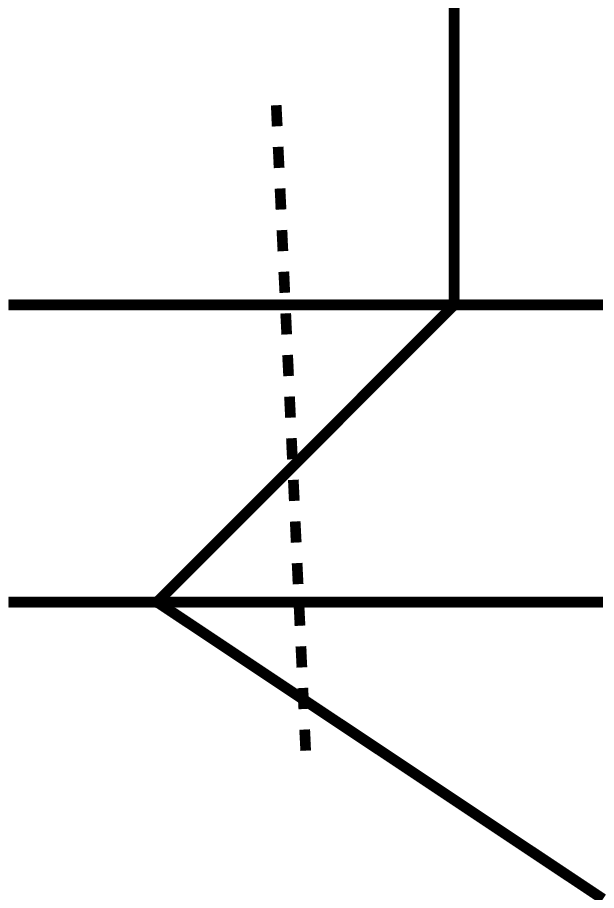}
                \includegraphics[height=1.6cm,width=1.2cm]{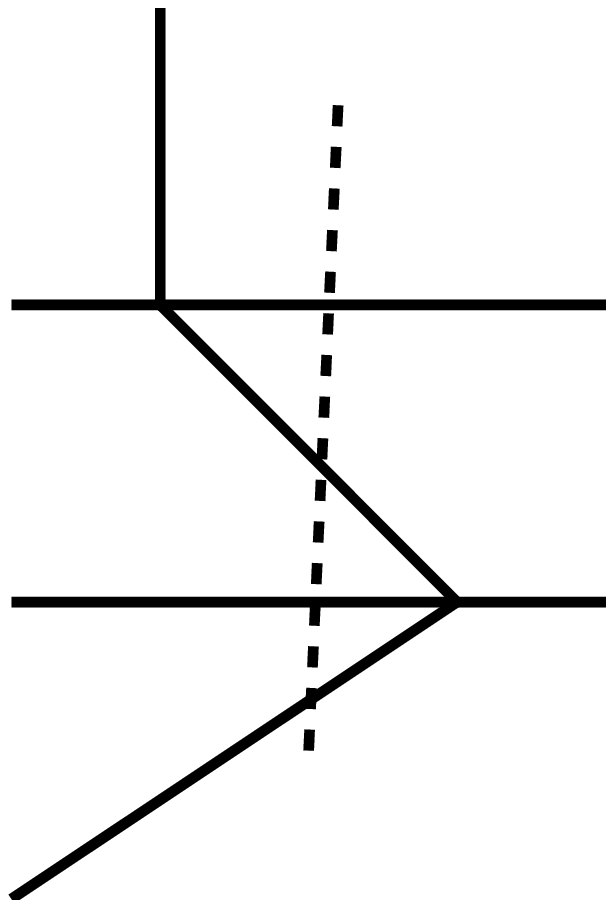}} \,
 \center{stretched cases.}

\subfloat[H2V1A\label{fig:h2v1}]{\includegraphics[height=1.6cm,width=1.2cm]{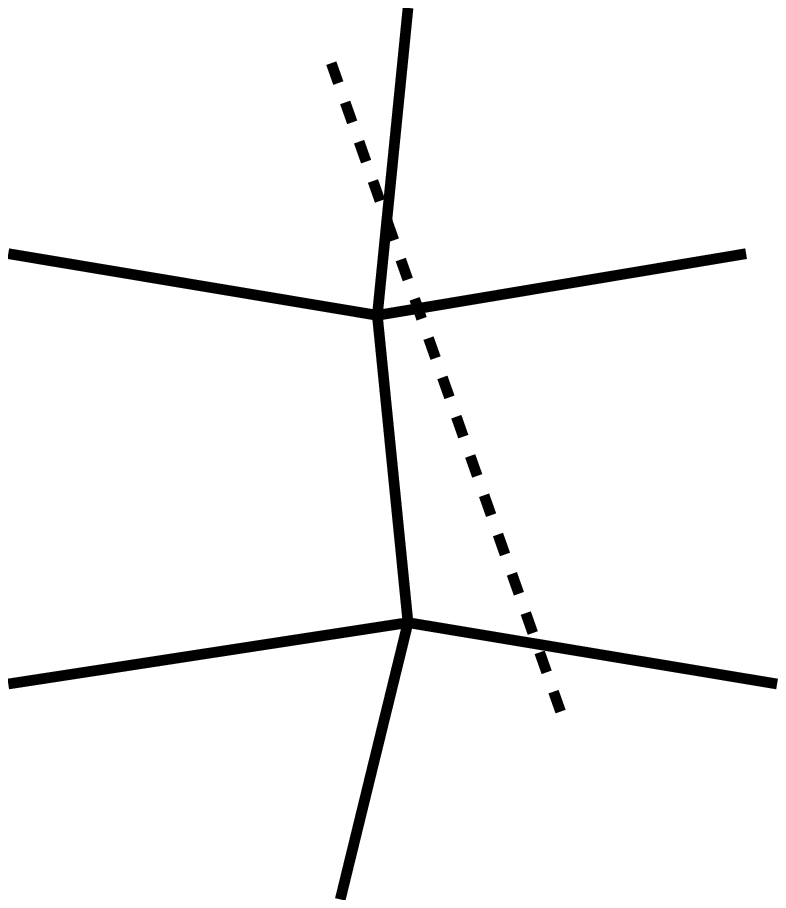}
            \includegraphics[height=1.6cm,width=1.2cm]{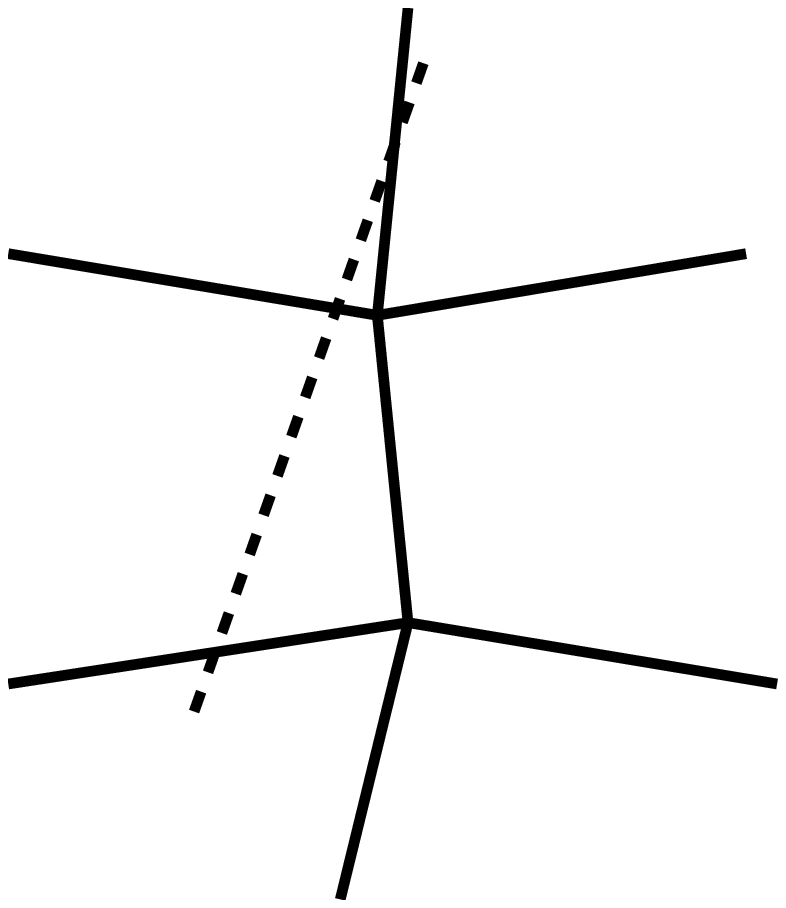}} \,
\subfloat[H2V1B]{\includegraphics[height=1.6cm,width=1.2cm]{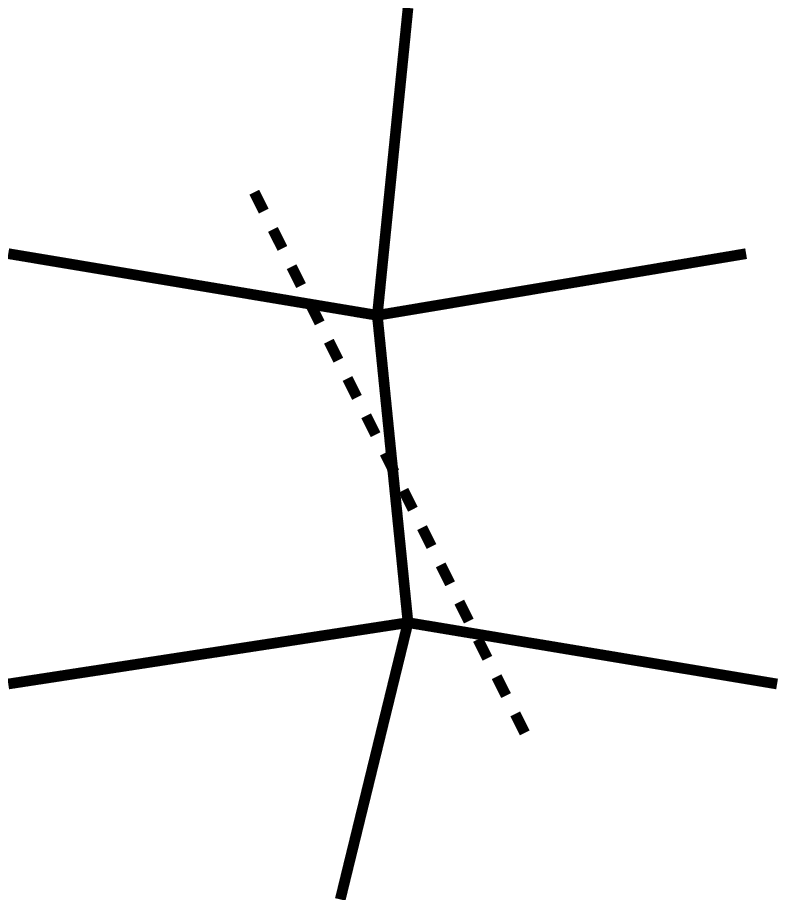}
            \includegraphics[height=1.6cm,width=1.2cm]{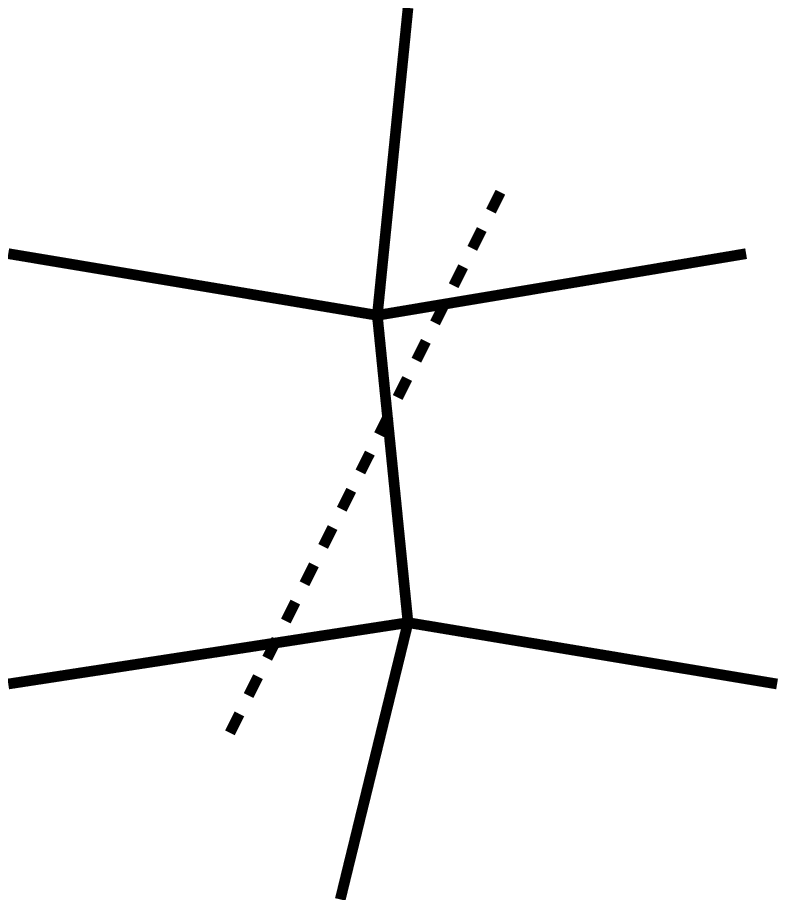}} \,
\subfloat[H2V1C]{\includegraphics[height=1.6cm,width=1.2cm]{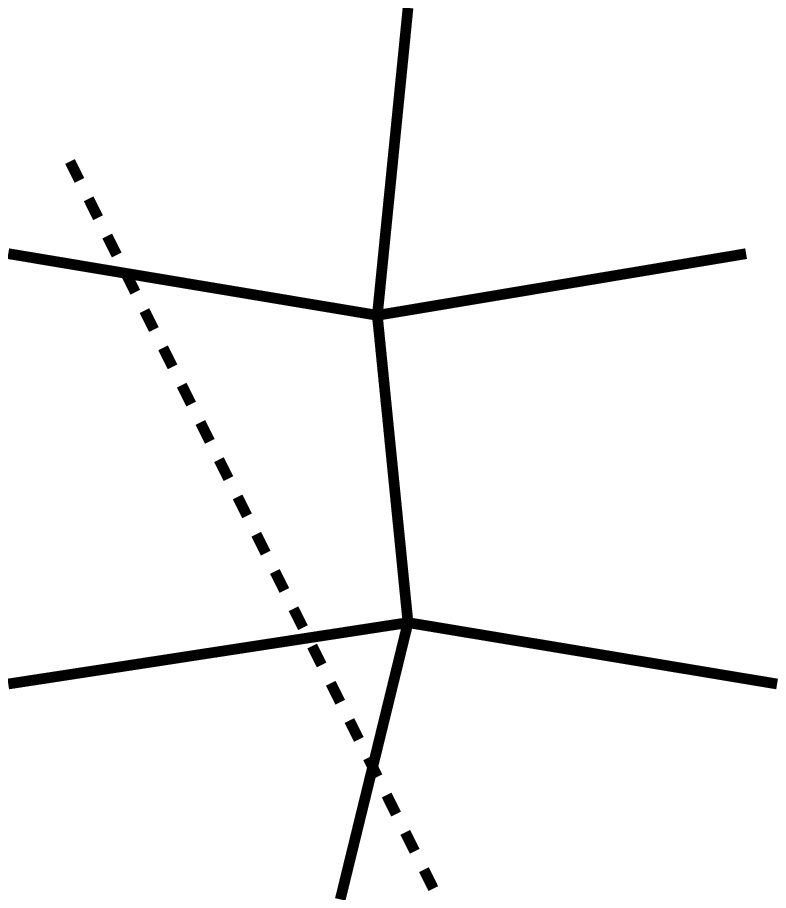}
            \includegraphics[height=1.6cm,width=1.2cm]{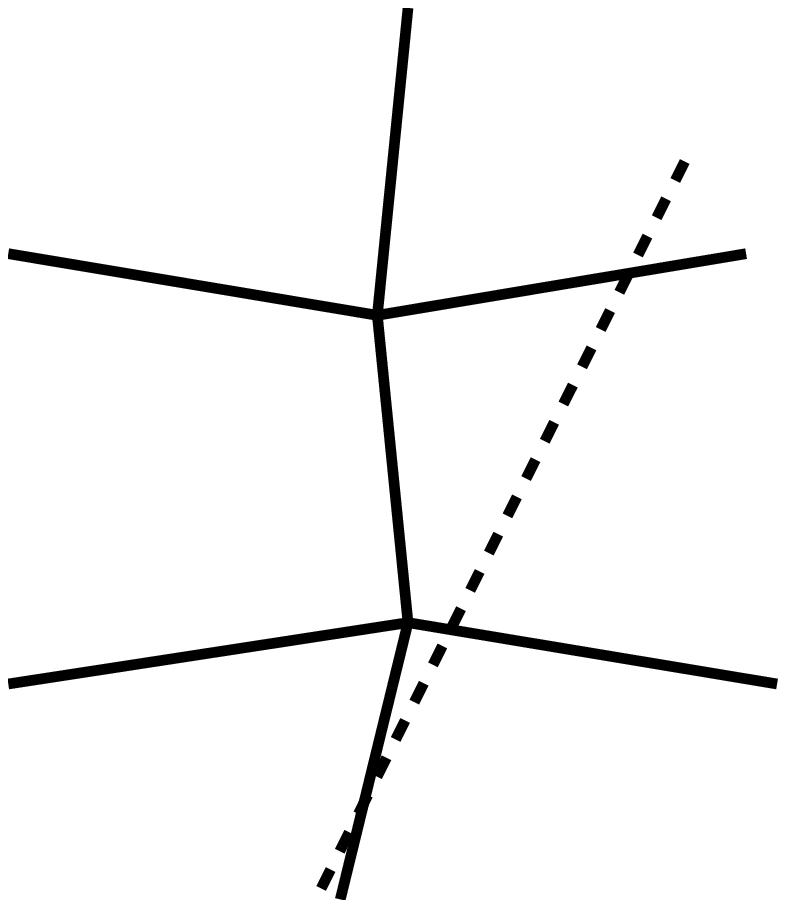}}
\subfloat[H2V3\label{fig:h2v3}]{\includegraphics[height=1.6cm,width=1.2cm]{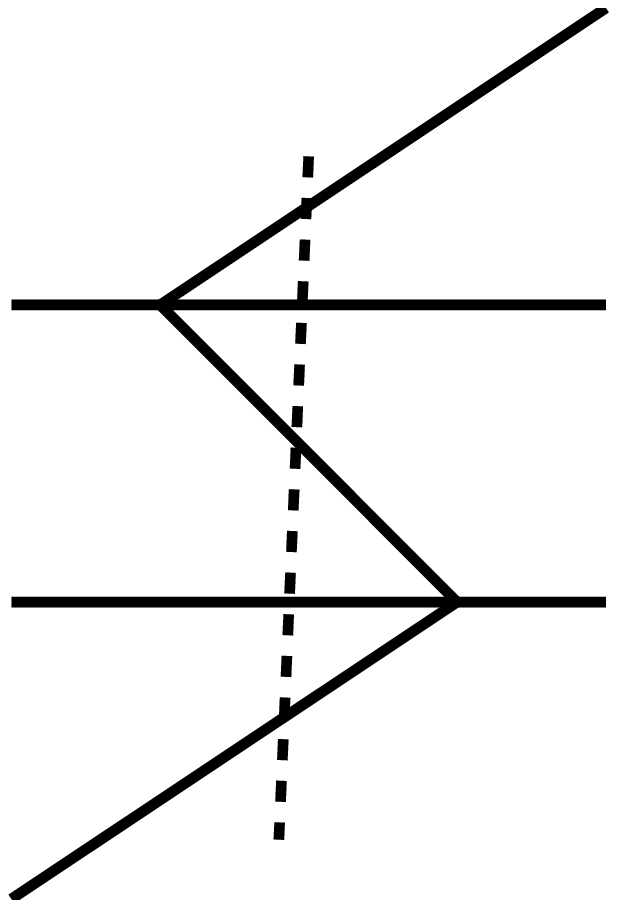}
               \includegraphics[height=1.6cm,width=1.2cm]{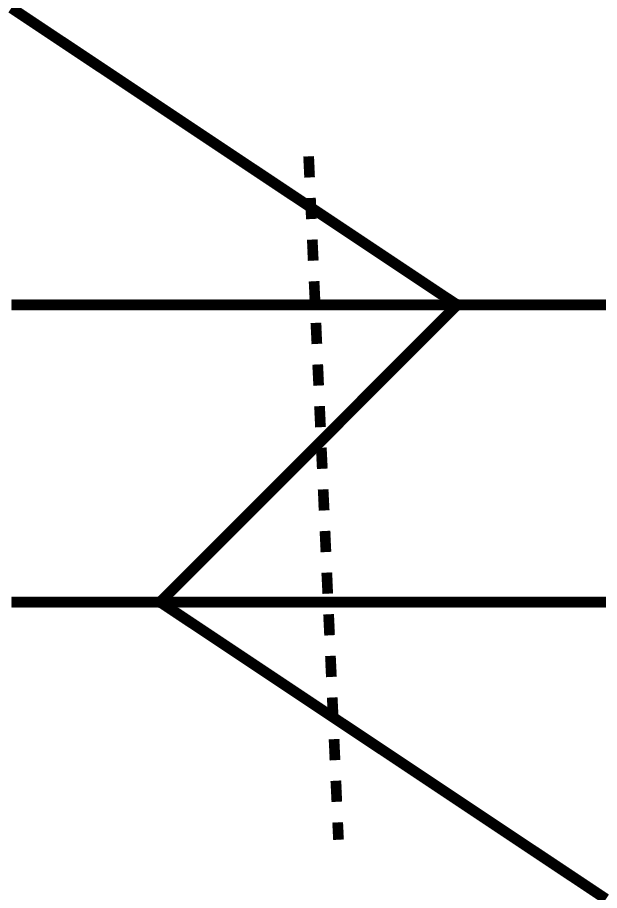}}
\center{diagonally stretched cases}
\caption{Intersections between a face(dashed line) and a local frame}
\label{fig:case}
\end{figure}

\subsection{Fluxing/swept area and local swap region in a local frame}
For the FB/DC method, once the intersection between $F^b_{i,j+\frac{1}{2}}$ between $\LF^a_{i,j+\frac{1}{2}}$ is determined, then the shape of the swept area $\partial F_{i,j+\frac{1}{2}}$ will be determined. There are total 17 symmetric cases as shown in Fig. \ref{fig:case}. On contrast, the local swap region requires additional effort to be identified.

The vertex $Q_{i,j}$ lies on the line segment $y^b_{j}(t)$, according the Assumption A2, $y^b_{j}(t)$ can only have one intersection with $x^a_i(t)$. This intersection point will be referred to as $V_1$.
The line segment $Q_{i,j}V_1$ can have 0 or 1 intersection with the local frame $\LF_{i,j+\frac{1}{2}}$ except $V_1$ itself; the intersection point, if any, will be denoted as $V_2$.

Table \ref{tab:V} describes how the line segment $Q_{i,j}V_1$ intersects with the local frame according to the relative position of $Q_{i,j}$. The north and south boundary of the local swap region therefore can have 1 or 2 intersection with the local frame.  Therefore each case in Figure \ref{fig:case} can result up to four possible local swap region. We use the intersection number between the up and south boundary and the local frame to classify the four cases, see Figure \ref{fig:cswap} for an illustration. For certain cases, it is impossible for the point $Q_{i,j}V_1$ have two intersection points with the local frame, 98 possible combinations are illustrated in Figure~\ref{fig:branch}.

 \begin{table}[!t]
 \centering
 \caption{Intersection cases of the up/down edge of a local swap region}
 \begin{tabular}{cc|c|c|c|c}
 \hline
     &         & \multicolumn{4}{c}{position of $Q_{i,j}$ in the local frame} \\\cline{3-6}
  \# & points  & $B_1$ & $B_2$ & $B_3$ & $B_4$ \\ \hline
1 & $V_1=$&  $F^b_{i-\frac{1}{2},j} \cap F^a_{i,j+\frac{1}{2}}$  & $F^b_{i+\frac{1}{2},j} \cap F^a_{i,j+\frac{1}{2}}$ &  $F^b_{i+\frac{1}{2},j} \cap F^a_{i,j-\frac{1}{2}}$  & $F^b_{i-\frac{1}{2},j} \cap F^a_{i,j-\frac{1}{2}}$  \\
                              \hline
 2 & $V_1=$  &   $F^b_{i-\frac{1}{2},j} \cap F^a_{i,j-\frac{1}{2}}$  & $F^b_{i+\frac{1}{2},j} \cap F^a_{i,j-\frac{1}{2}}$ &  $F^b_{i+\frac{1}{2},j} \cap F^a_{i,j+\frac{1}{2}}$  & $F^b_{i-\frac{1}{2},j} \cap F^a_{i,j+\frac{1}{2}}$  \\
  & $V_2=$ & $F^b_{i-\frac{1}{2},j} \cap F^a_{i+\frac{1}{1},j}$  & $F^b_{i+\frac{1}{2},j} \cap F^a_{i-\frac{1}{2},j}$ &  $F^b_{i+\frac{1}{2},j} \cap F^a_{i-\frac{1}{2},j}$  & $F^b_{i-\frac{1}{2},j} \cap F^a_{i+\frac{1}{2},j}$  \\
 \hline

 \end{tabular}
 \label{tab:V}
 \end{table}

\begin{figure}[!t]
\centering

\subfloat[U1S1]{\includegraphics[height=1.6cm,width=1.2cm]{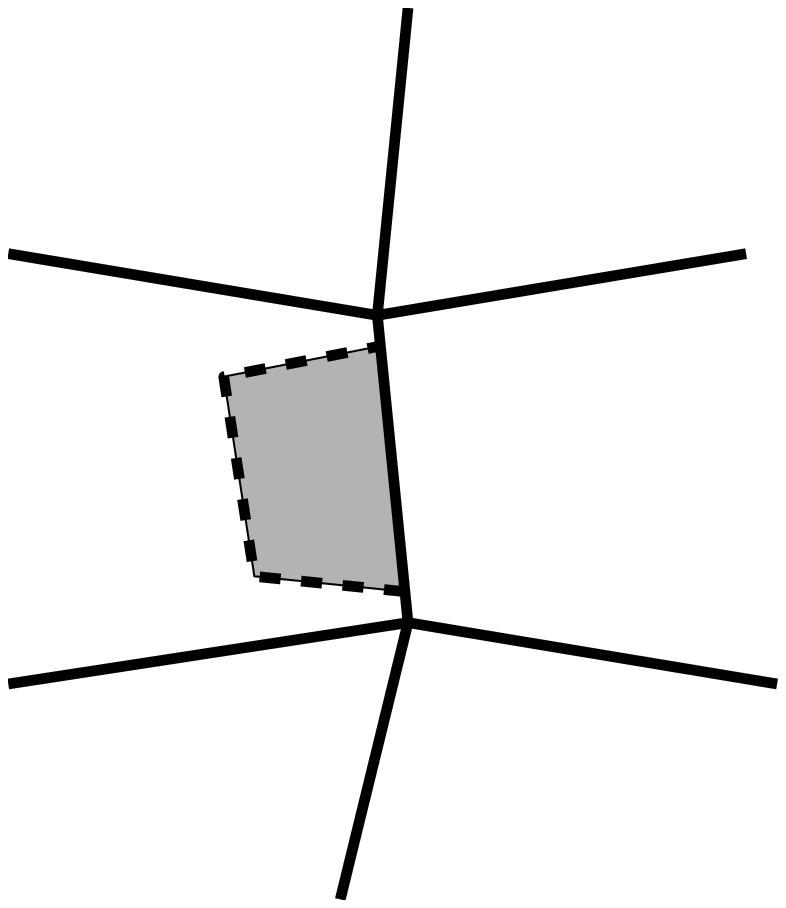}
               \includegraphics[height=1.6cm,width=1.2cm]{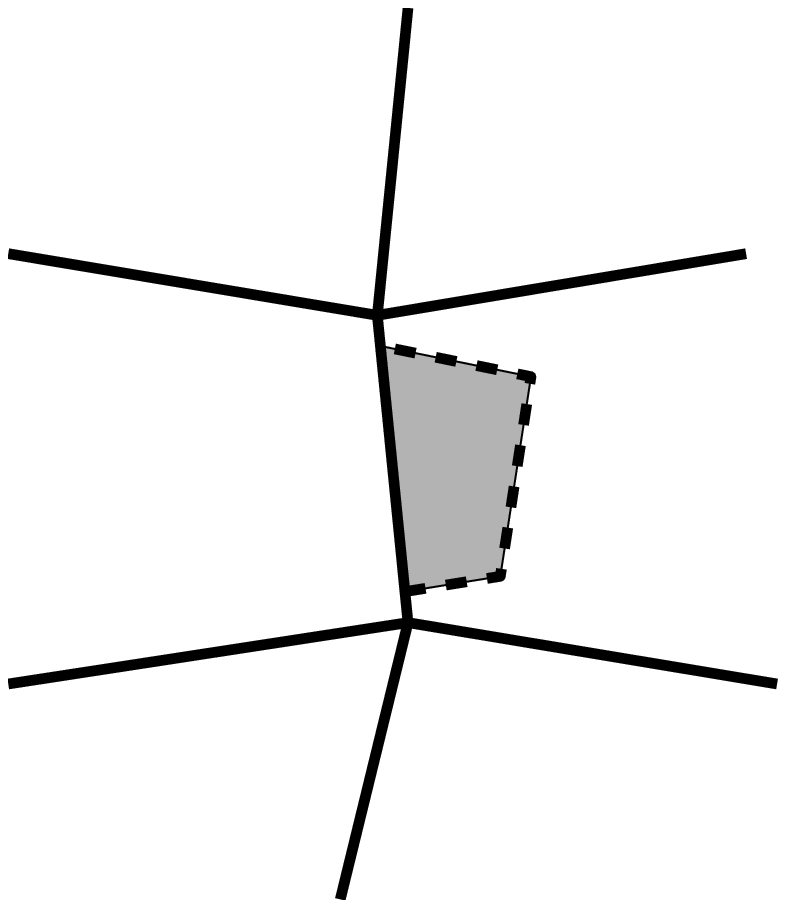}}\,
\subfloat[U2S2]{\includegraphics[height=1.6cm,width=1.2cm]{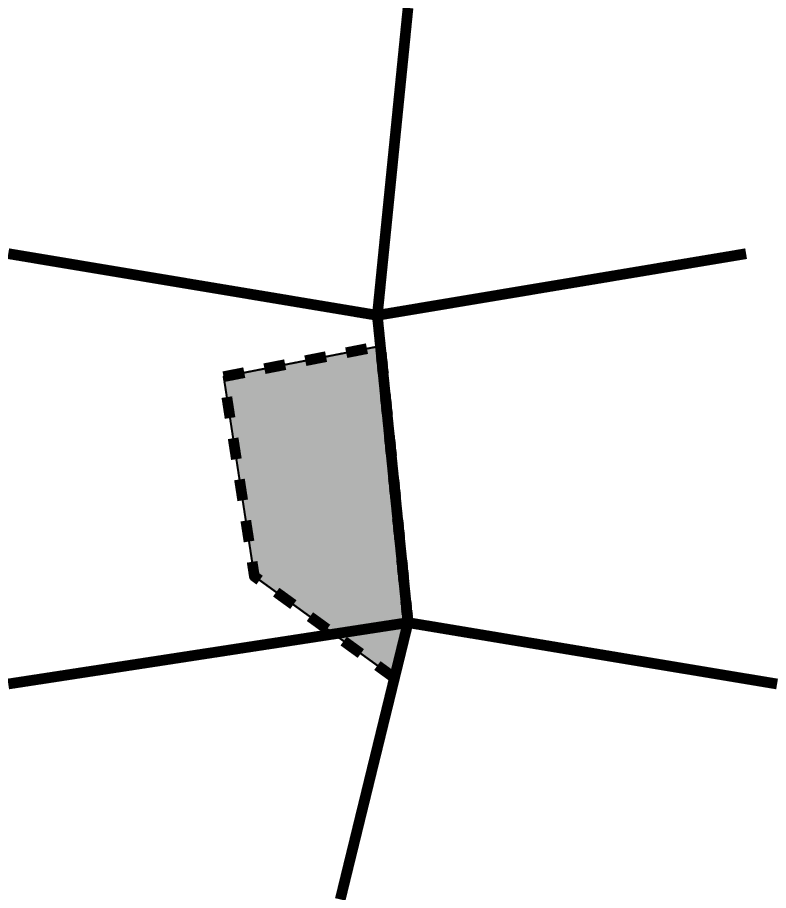}
               \includegraphics[height=1.6cm,width=1.2cm]{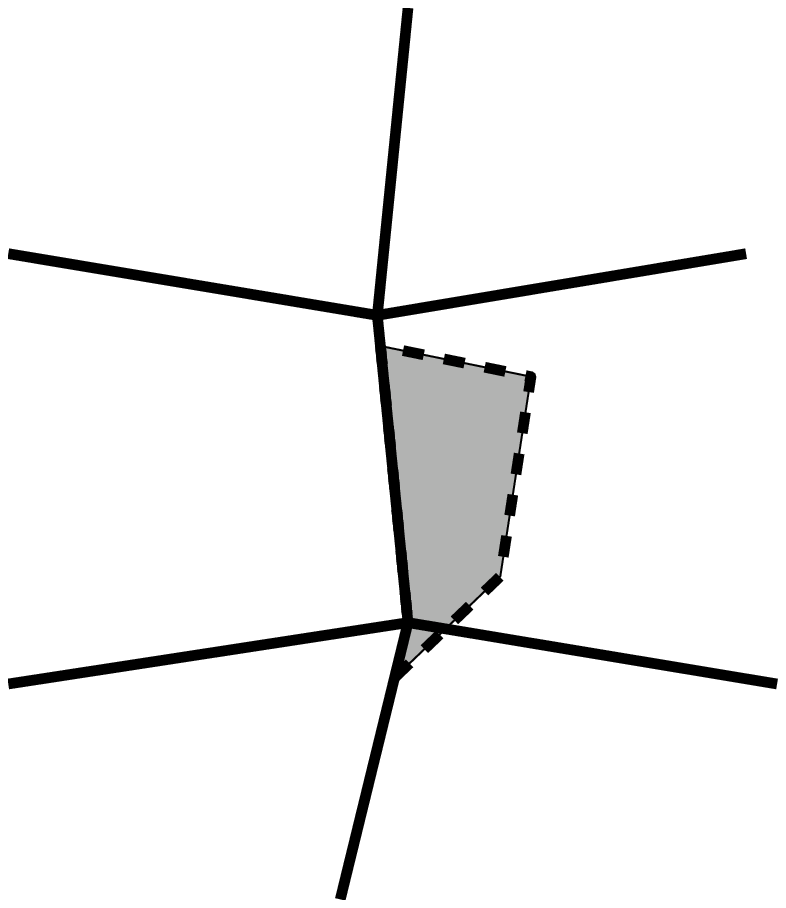}} \,
 \subfloat[U2S1]{\includegraphics[height=1.6cm,width=1.2cm]{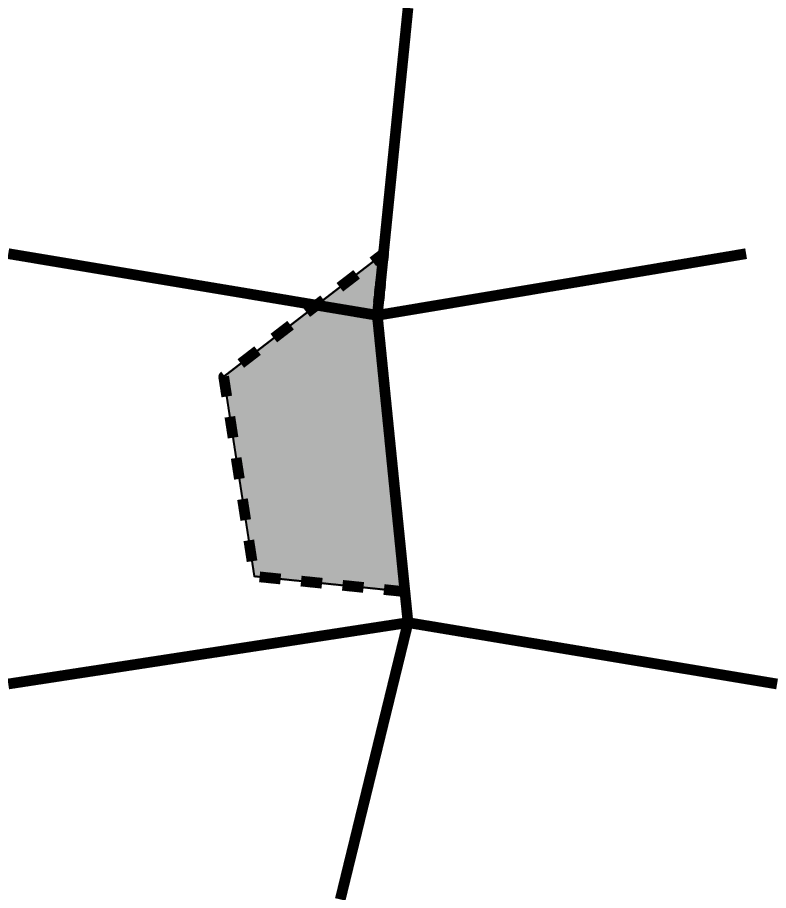}
               \includegraphics[height=1.6cm,width=1.2cm]{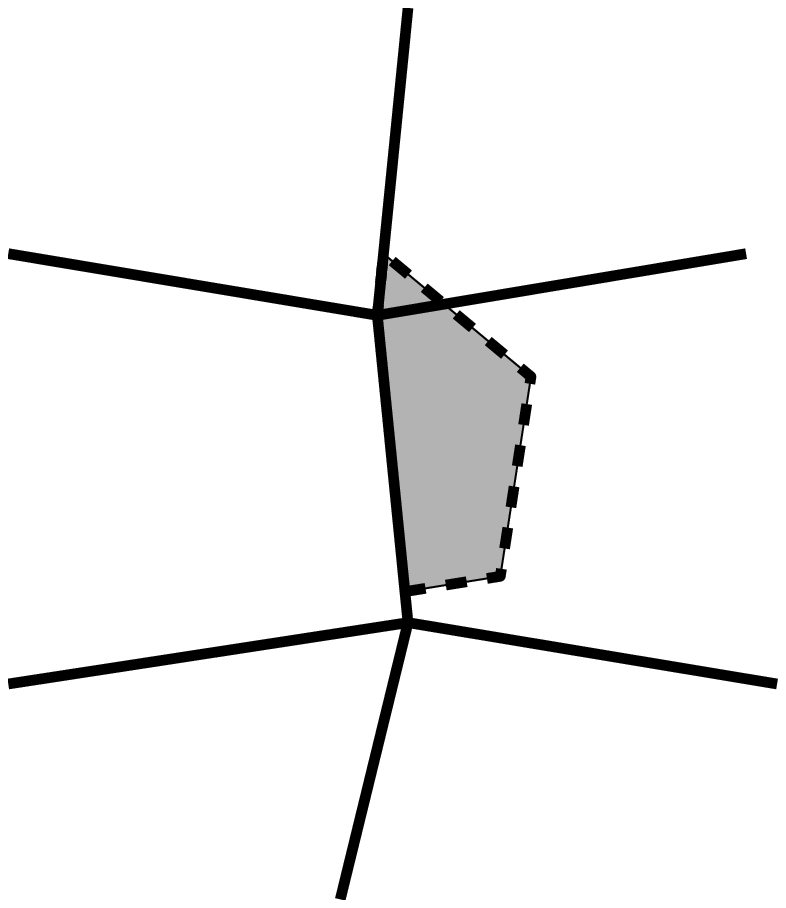}} \,
  \subfloat[U2S2]{\includegraphics[height=1.6cm,width=1.2cm]{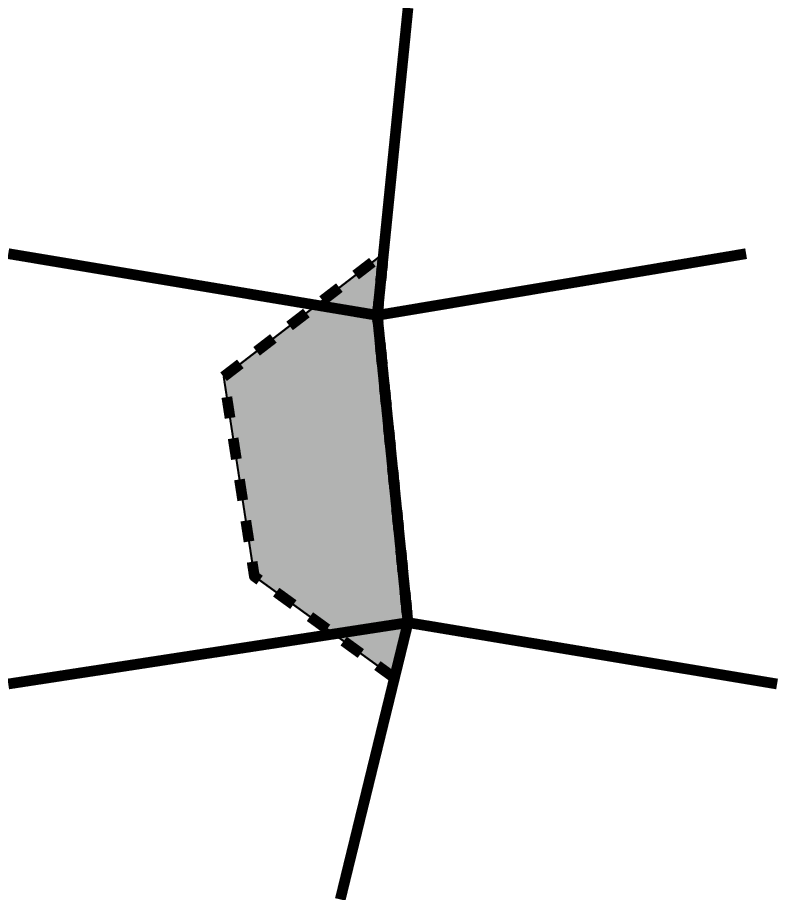}
               \includegraphics[height=1.6cm,width=1.2cm]{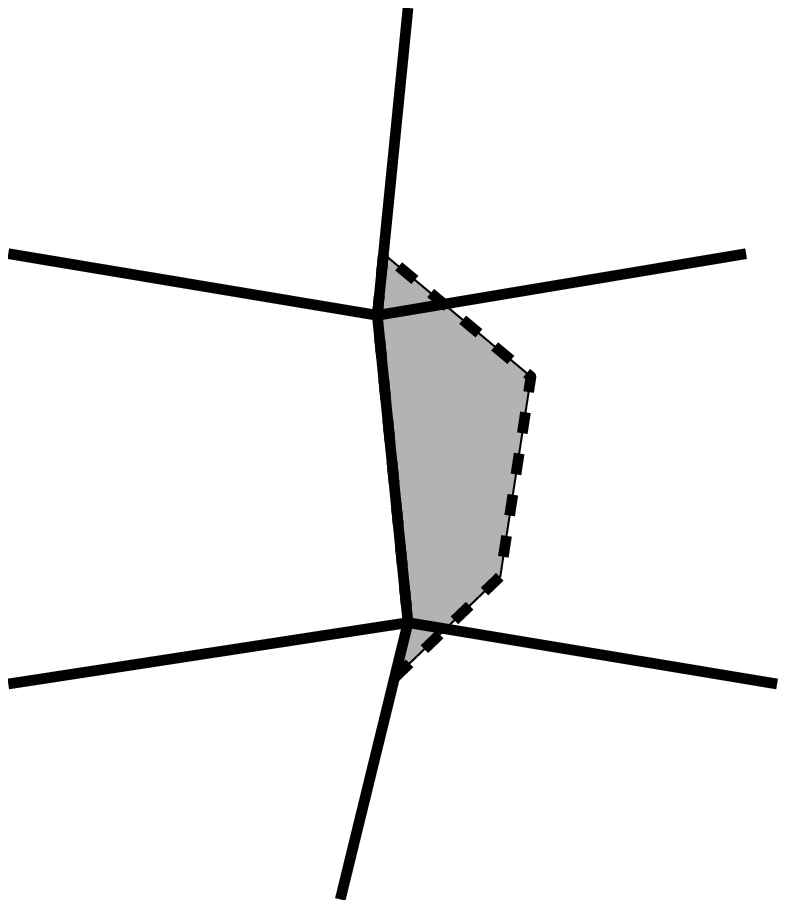}}
\caption{The shapes of local swap region based on the case of H0V0 in Fig.~\ref{fig:case}.}
\label{fig:cswap}
\end{figure}
 \begin{figure}
\centering
\footnotesize
\resizebox{\textwidth}{0.8\textheight}{
\input{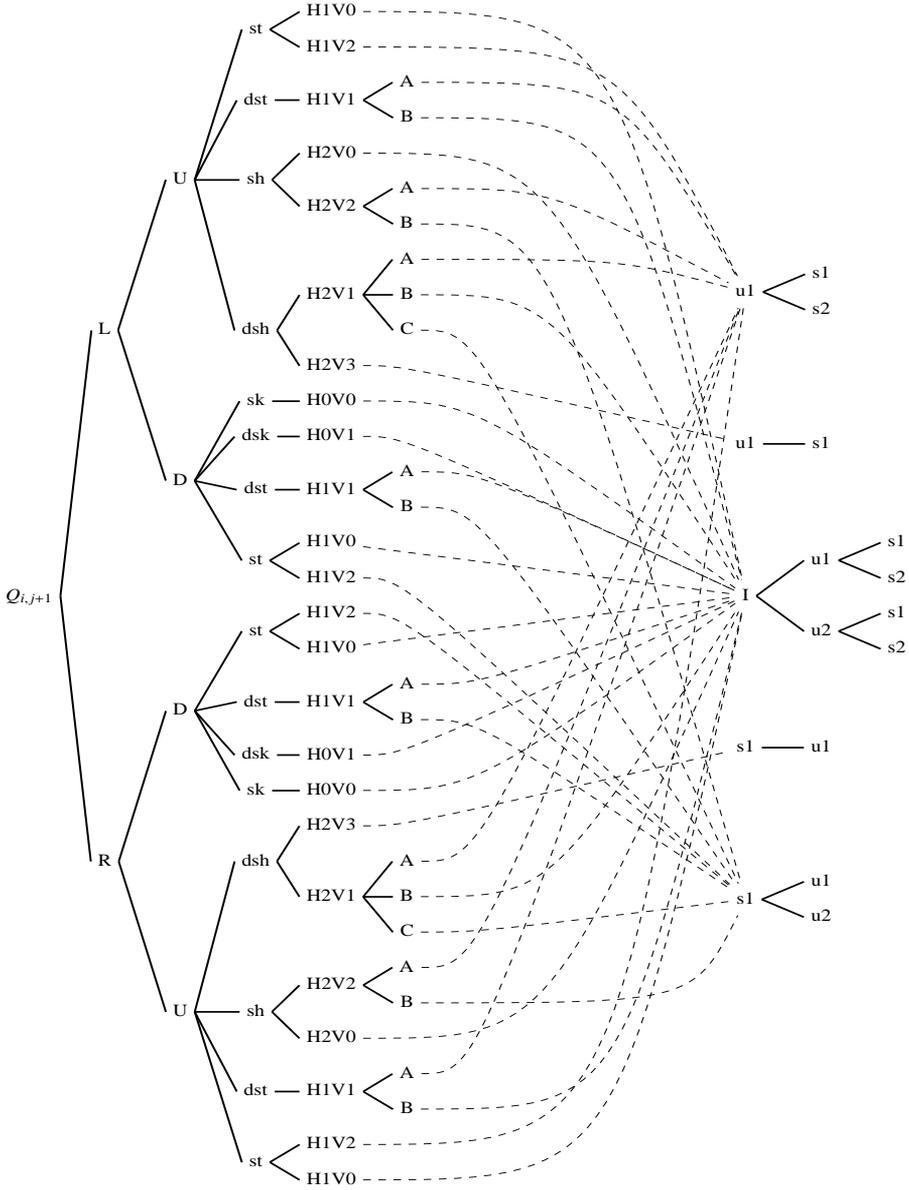}
}
\caption{Classification tree for all possible cases of a local swap region. $Q_{i+1,j+1}$ is on of the vertex of $F^b_{i,j+1/2}$, $L$, $R$ $U$ and $D$ stand for the relative position of $Q_{i,j}$ in the local frame $\LF^a_{i,j+\frac{1}{2}}$. The edge $F^b_{i,j+\frac{1}{2}}$ can be shifted(st), shrunk(sk), stretched(sh), diagonally shifted(dst), diagonally shrunk(dsk), and diagonally stretched(dsh).  }
\label{fig:branch}
\end{figure}

\begin{fact}
 Let $\T^a$ and $\T^b$ be two admissible quadrilateral mesh of the same structure. Under the Assumption A1, A2 and A3, an local swap region has up 98 cases with respect to a local frame.
\end{fact}
Form the Fig.1 in \cite{D84} and \cite{Ramshaw1986}, we shall see that Ramshaw's approach requires at least 98 programming cases Under the Assumption A1, A2 and A3, while the swept region approach requires only 34 programming cases. This is a significant improvement.  When the Assumption A3 fails, or the so called degeneracy of the overlapped region arises, more benefit can be obtained.

\subsection{Degeneracy of the intersection and  signed area of a polygon}
\begin{figure}[!t]
\centering
\subfloat[ ]{\includegraphics[height=1.6cm,width=1.2cm]{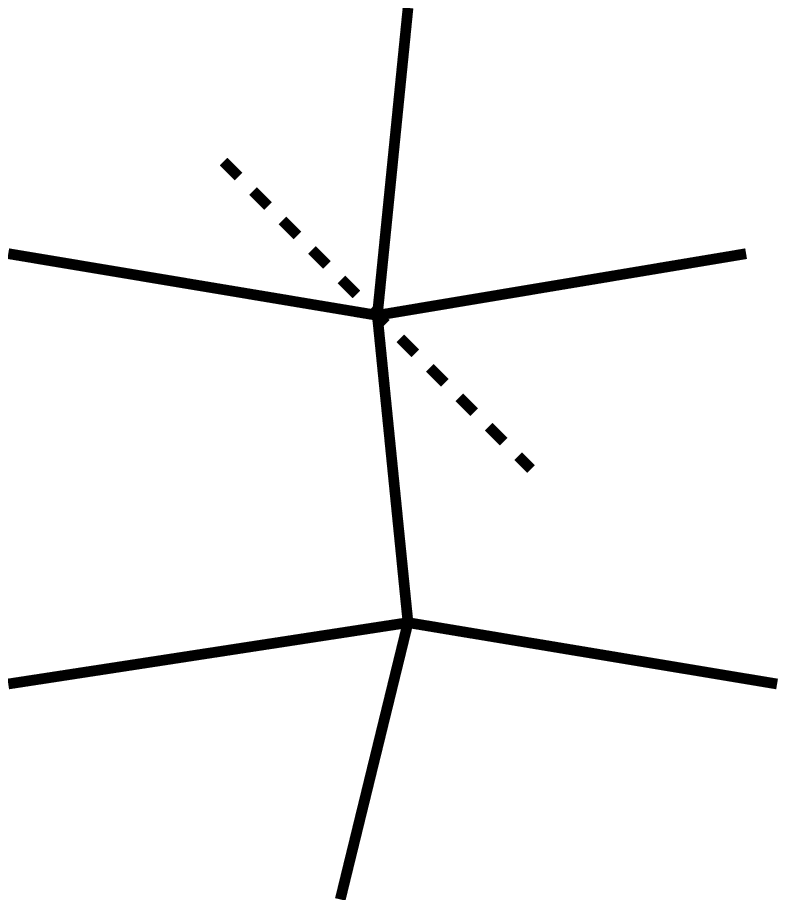}}
\subfloat[ ]{\includegraphics[height=1.6cm,width=1.2cm]{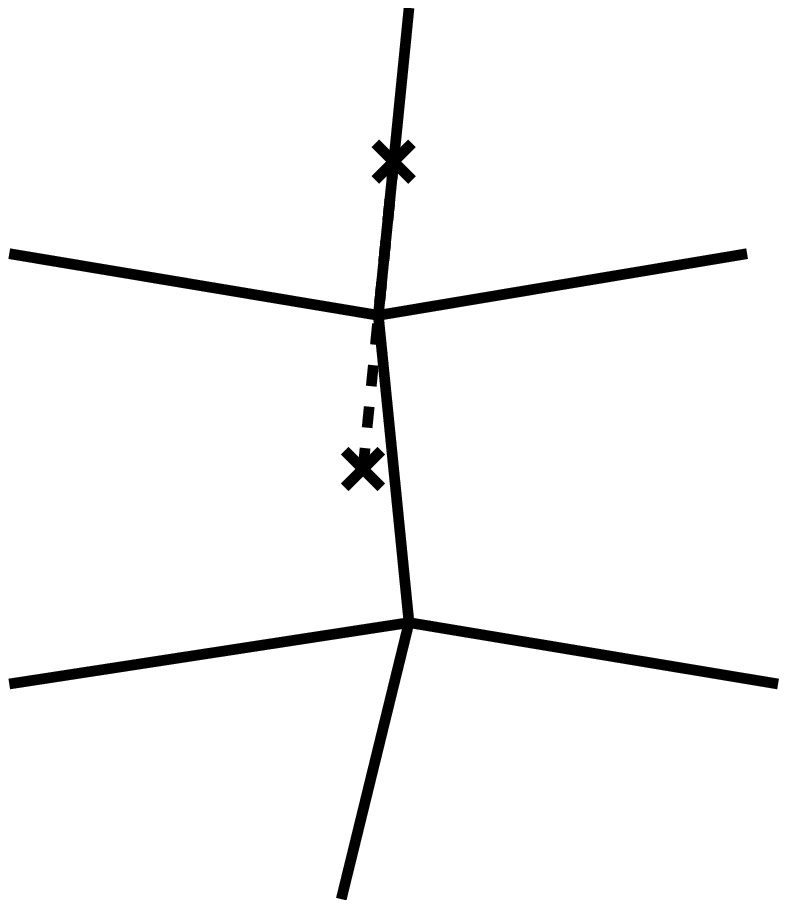}}
\subfloat[ ]{\includegraphics[height=1.6cm,width=1.2cm]{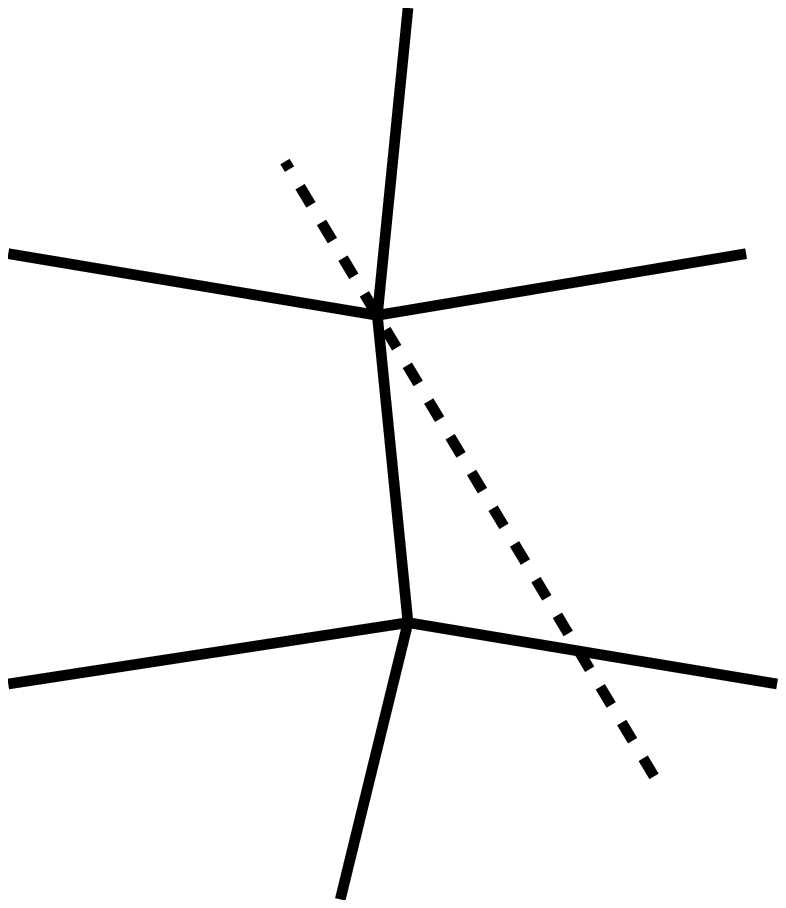}}
\subfloat[ ]{\includegraphics[height=1.6cm,width=1.2cm]{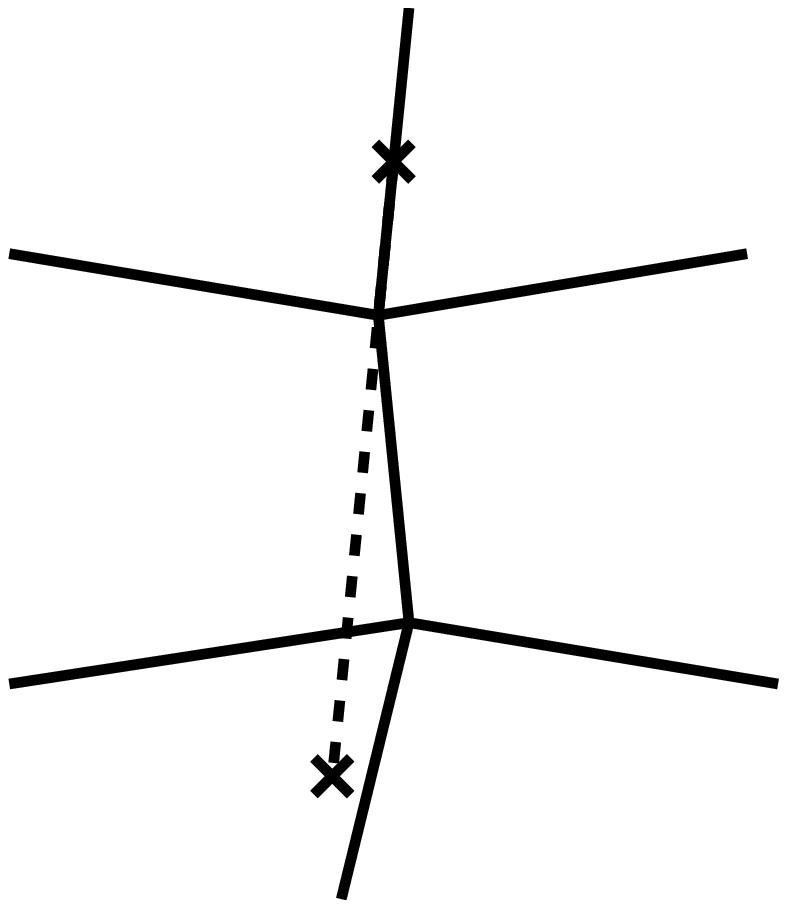}}
\subfloat[ ]{\includegraphics[height=1.6cm,width=1.2cm]{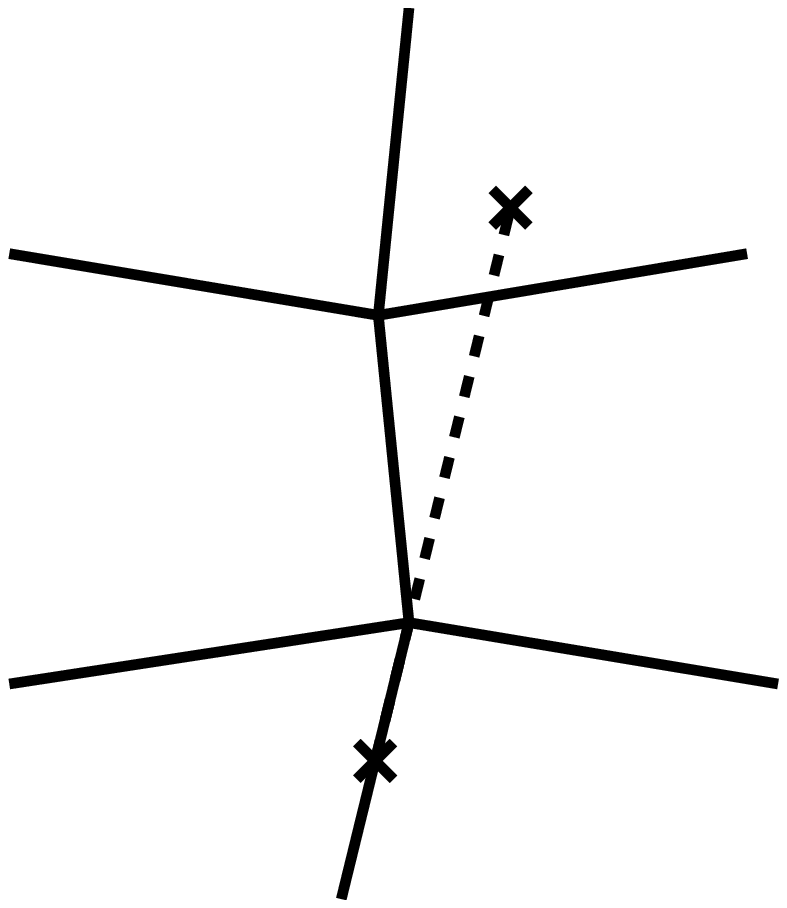}}
\subfloat[ ]{\includegraphics[height=1.6cm,width=1.2cm]{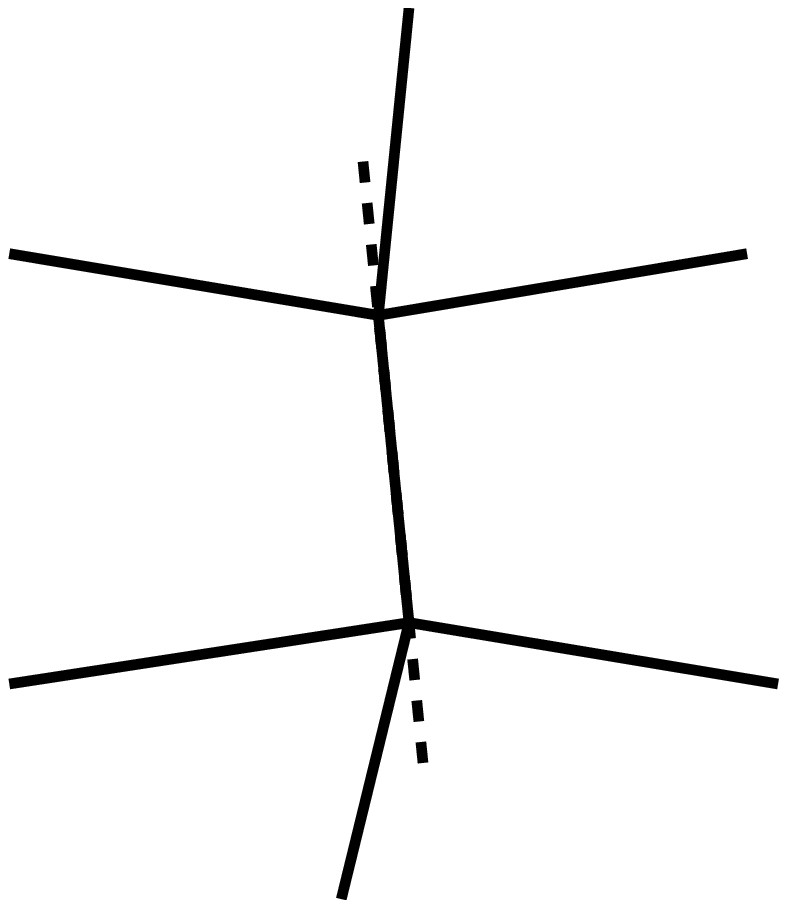}}
\subfloat[ ]{\includegraphics[height=1.6cm,width=1.2cm]{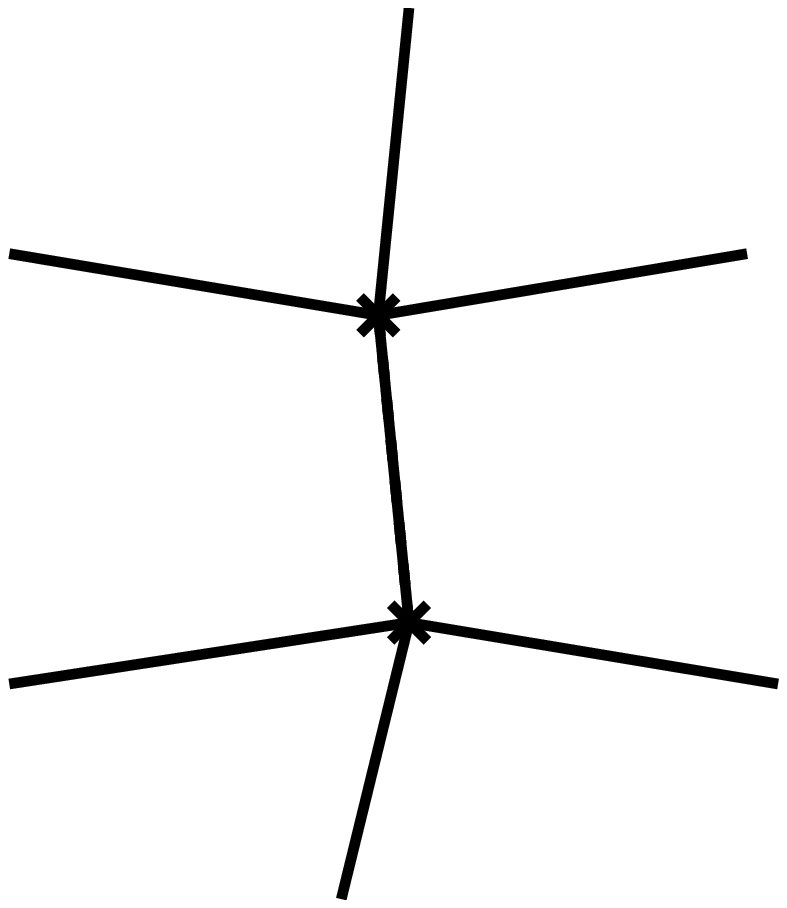}}
\subfloat[ ]{\includegraphics[height=1.6cm,width=1.2cm]{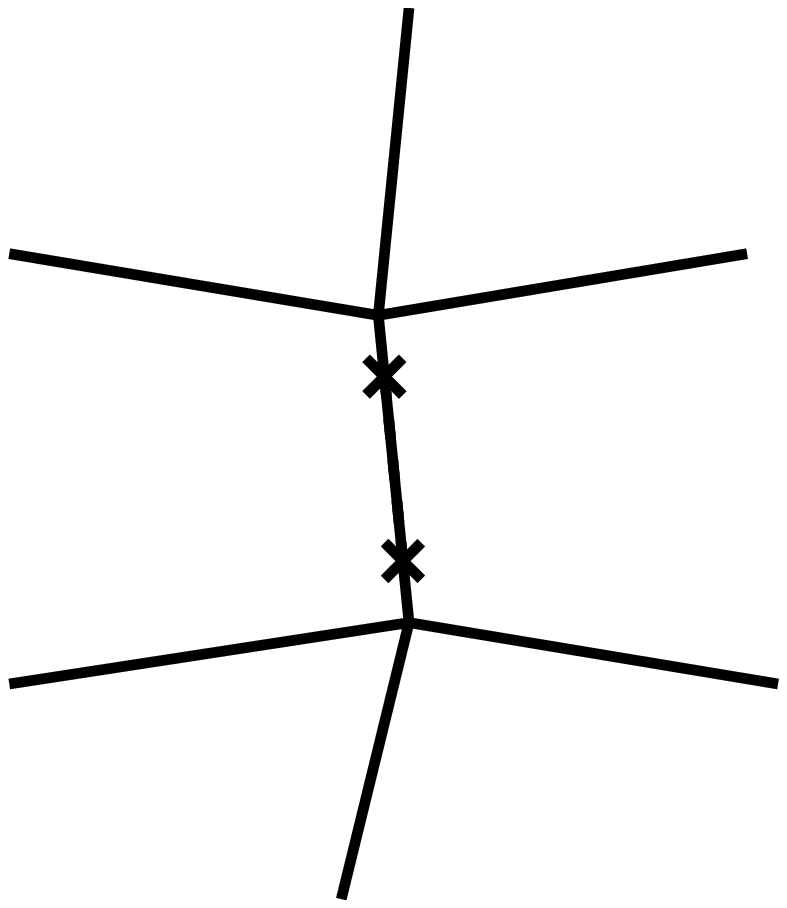}}
\caption{Some degeneracy of the intersections. The swept/fluxing area for (a)-(e) degenerates to be a triangle plus a segment, for (f),(g) and (h), the intersection degenerate to a line segment. While for the swap region, an overlap line segment of a new edge and the old edge can be viewed a degeneracy of a quadrilaterals. A overlap of the horizontal and vertical intersection point can be viewed as the degeneracy of a triangular.}
\end{figure}

The intersection of $T^b_{i,j}$ and $T^a_{k,s}$ for $(k,s) \in \mathcal{J}^a_{i,j}$ can be an polygon, an edge or even only a vertex. The degeneracy of the intersection was believed as one of the difficulty of the challenge of the CIB/DC method\cite[p.273]{ms03}. In fact, some cases can be handled by the Green formula to calculate the planar polygon with line integrals. Suppose the vertices of a polygon are arranged in counterclockwise, say, $P_1P_2,\ldots, P_s$, then it area can be calculated as
\begin{linenomath}
\begin{align}
\iint dxdy &= {\ointctrclockwise}_{\overrightarrow {P_1P_2}+\overrightarrow {P_2P_3}+ \cdots +\overrightarrow {P_sP_1}} xdy   \nonumber \\
& =\sum_{k=2}^{s+1} \frac{(y_k-y_{k-1})(x_k+x_{k-1})}{2}=
\frac{1}{2} \sum_{k=1}^{s-1}(x_ky_{k+1}-y_kx_{k+1}),  \label{eq:darea}
\end{align}
\end{linenomath}
where $P_{s+1}=P_1$. This is due to the fact
\begin{linenomath}
$$
\int_{\overrightarrow{P_1P_2}} x dy = \int_{x_1}^{x_2} x  \frac{y_2 -y_1}{x_2-x_1} dx = \frac{(y_2-y_1)(x_1+x_2)}{2}.
$$
\end{linenomath}
The formula \eqref{eq:darea} can handle any polygon including the degenerate cases: a polygon with $s+1$ vertices degenerate to one with $s$ vertices, a triangle degenerates to a vertex or a quadrilateral polygonal degenerates to a line segment. Such degenerated cases arise when one or two of the vertices of the the face $F^b_{i,j+\frac{1}{2}}$ lie on the local frame $\LF^a_{i,j+\frac{1}{2}}$, or the horizontal intersection points overlaps with the vertical intersection points. The later cases bring no difficulty. What really brings difficulties are the cases when $Q_{ij}$ locates in the vertical lines in the local frame or the faced overlaps with partial of the vertical lines in the local frame.. To identify such degeneracies, one only need more flags to identify whether such cases happens when calculating the intersecting points between face $F^b_{i,j+\frac{1}{2}}$ and $F^a_{i,j\pm\frac{1}{2}}$ and $F^a_{i,j+\frac{3}{2}}$.

\subsection{Assign a new vertex to an old cell}
 As shown in Fig.~\ref{fig:branch}, the relative position of a new vertex in an old local frame is the basis to classify all the intersection possibilities. This can also be obtained by the Green formula for the singed area of a polygon. We shall denote the signed area of $P_{i,j}P_{i,j+1}Q_{i,j}$ as $A_1$, $P_{i,j}P_{i,j+1}Q_{i,j}$ as $A_2$, $P_{i,j-1}P_{i,j}Q_{i,j}$ as $A_3$ and $P_{i,j-1}P_{i,j}Q_{i,j}$ as $A_4$. Then the vertex $Q_{i,j}$ can be assigned according to Algorithm~\ref{alg:assign}. This determine the first two level of branches of the classification tree in Fig. \ref{fig:branch}.

\begin{algorithm}
\caption{Assign current new vertex to an old cell}
\label{alg:assign}
\begin{algorithmic}
\State{ Compute $A_1$ }
\If {$A_1\ge 0$}
    \State{ Compute $A_2$}
    \If { $A_2\leq 0$}
        \State{ \Return flag='RU'; }\Comment{$Q_{i,j} \in C_{i+\frac{1}{2},j+\frac{1}{2}}$}
    \Else
        \State{Compute $A_3$}
        \If {$A_3\geq 0$}
            \State{ \Return flag='LU';} \Comment{$Q_{i,j} \in C_{i-\frac{1}{2},j+\frac{1}{2}}$}
        \Else
             \State{ \Return flag='RD';} \Comment{$Q_{i,j} \in C_{i-\frac{1}{2},j-\frac{1}{2}}$}
        \EndIf
    \EndIf
\Else
    \State{ Compute $A_4$}
     \If { $A_4 \leq 0$}
        \State{ \Return flag='RD'; }\Comment{$Q_{i,j} \in C_{i+\frac{1}{2},j-\frac{1}{2}}$}
    \Else
        \State{Compute $A_3$}
        \If {$A_3<0$}
            \State{ \Return flag='LD';} \Comment{$Q_{i,j} \in C_{i-\frac{1}{2},j-\frac{1}{2}}$}
        \Else
             \State{ \Return flag='LU';} \Comment{$Q_{i,j} \in C_{i-\frac{1}{2},j+\frac{1}{2}}$}
        \EndIf
    \EndIf
\EndIf
\end{algorithmic}
\end{algorithm}

\begin{figure}[!t]
\centering
\subfloat[]{\includegraphics[width=0.32\textwidth, height=0.32\textwidth]{swapping_03}}
\subfloat[\label{fig:VS}]{\includegraphics[width=0.32\textwidth, height=0.32\textwidth]{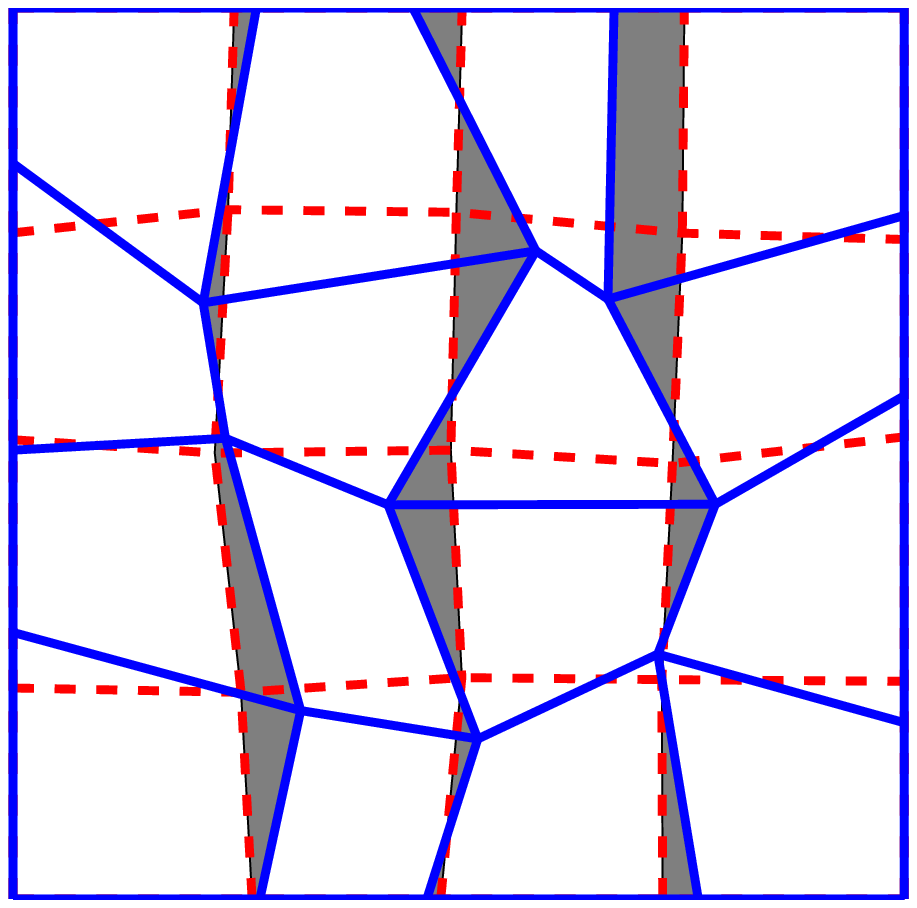}}
\subfloat[\label{fig:HS}]{\includegraphics[width=0.32\textwidth, height=0.32\textwidth]{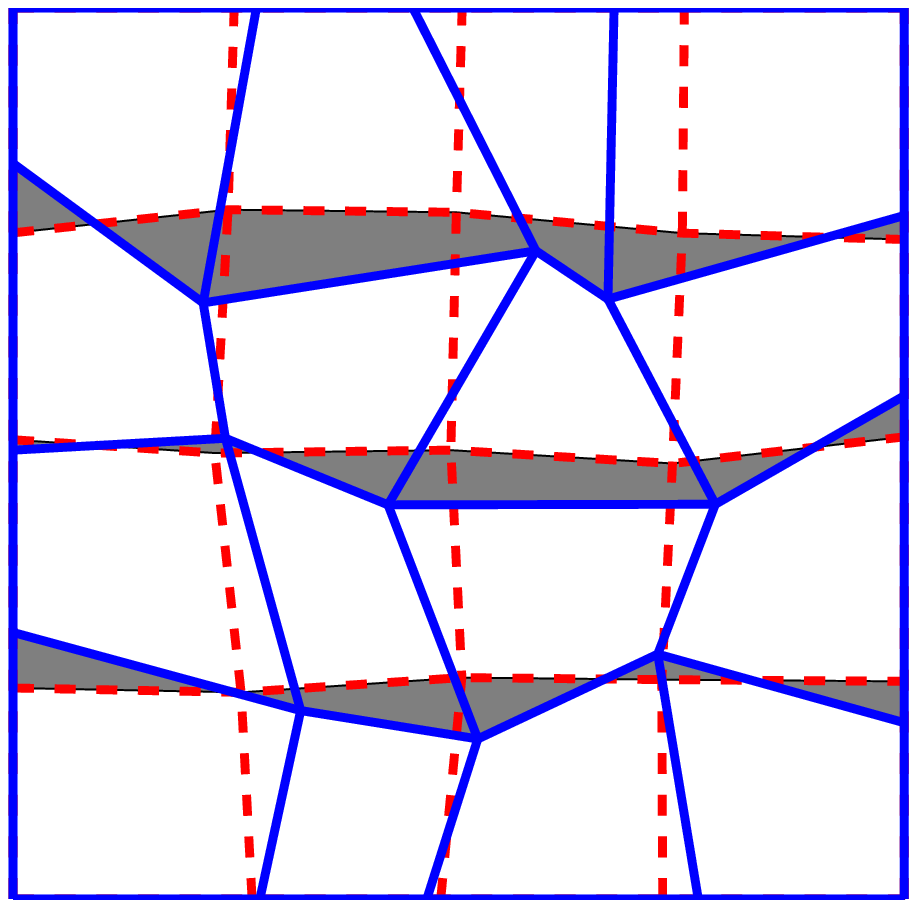}}
\caption{Illustration for the vertical and horizontal swap and swept regions}
\end{figure}
\subsection{Alternative direction sweeping}
To calculate all intersections in the swap area and swept/fluxing area, we can apply the alternative direction idea: view the union of the swap area of two admissible quadrilateral mesh of a domain as the union of (logically) vertical strips (shadowed region in Fig.~\ref{fig:VS}) and horizonal strips (shadowed area in Fig.~\ref{fig:HS}). One can alternatively sweep the vertical and horizontal swap strips. Notice that the horizonal strips can be viewed as the vertical strip by exchanging the x-coordinates and the y-coordinates. Therefore, one can only program the vertical sweep case. Each sweep calculate the intersections in the vertical/horizontal strips chunk by chunk. For the CIB/DC method, each chunk is a local swap region, while the FB/DC method, each chunk is a fluxing/swept area.

For the CIB/DC method, one can avoid to repeat calculating the corner contribution by thinning the second sweeping.   The first vertical sweep calculate all the intersection areas in the swap region.  The second sweep only calculates the the swap region due to the intersection of $T^b_{i,j} \cap T^a_{i,j\pm1}$. In contrast, in the FB/DC methods, the two sweep is totaly symmetric, repeat calculating the corner contribution is necessary.
\section{Applications}
Consider the following two kinds of grids.

\subsection{Tensor product grids}
The mesh on the unit square $[0,1] \times [0,1]$ is generated by the following function
\begin{linenomath}
\begin{equation}
\begin{cases}
x(\xi, \eta, t)=(1-\alpha(t)) \xi +\alpha(t) \xi^3, \\
y(\xi, \eta, t)=(1-\alpha(t)) \eta^2 \\
\alpha(t)= \sin(4\pi t)/2, \quad \xi,\eta, t \in[0,1].
\end{cases}
\label{eq:tensor}
\end{equation}
\end{linenomath}
This produce a  sequnce of tensor product grids $ {x_{i,j}^n}$ given by
\begin{linenomath}
\begin{equation}
x_{i,j}^n=x(\xi_i,\eta_i, t^n), y_{i,j}^n=y(\xi_i, \eta_j, t^n).
\end{equation}
\end{linenomath}
where $\xi_i $ and $\eta_j$ are $nx$ and $ny$ equally spaced points in $[0,1]$. For the old grid, $t_1=1/(320+nx)$, $t_2=2t_1$. We choose $nx=ny=11,21,31,41,\ldots,101$.
\begin{figure}
\centering
\subfloat[ Tensor grids ]{\includegraphics[width=0.45\textwidth]{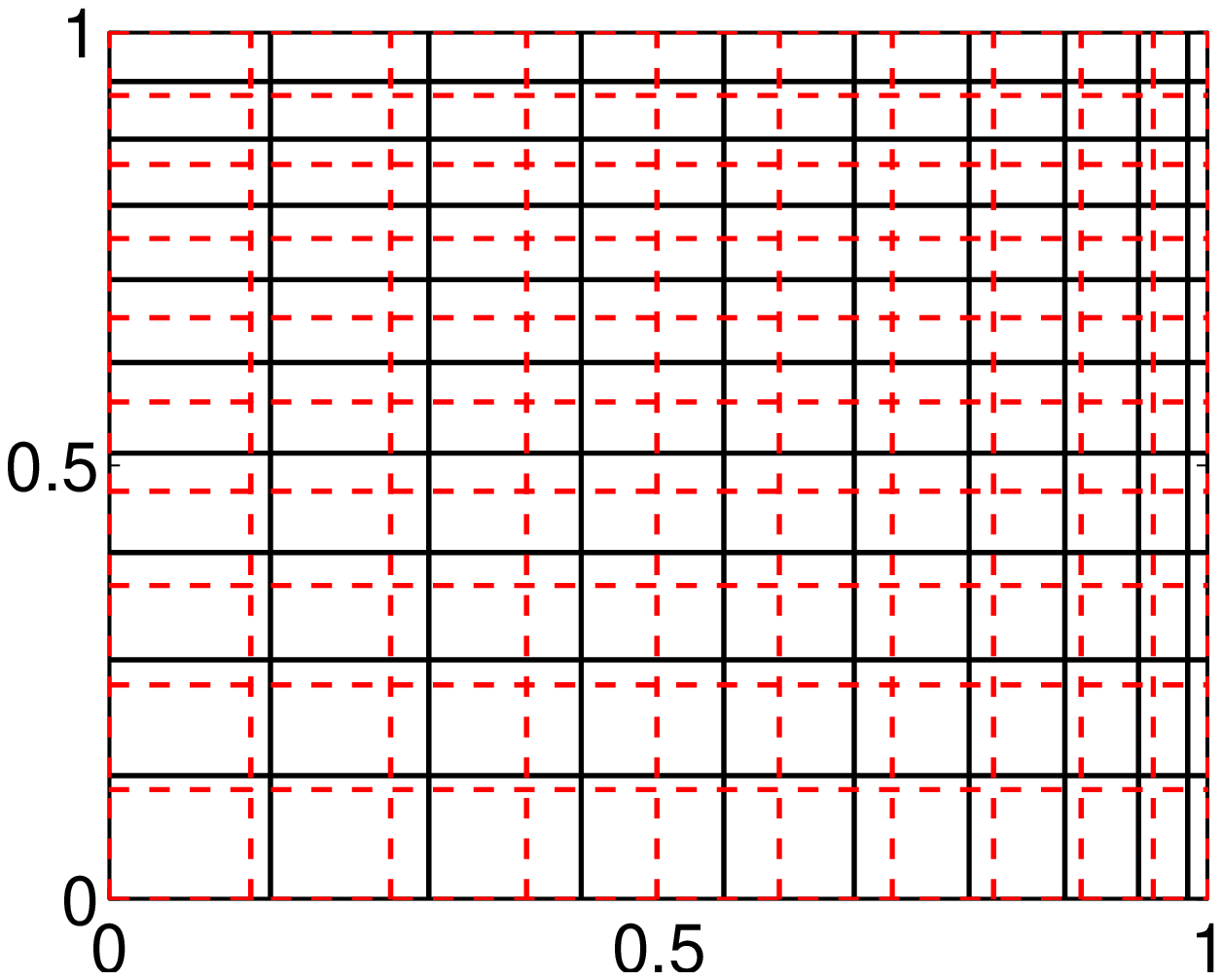}}
\subfloat[ Random grids ]{\includegraphics[width=0.45\textwidth]{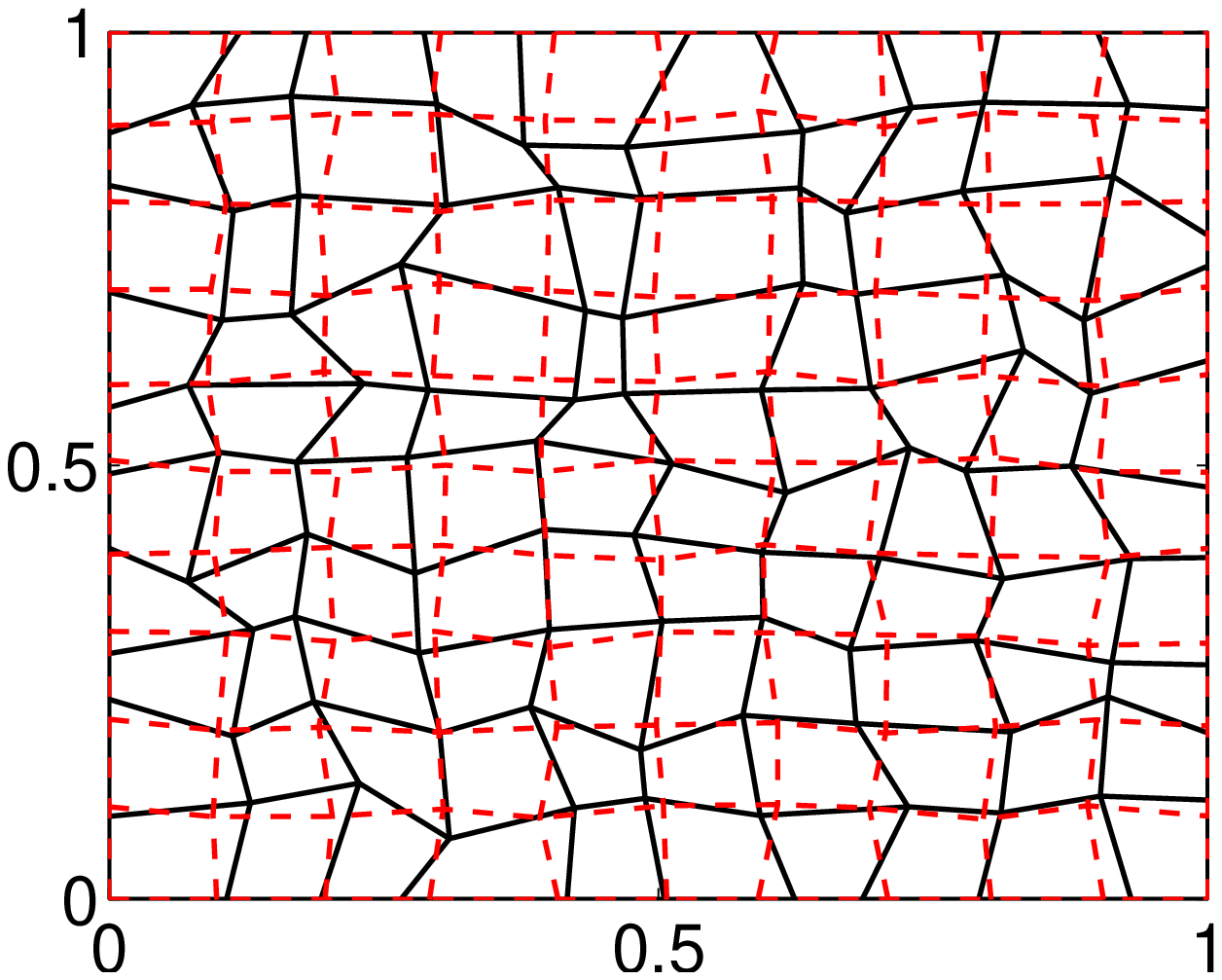}}
\caption{Illustration of the tensor grids and the random grids. }
\end{figure}
\subsection{Random grids}
A random grid is a perturbation of a uniform grid,
\begin{linenomath}
\begin{equation}
\begin{cases}
x_{ij}^n= \xi_i + \gamma r_i^n h, \\
y_{ij}^n= \eta_j+ \gamma r_j^n h.
\end{cases}
\end{equation}
\end{linenomath}
where $\xi_i$ and $\eta_j$ are constructed as that in the above tensor grids. $h=1/(nx-1)$. We use $\gamma=0.4$ as the old grid and $\gamma=0.1$ as the new grid, $nx=ny$.

\subsection{Testing functions}
\begin{figure}
\centering
\subfloat[Franke function \label{fig:franke} ]{\includegraphics[width=0.32\textwidth]{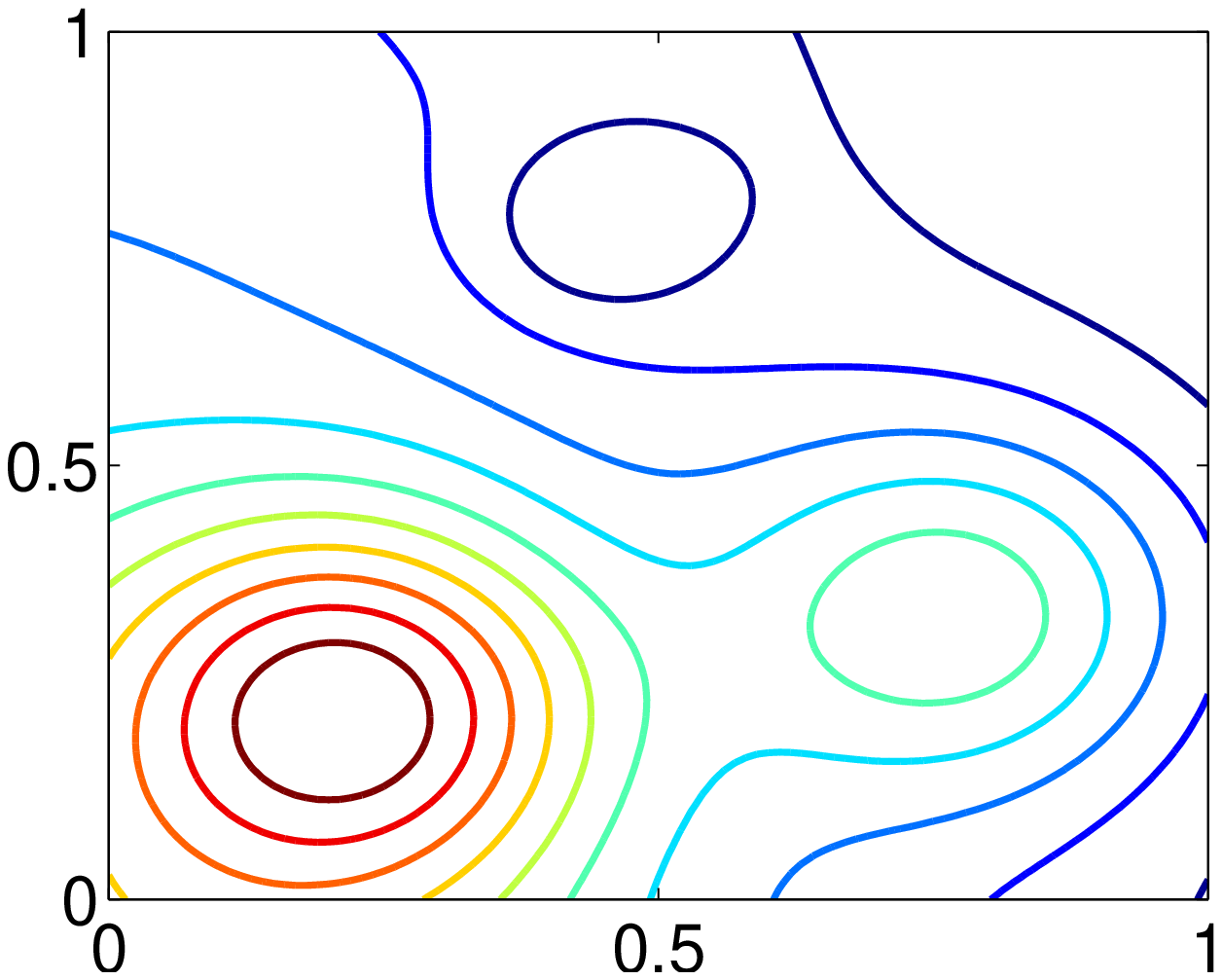}}
\subfloat[tanh function \label{fig:tanh} ]{\includegraphics[width=0.32\textwidth]{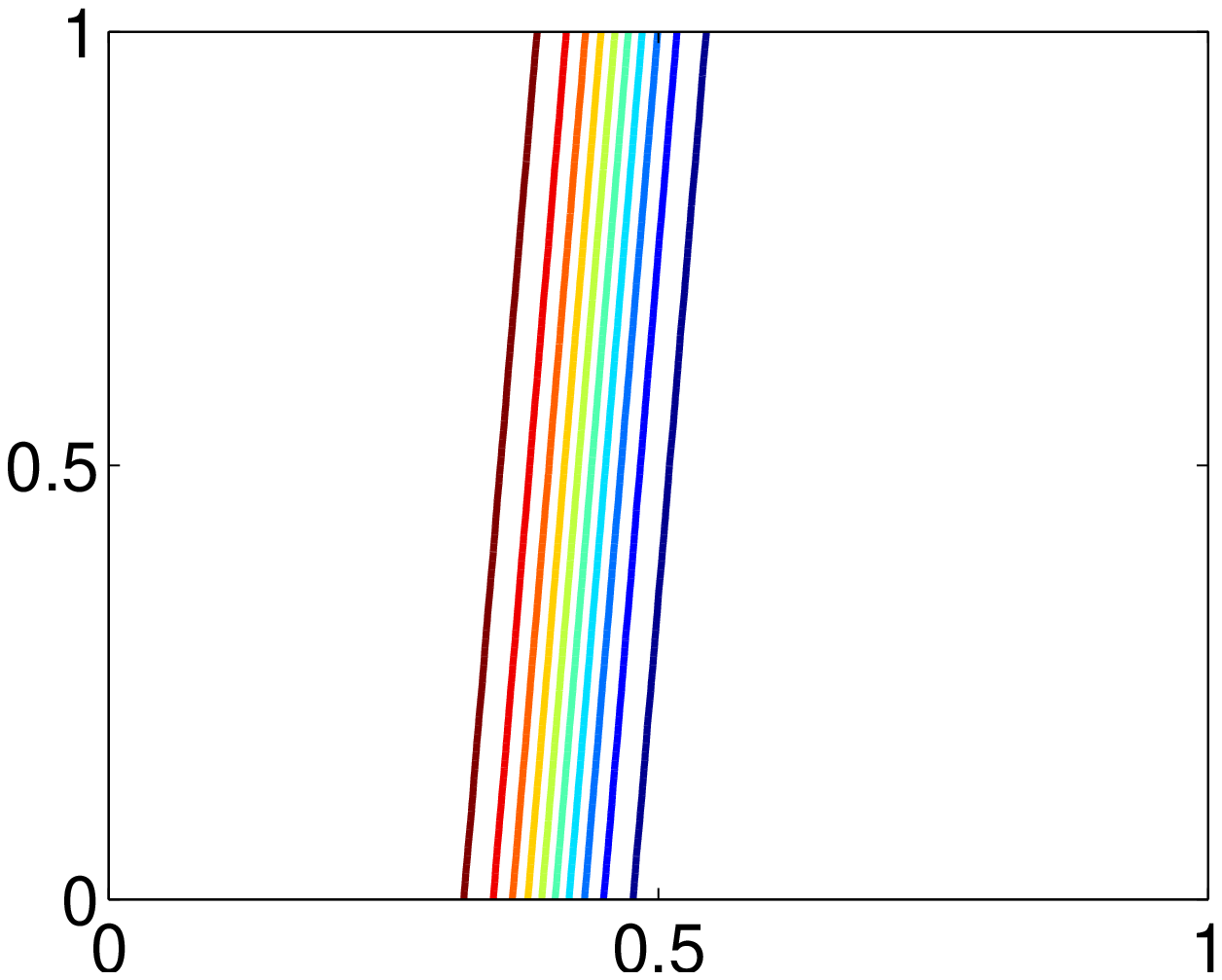}}
\subfloat[peak function \label{fig:peak} ]{\includegraphics[width=0.32\textwidth]{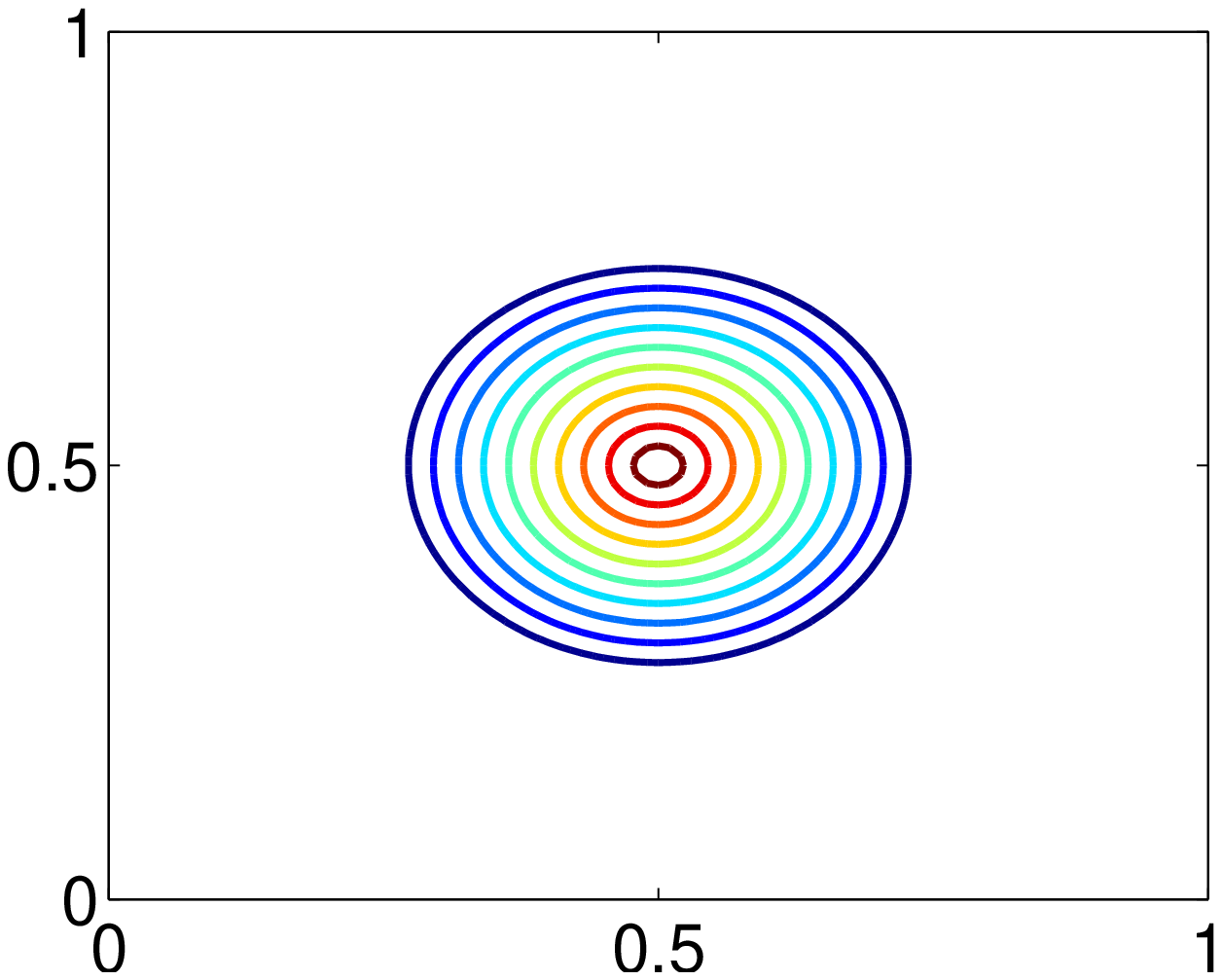}}
\caption{Contour lines of the test functions on  $101 \times 101$ equally spaced mesh in the unite square. }
\label{fig:test}
\end{figure}
We use three examples as the density function, the first one is the \texttt{franke} function in Matlab.
\begin{linenomath}
\begin{align}
\rho_1(x,y)=franke(x,y)= &.75\exp(-((9x-2)^2 + (9y-2)^2)/4) + \nonumber \\& .75\exp(-((9x+1).^2)/49 - (9y+1)/10) +  \nonumber  \\
             &.5\exp(-((9x-7)^2 + (9y-3)^2)/4) -  \nonumber \\& .2\exp(-(9x-4)^2 - (9y-7)^2).
\end{align}
\end{linenomath}
The second example is a continuous function with sharp gradient. It is a shock like function.
\begin{linenomath}
\begin{equation}
\rho_2(x,y)=\tanh(y-15x+6)+1.2.
\end{equation}
\end{linenomath}
The third example is the peak function used in \cite{margolin03}
\begin{linenomath}
\begin{equation}
\rho_3(x,y)=
\begin{cases}
0, & \sqrt{(x-0.5)^2+(y-0.5)^2} >0.25; \\
\max\{ 0.001, 4(0.25-r) \}, &  \sqrt{(x-0.5)^2+(y-0.5)^2} \leq 0.25.
\end{cases}
\end{equation}
\end{linenomath}
Fig.~\ref{fig:test} illustrates the contour lines of the test functions. The initial mass of on the old cell are calculated by a fourth order quadrature, the remapped density function is calculated by the exact FB/DC method: the swept region is calculated exactly. The density are assumed to be a piecewise constant on each old cell. Since the swept/flux area are calculated exactly. The remapped error only depends on the approximation scheme to the density function on the old cell. The $L_\infty$ norm
\begin{linenomath}
\begin{equation}
\|\rho^{*} -\rho \|_{\infty} =\max_{i_j}| \rho^{h}_{i+\frac{1}{2},j+\frac{1}{2}}- \rho(x_{i+\frac{1}{2},j+\frac{1}{2}}) |
\end{equation}
\end{linenomath}
is expected in the order of $\mathcal{O}(h)$ for piecewise constant approximation to the density function on the old mesh. While the $L_1$ norm
\begin{linenomath}
$$
\|m^*-m \|_{\infty}=\max_{i,j} | (\rho^{h}_{i+\frac{1}{2},j+\frac{1}{2}}- \rho(x_{i+\frac{1}{2},j+\frac{1}{2}}))\mu(C_{i+\frac{1}{2}, j+\frac{1}{2}}) |.
$$
\end{linenomath}
is expected to be in the order of $\mathcal{O}(h^3)$.  We don't use the $L_1$ normal like in other publications, because when plot the convergence curve in in the same figure, the $L_1$ norm and the $L_\infty$ norm for the density function converges also the same rate.
 Fig.~\ref{fig:error} demonstrates the convergence of the remmapping error based on the piecewise constant reconstruction of the density function in the old mesh.

\begin{figure}
\centering
\subfloat[ Tensor grids ]{\includegraphics[width=0.45\textwidth]{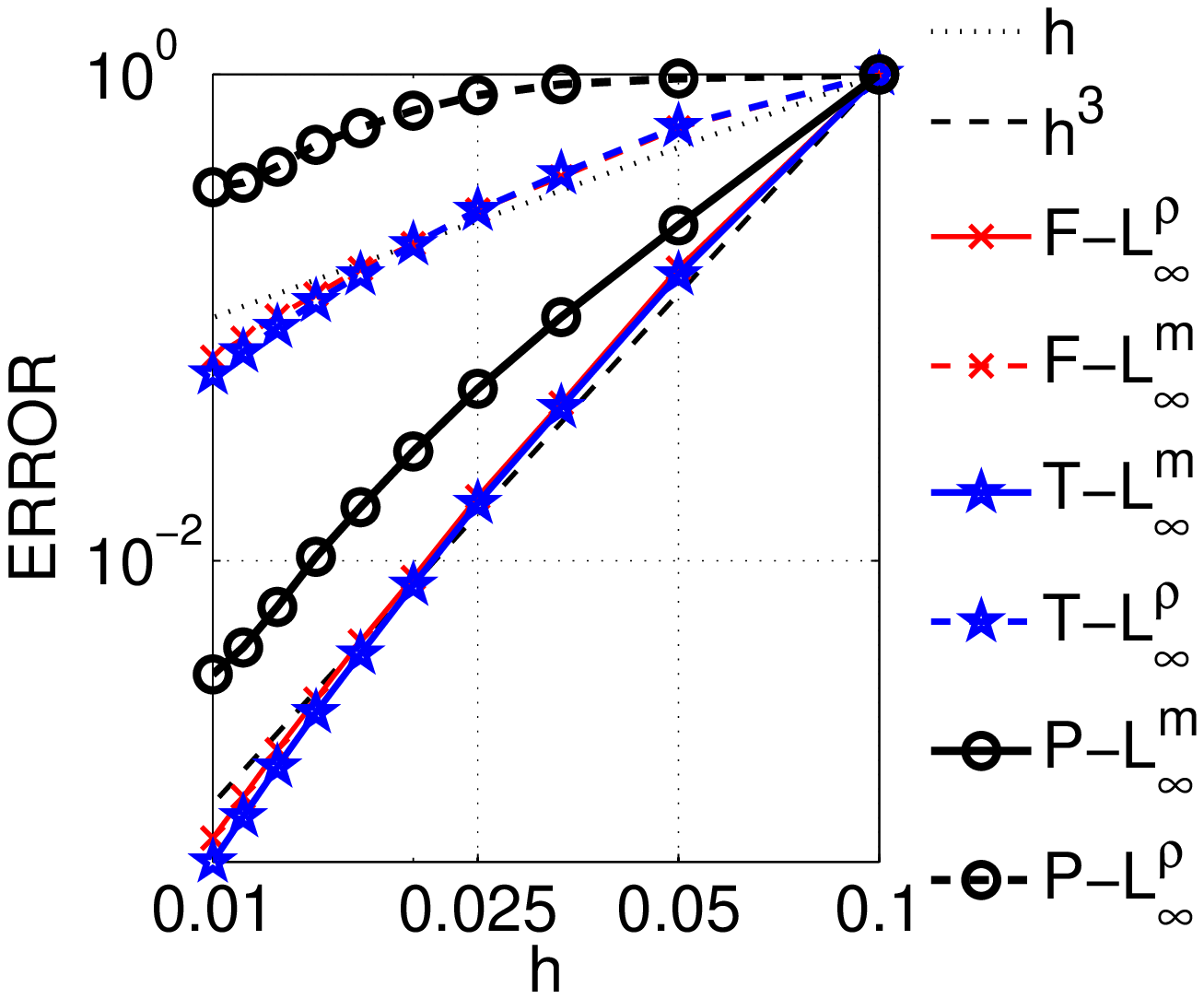}}
\subfloat[ Random grids ]{\includegraphics[width=0.45\textwidth]{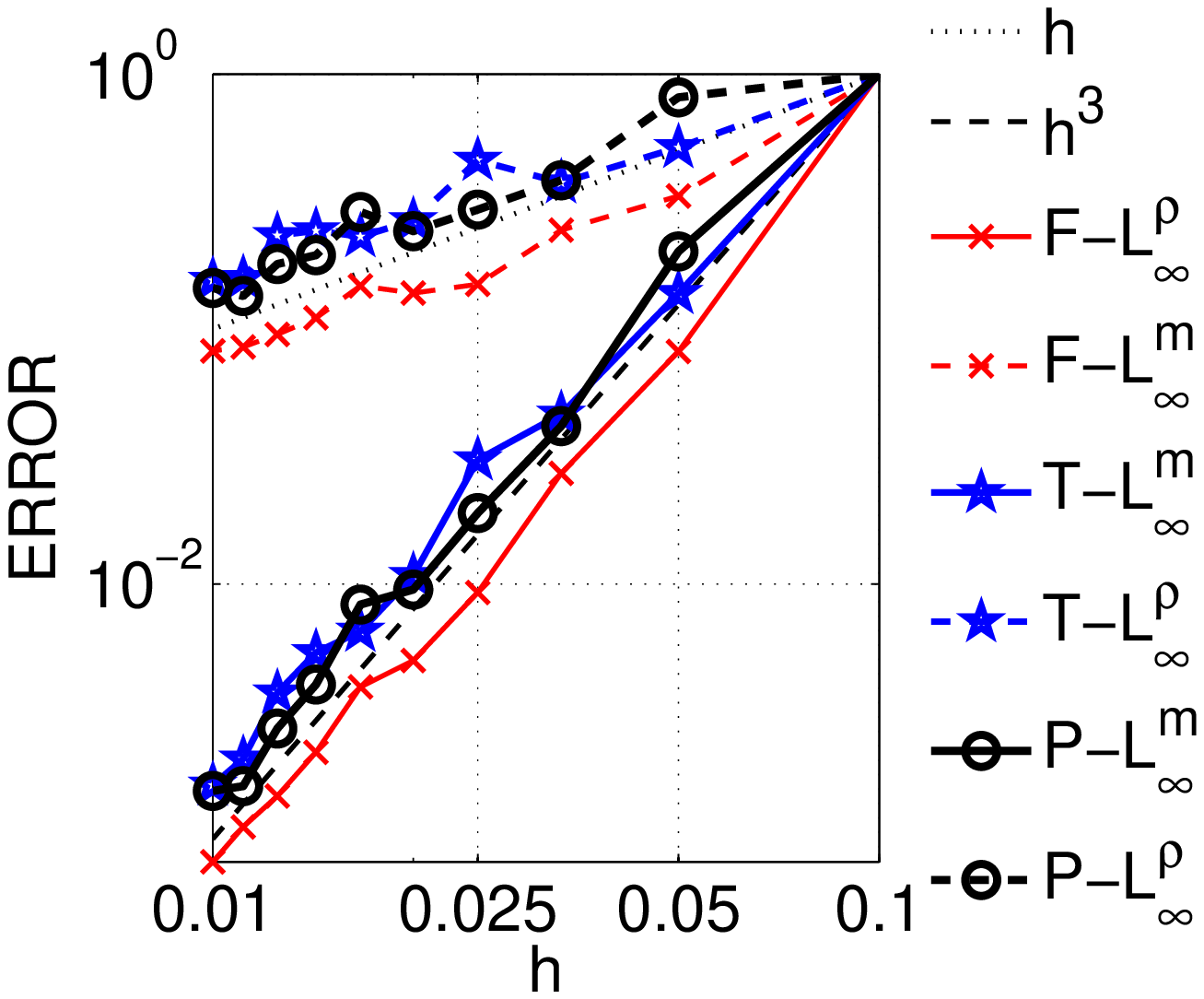}}
\caption{Convergence of the error between the remapped density functions and the true density functions.$F$: the Franke function, P: the peak function, and T: tanh function. The error are scaled by the level on the coarse level.}
\label{fig:error}
\end{figure}
\begin{figure}
\centering
\subfloat[ franke ]{\includegraphics[width=0.32\textwidth]{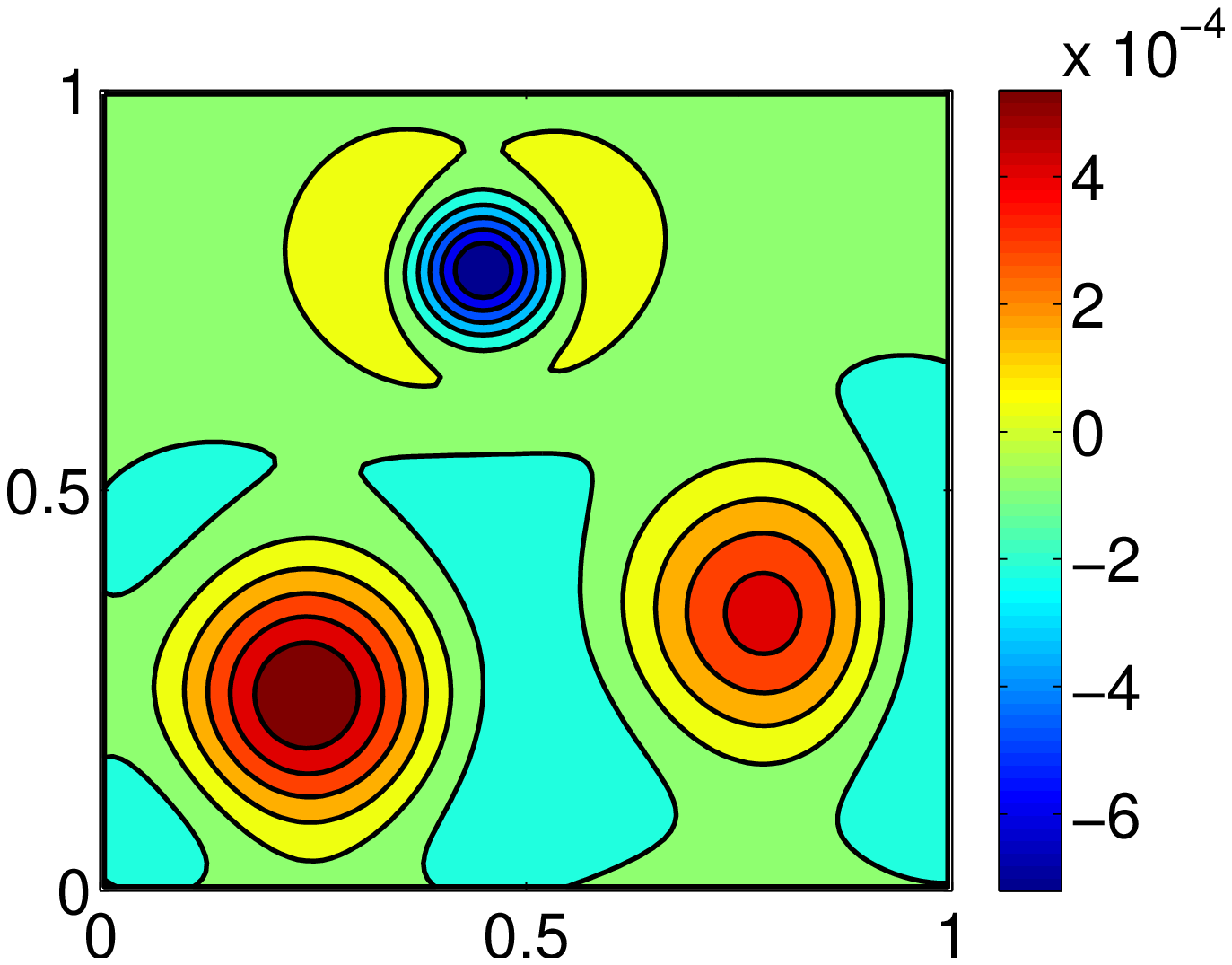}}
\subfloat[ tanh ]{\includegraphics[width=0.32\textwidth]{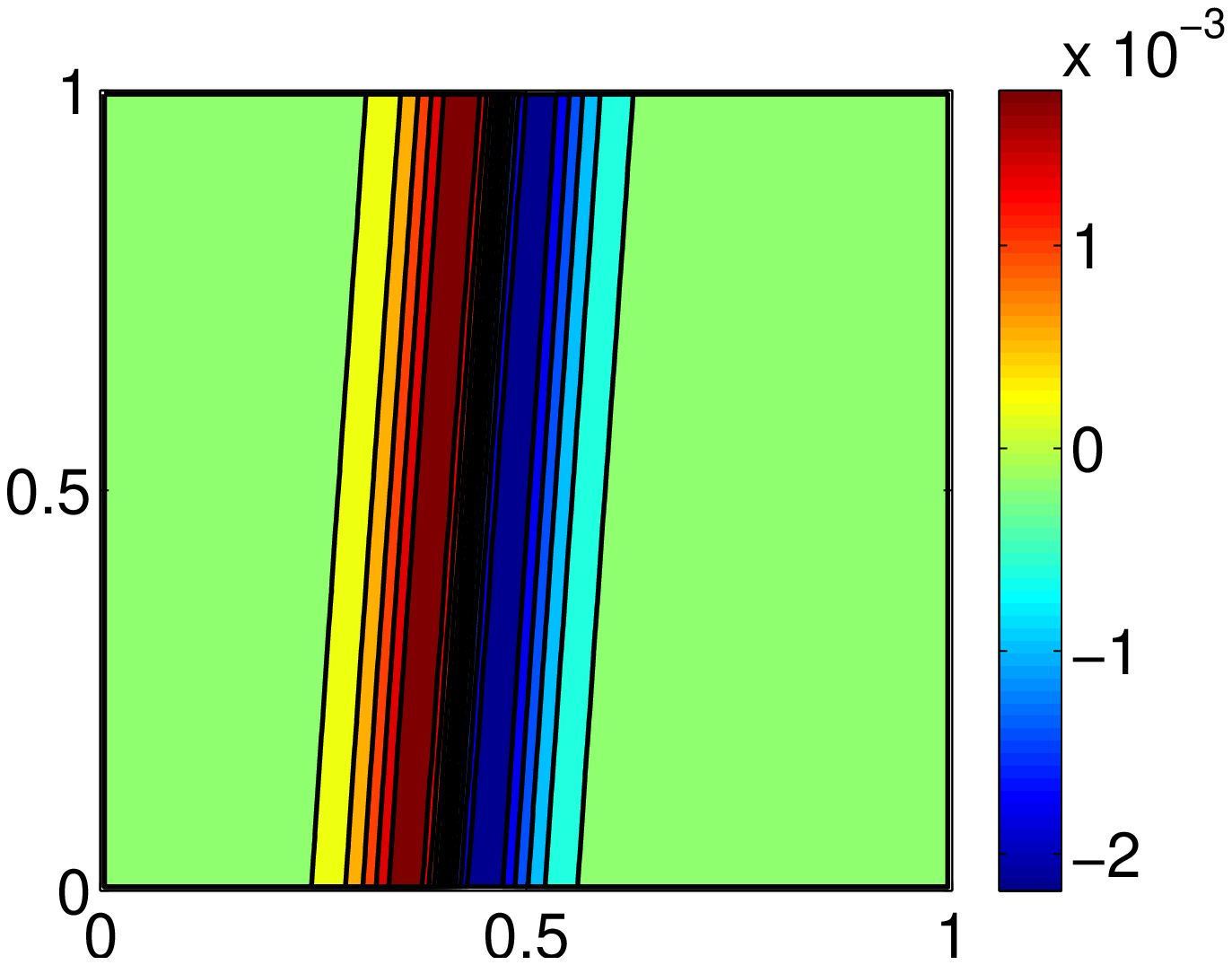}}
\subfloat[ peak ]{\includegraphics[width=0.32\textwidth]{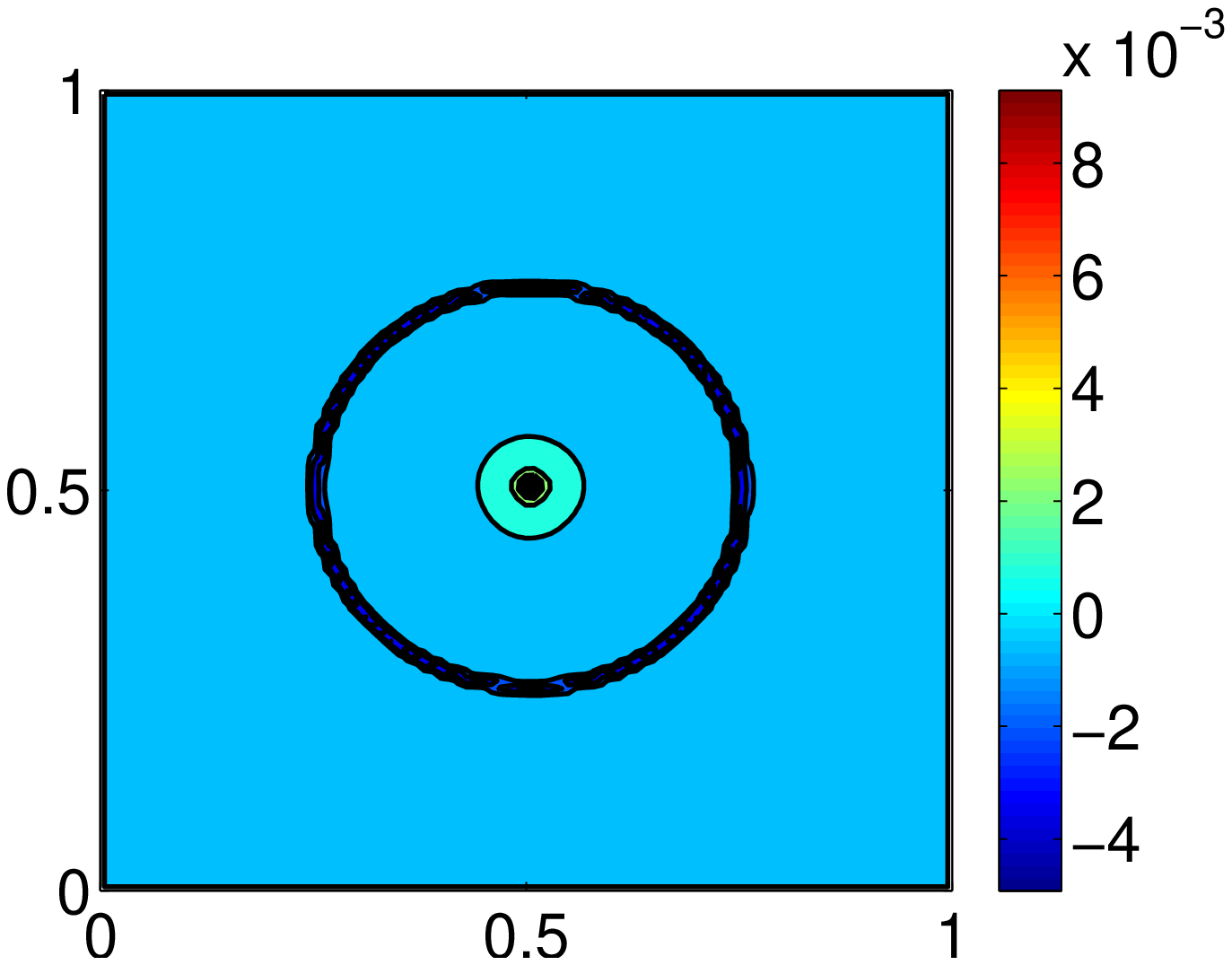}}
\caption{Remapped error of the density function on $101 \times 101$ tensor grids.}
\end{figure}

\begin{figure}
\centering
\subfloat[ franke ]{\includegraphics[width=0.31\textwidth]{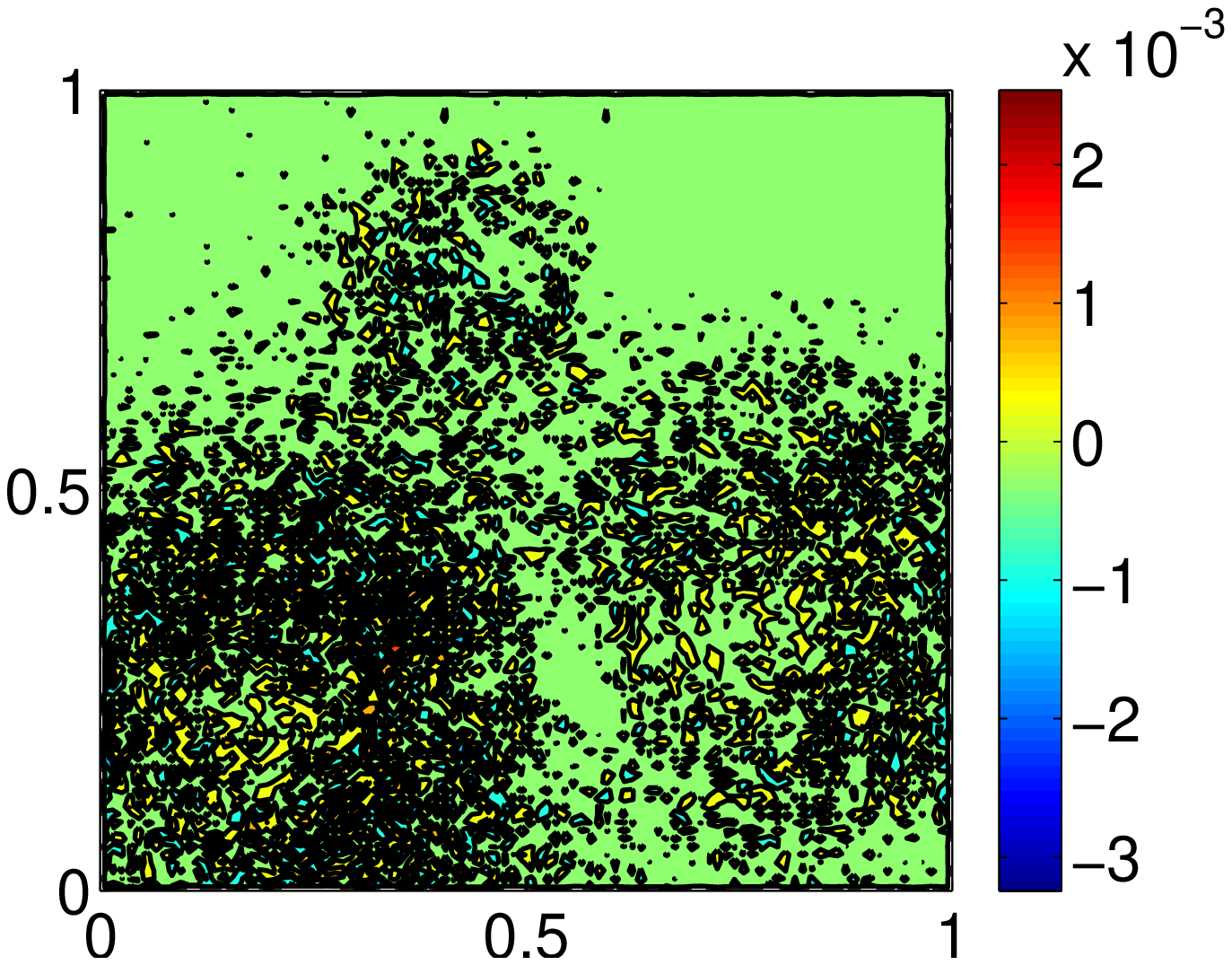}}
\subfloat[ tanh ]{\includegraphics[width=0.34\textwidth]{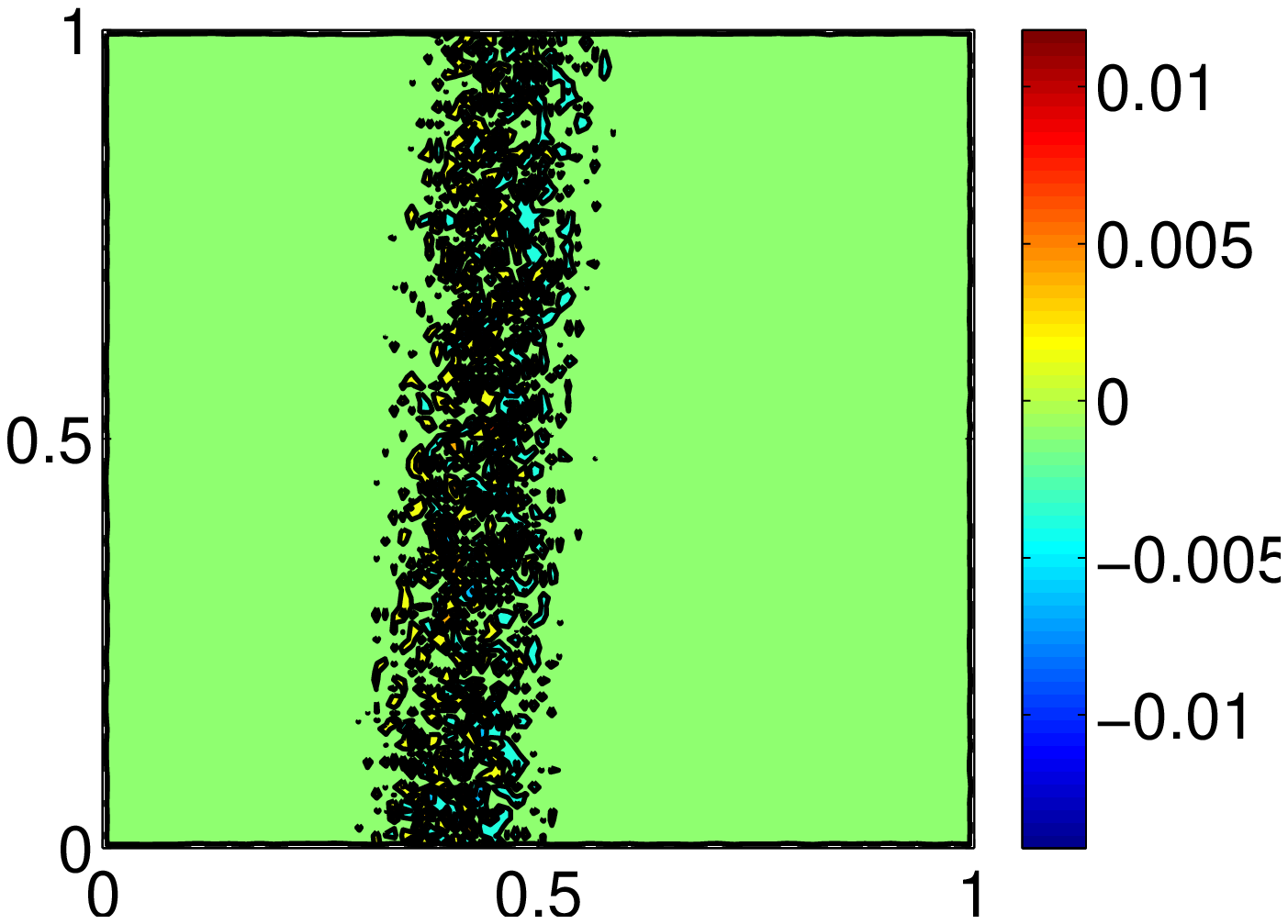}}
\subfloat[ peak ]{\includegraphics[width=0.31\textwidth]{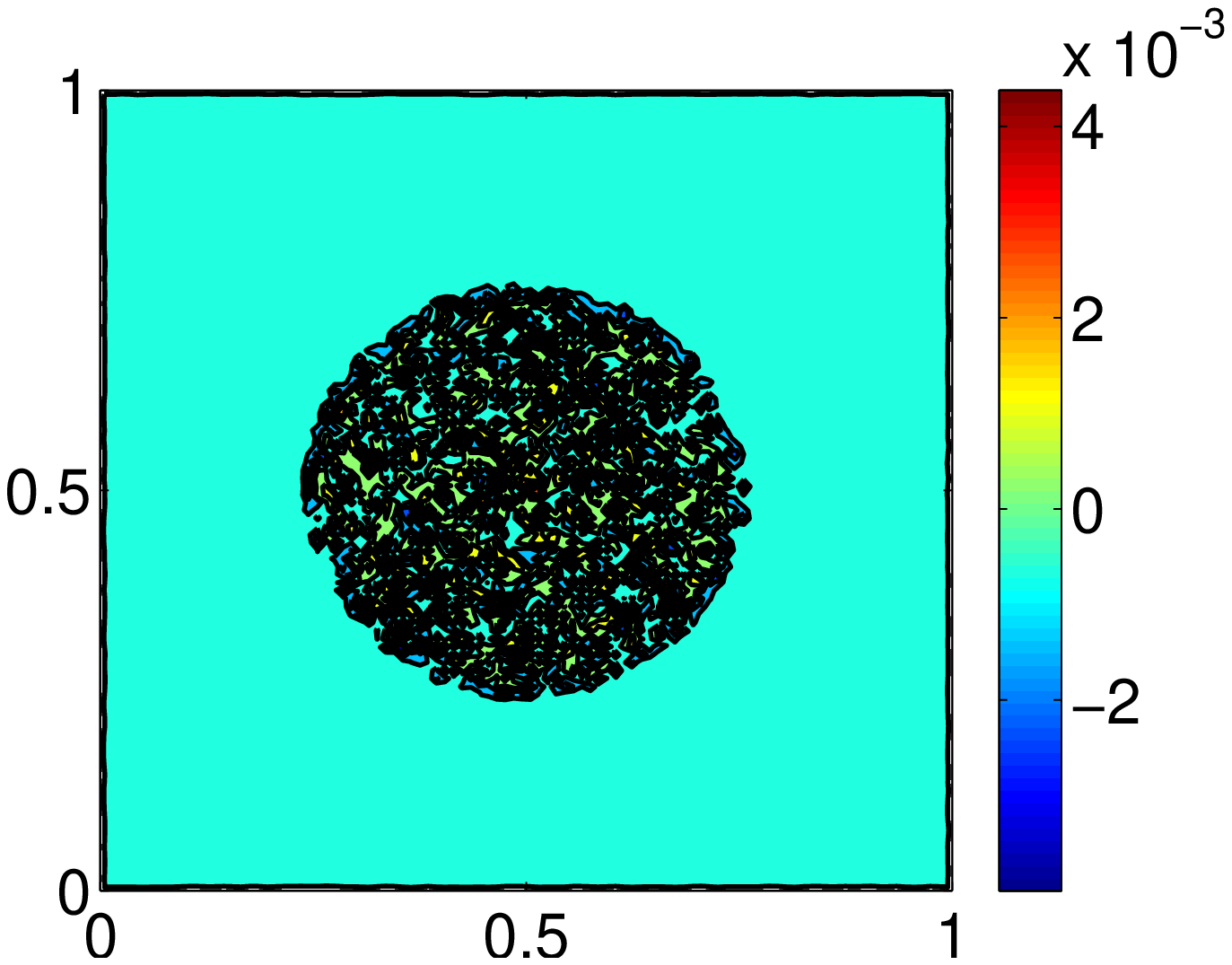}}
\caption{Remapped error of the density function on $101 \times 101$ random grinds.}
\end{figure}

\begin{figure}
\subfloat[ ]{\includegraphics[width=0.32\textwidth]{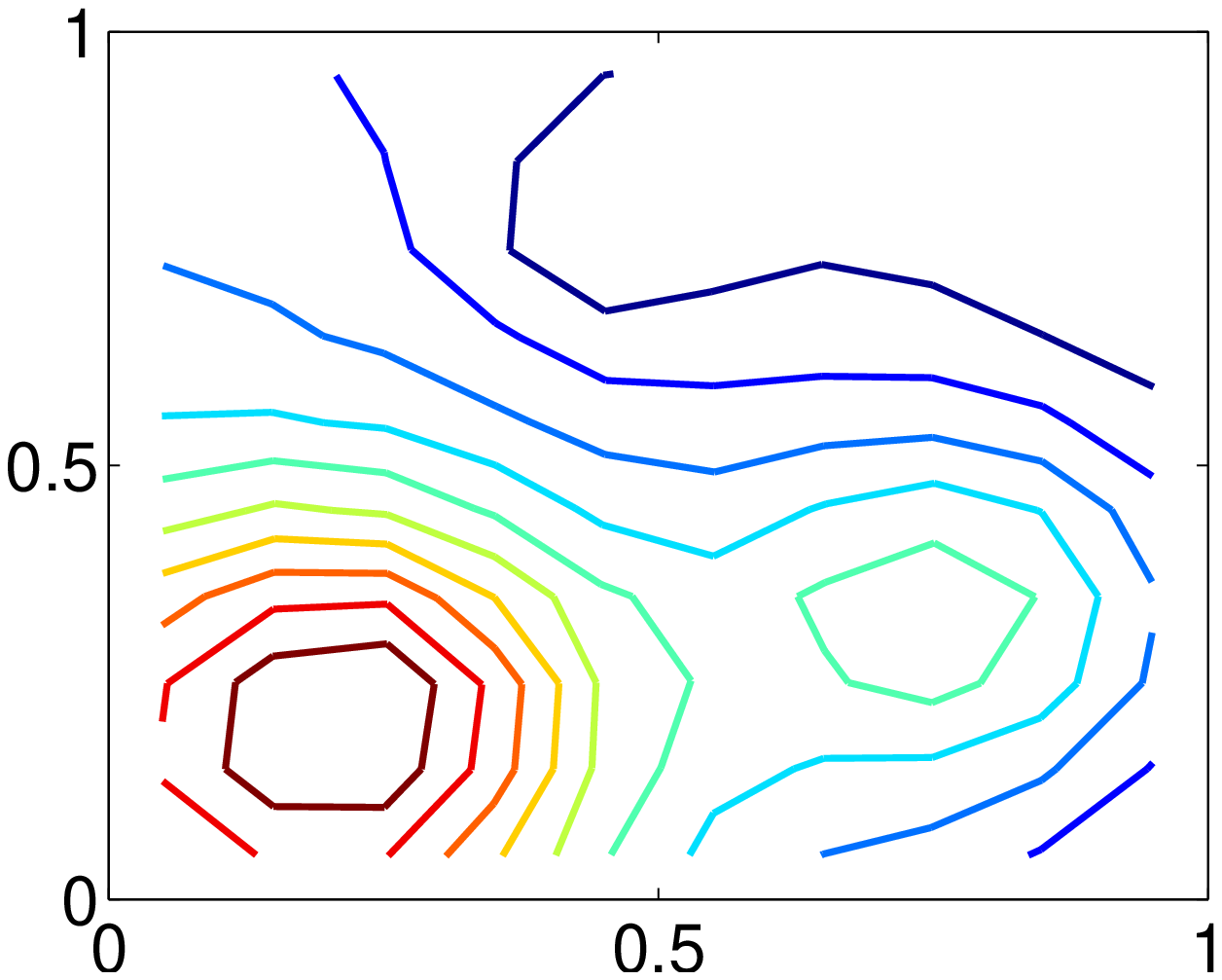}}
\subfloat[ ]{\includegraphics[width=0.32\textwidth]{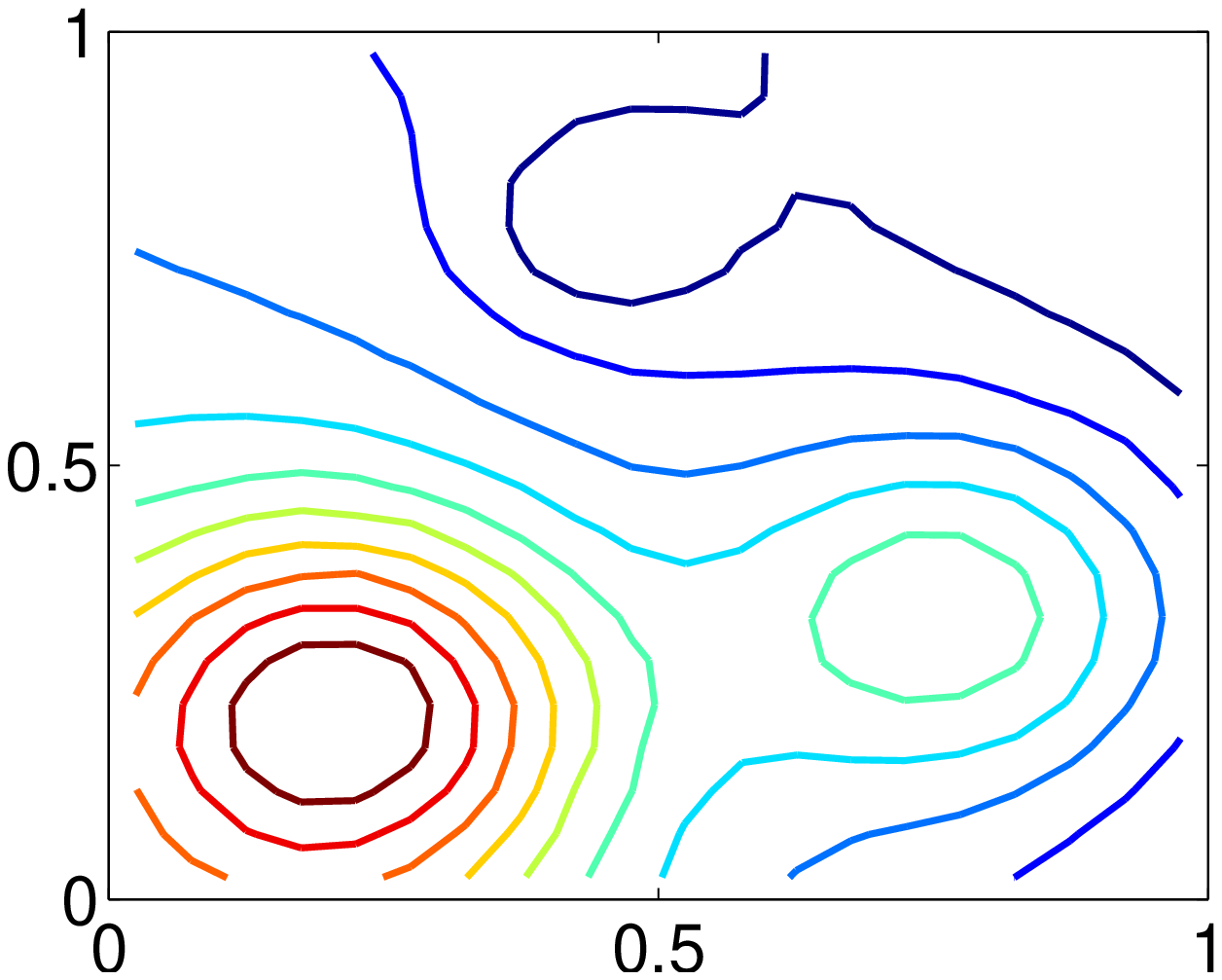}}
\subfloat[ ]{\includegraphics[width=0.32\textwidth]{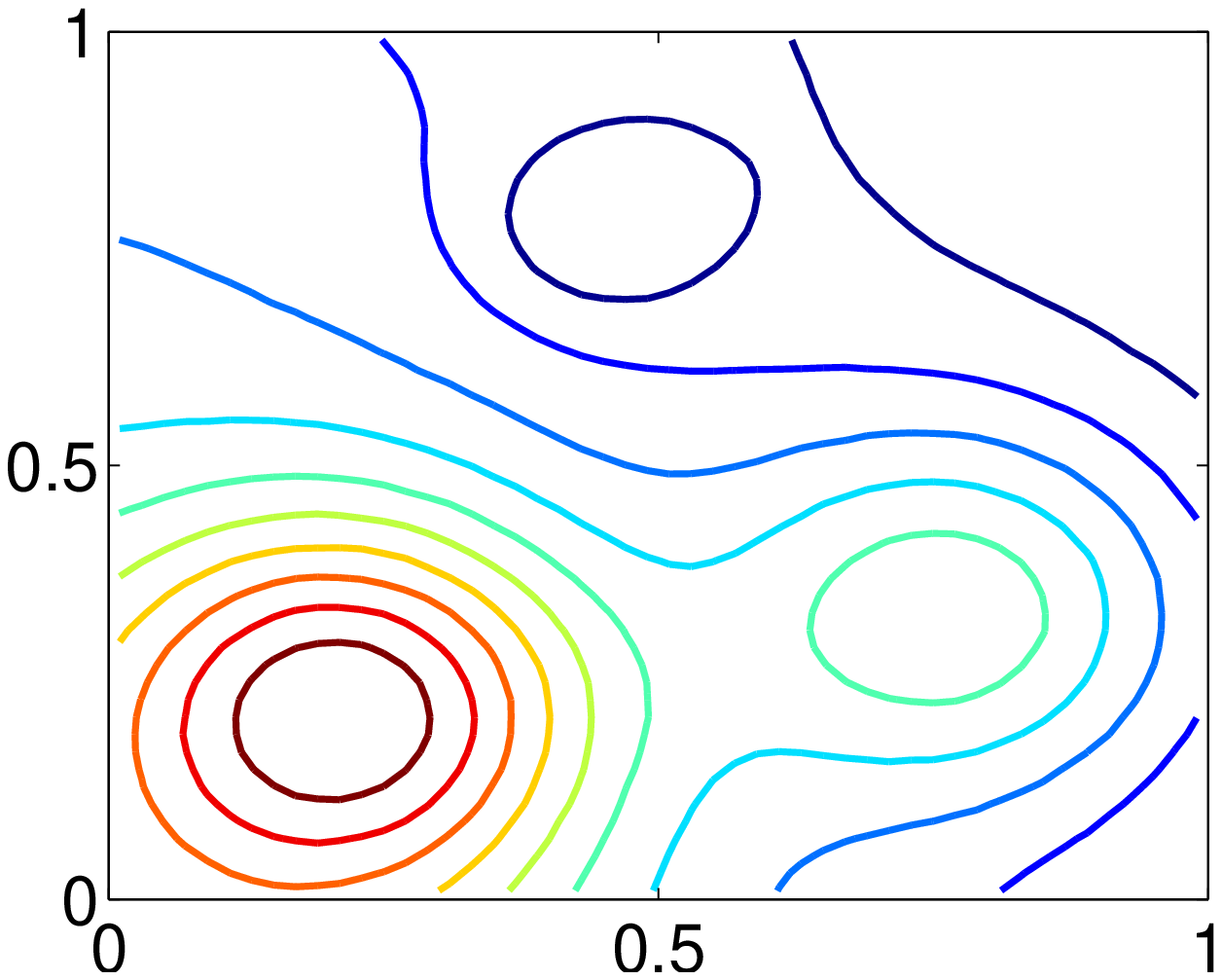}}

\subfloat[ ]{\includegraphics[width=0.32\textwidth]{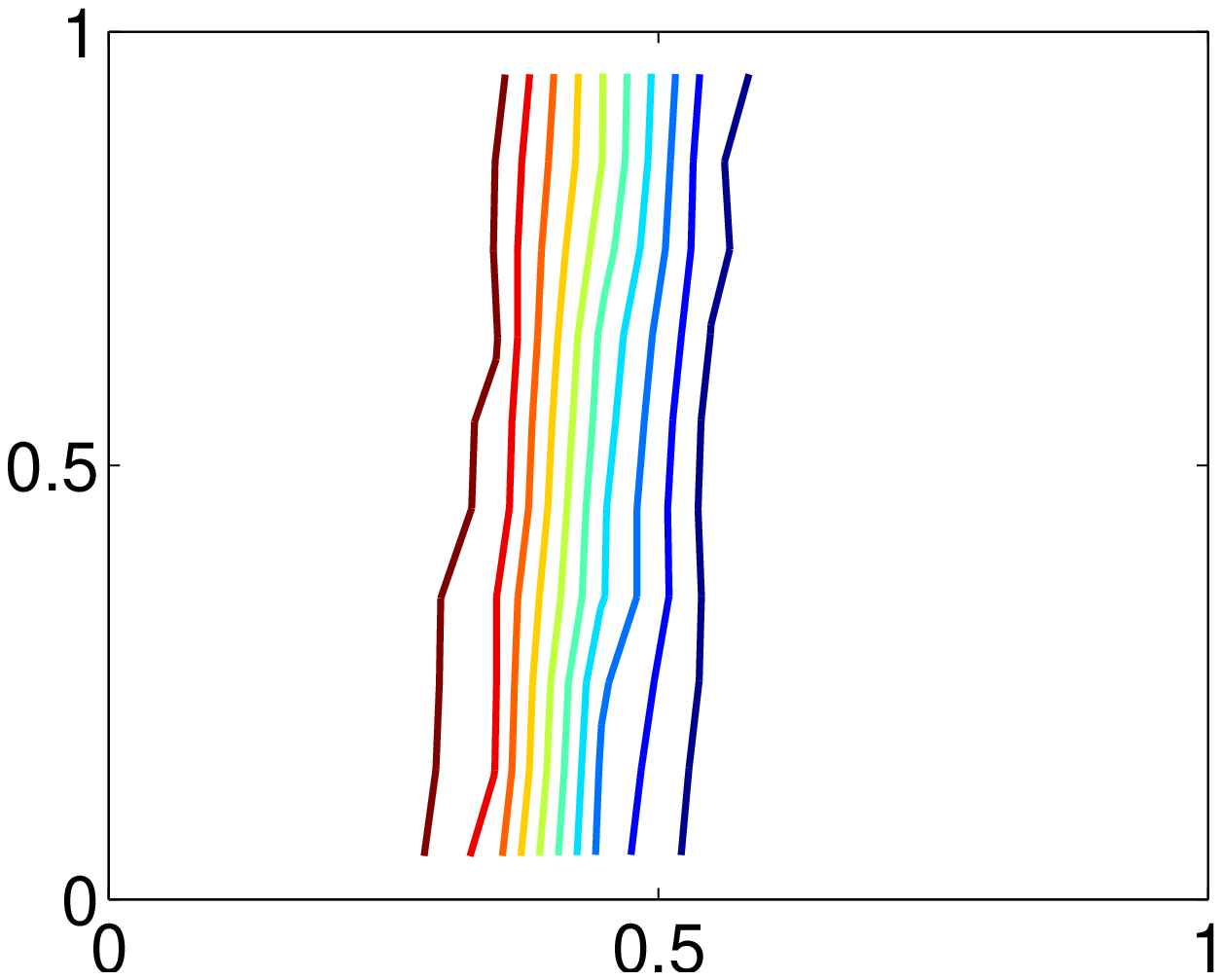}}
\subfloat[ ]{\includegraphics[width=0.32\textwidth]{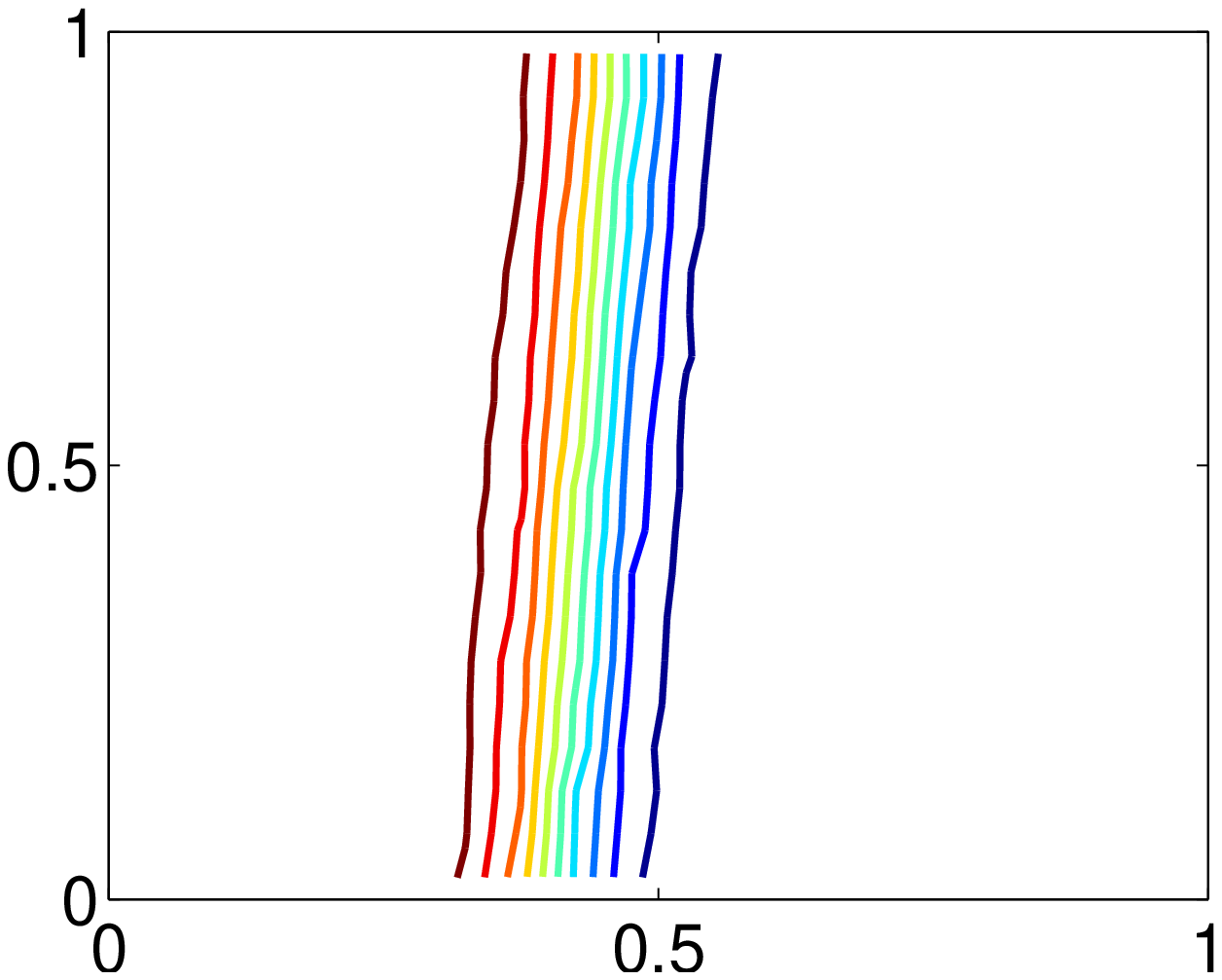}}
\subfloat[ ]{\includegraphics[width=0.32\textwidth]{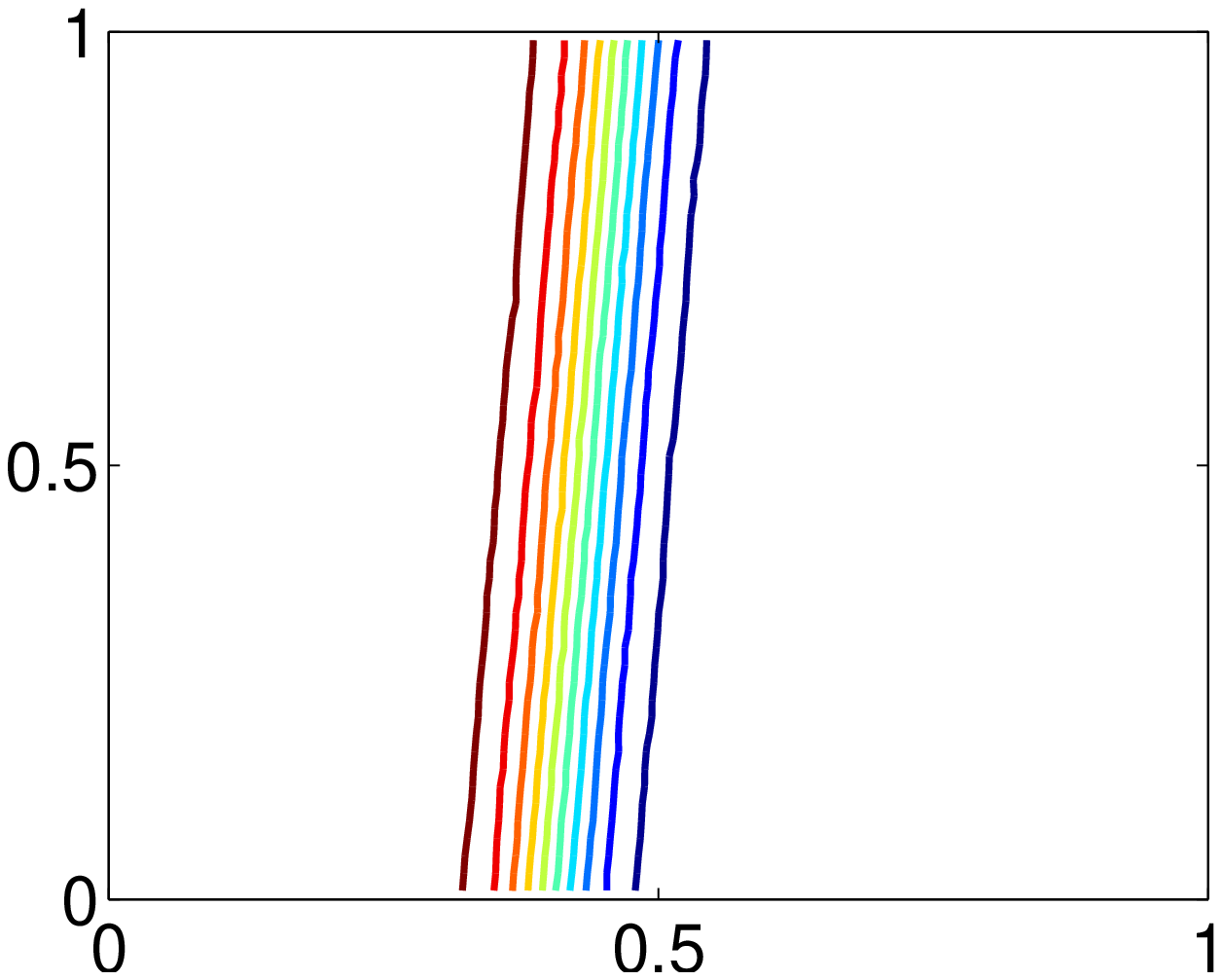}}

\subfloat[ ]{\includegraphics[width=0.32\textwidth]{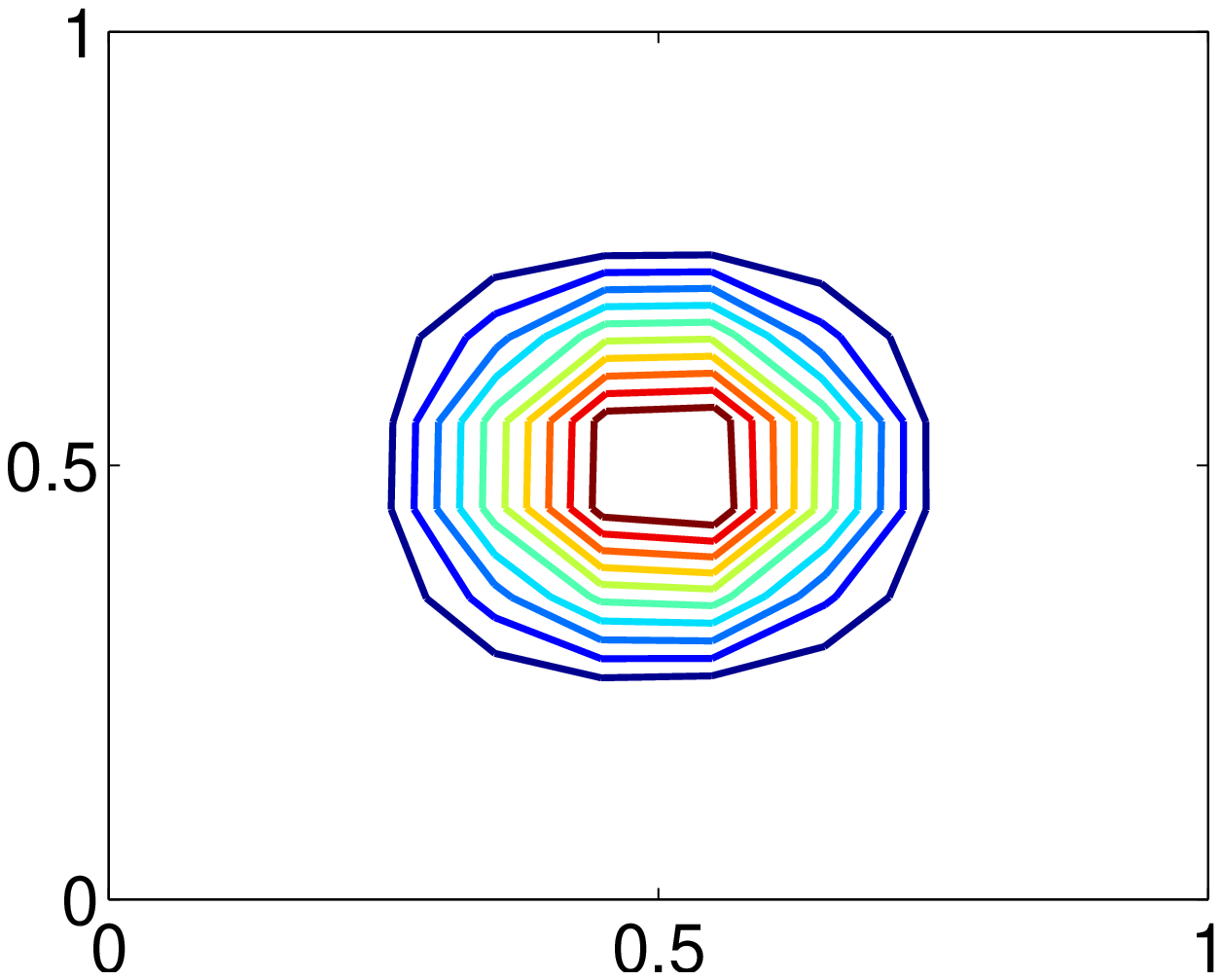}}
\subfloat[ ]{\includegraphics[width=0.32\textwidth]{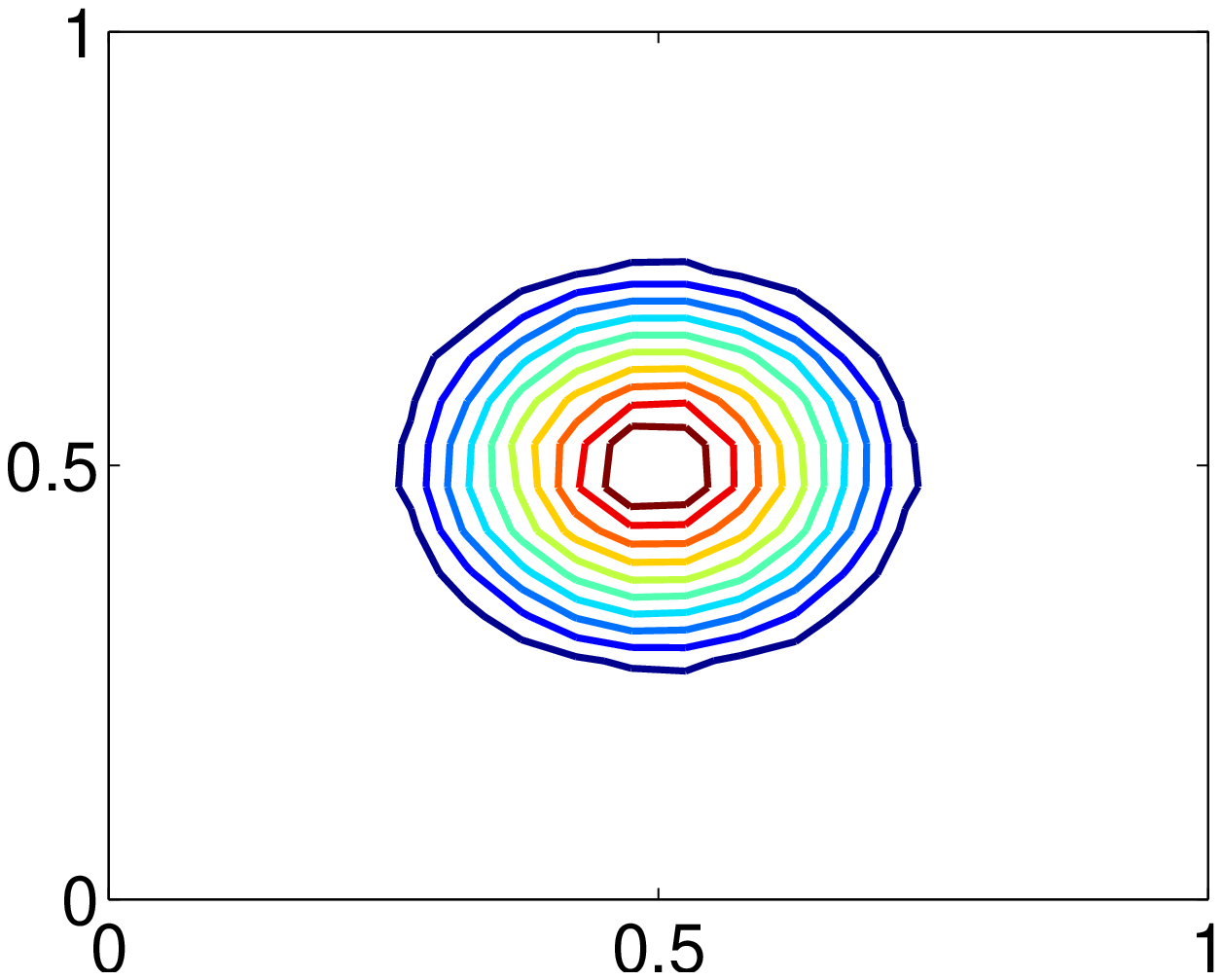}}
\subfloat[ ]{\includegraphics[width=0.32\textwidth]{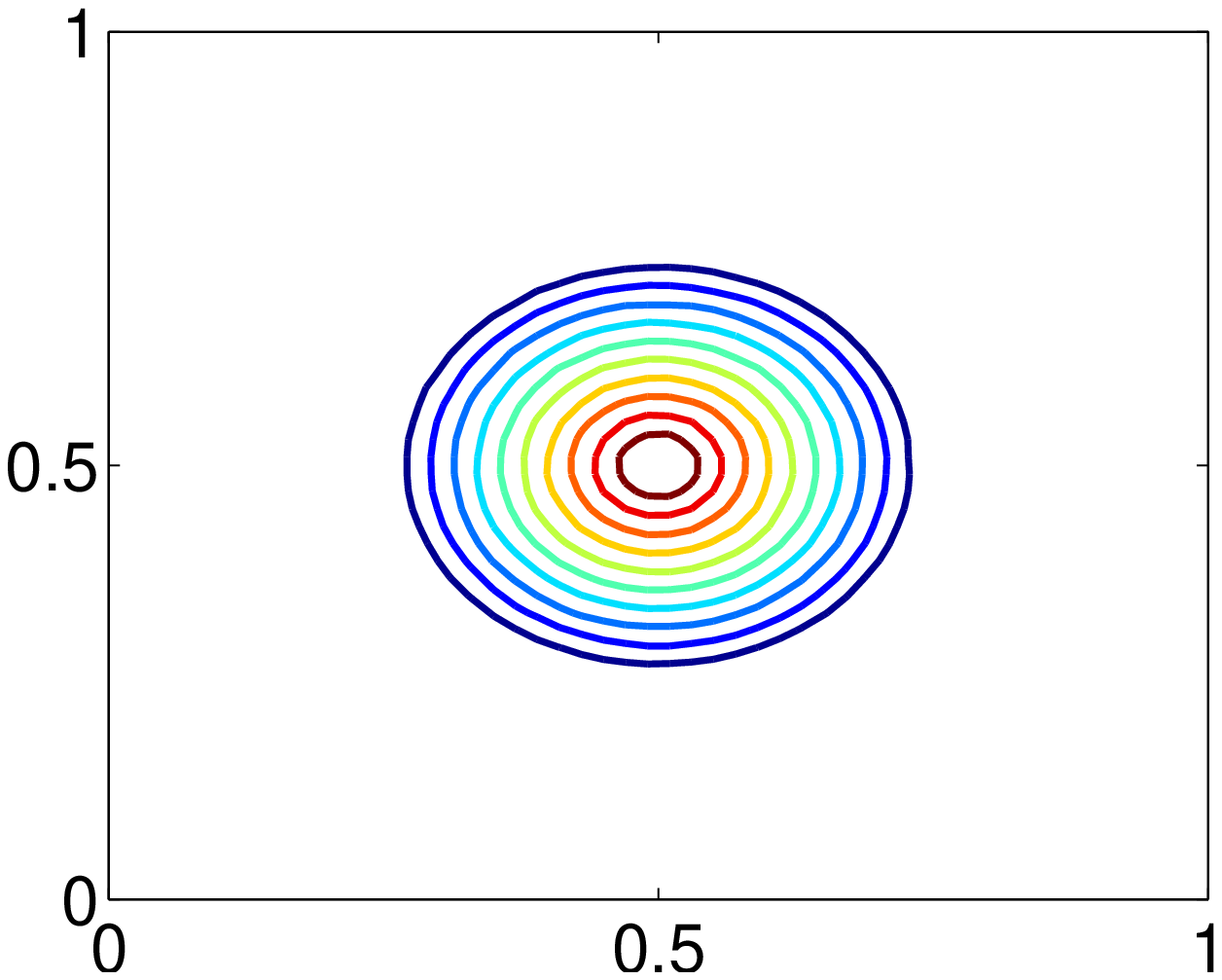}}

\caption{Contour lines of the remapped density functions on $nx \times ny $ random grids. $nx=ny=11$(left), 21(middle) and 51(right).}
\end{figure}

\section{Discussion}
Computing the overlapped region of a Lagrangian mesh(old mesh) and a rezoned mesh(new mesh) and reconstruction (the density or flux) on the Lagrangian mesh are two aspects of a remapping scheme. According to the way how the overlapped (vertical) strips are divided, a remmapping scheme can be either an CIB/DC approach or an FB/DC approach. Both approach are used in practice.  The CIB/DC methods is based on pure geometric or set operation. It is conceptually simple, however Fig.\ref{fig:branch} shows the complexity of computing all the intersections of the two quadrilateral mesh of the same connectivity, it requires 98 programming cases for the non-degenerate intersections to cover all the possible intersection cases which are more than 256.  On contrast, the FB/DC approach is based on the physical law, flux exchanges. Mathematically, this is based on the Green formula and the line integral formula. Fig. \ref{fig:case} shows for the FB/DC method, it only requires 34 programming cases for the non-degenerate intersections. This approach is attractive for the case when the two mesh share the same connectivity. Here we present method to calculate fluxing/swept area or the local swap area. They are calculated exact. The classification on the intersection types can help us to identify the possible degenerate cases, this is convenient when develop a robust remapping procedure.  Based on the Fact 3, we know there are at least 256 possible ways for a new cell to intersect with an old tessellation. But according to Fig.\ref{fig:case}, we can tell there are more cases than 256. What is the exactly possibilities?  This problem remains open as far as we know.


\bibliographystyle{spmpsci}      


\end{document}